\documentclass[a4paper, 12pt, reqno]{amsart}
\usepackage{amsmath,amsthm,amssymb,latexsym,mathrsfs,stmaryrd,color,mathtools}
\usepackage[all]{xy}
\usepackage{url}
\usepackage[a4paper, margin=1.2in]{geometry}
\usepackage{enumerate}

\usepackage[utf8]{inputenc}
\usepackage[T1]{fontenc}

\usepackage{relsize}
\usepackage[bbgreekl]{mathbbol}
\usepackage{amsfonts}

\DeclareSymbolFontAlphabet{\mathbb}{AMSb}
\DeclareSymbolFontAlphabet{\mathbbl}{bbold}

\usepackage{enumitem}
\usepackage[colorlinks=true, hyperindex, linkcolor=magenta, pagebackref=false, citecolor=cyan, pdfpagelabels]{hyperref}
\usepackage[capitalize]{cleveref}
\setlist[enumerate]{itemsep=2pt,parsep=2pt,before={\parskip=2pt}}

\newcommand{\cosimp}[3]{\xymatrix@1{#1 \ar@<.4ex>[r] \ar@<-.4ex>[r] & {\ }#2 \ar@<0.8ex>[r] \ar[r] \ar@<-.8ex>[r] & {\ } #3 \ar@<1.2ex>[r] \ar@<.4ex>[r] \ar@<-.4ex>[r] \ar@<-1.2ex>[r] & \cdots }}

\newcommand{\adjunction}[4]{\xymatrix@1{#1{\ } \ar@<0.3ex>[r]^{ {\scriptstyle #2}} & {\ } #3 \ar@<0.3ex>[l]^{ {\scriptstyle #4}}}}

\numberwithin{equation}{section}

\DeclareMathOperator{\Set}{Set}
\DeclareMathOperator{\et}{\acute{e}t}

\DeclareMathOperator{\Hom}{Hom}
\DeclareMathOperator{\Sch}{Sch}

\DeclareMathOperator{\Spa}{Spa}
\DeclareMathOperator{\Spec}{Spec}

\DeclareMathOperator{\proet}{pro\acute{e}t}
\DeclareMathOperator{\ad}{ad}
\DeclareMathOperator{\cont}{cont}

\DeclareMathOperator{\Fil}{Fil}

\DeclareMathOperator{\GL}{GL}
\DeclareMathOperator{\GSp}{GSp}
\DeclareMathOperator{\rig}{rig}

\DeclareMathOperator{\Rep}{Rep}

\DeclareMathOperator{\Lie}{Lie}

\DeclareMathOperator{\End}{End}

\DeclareMathOperator{\std}{std}

\DeclareMathOperator{\dR}{dR}

\DeclareMathOperator{\Res}{Res}

\DeclareMathOperator{\gr}{gr}
\DeclareMathOperator{\Ext}{Ext}

\DeclareMathOperator{\op}{op}
\DeclareMathOperator{\trace}{tr}
\DeclareMathOperator{\Sh}{Sh}
\DeclareMathOperator{\tor}{tor}
\DeclareMathOperator{\Fl}{Fl}
\DeclareMathOperator{\an}{an}
\DeclareMathOperator{\ket}{k\acute{e}t}
\DeclareMathOperator{\proket}{prok\acute{e}t}
\DeclareMathOperator{\HT}{HT}
\DeclareMathOperator{\la}{la}
\DeclareMathOperator{\sm}{sm}
\DeclareMathOperator{\SL}{SL}

\DeclareMathOperator{\fin}{fin}

\newtheorem{theorem}{Theorem}[section]
\newtheorem*{theorem*}{Theorem}
\newtheorem*{definition*}{Definition}
\newtheorem{proposition}[theorem]{Proposition}
\newtheorem{lemma}[theorem]{Lemma}
\newtheorem{corollary}[theorem]{Corollary}

\theoremstyle{definition}
\newtheorem{definition}[theorem]{Definition}

\newtheorem{remark}[theorem]{Remark}
\newtheorem{condition}[theorem]{Condition}

\crefname{assumption}{assumption}{assumptions}
\crefname{construction}{construction}{constructions}

\title[Locally analytic completed cohomology]{Locally analytic completed cohomology of Shimura varieties of Hodge type}
\author{Kensuke Aoki}
\address{Department of Mathematics, Faculty of Science, Kyoto University
Kyoto, 606-8502, Japan}
\email{aoki.kensuke.88s@st.kyoto-u.ac.jp}

\begin{document}

\begin{abstract} 
For Shimura varieties of Hodge type, we show that there are natural isomorphisms between locally analytic complete cohomology groups and cohomology groups for flag varieties with coefficient which is given by their perfectoid covers. 
This result is a generalization of that of Pan for the modular curve and Qiu-Su for unitary Shimura curves. 
\end{abstract}

\maketitle

\tableofcontents

\section{Introduction}

Let $p$ be a prime number and $\mathbb{C}_p$ be the completion of an algebraic closure of $\mathbb{Q}_p$. 
In this paper, we relate the locally analytic completed cohomology of Shimura varieties of Hodge type to certain cohomology of flag varieties. 
We also study $\tau$-locally analytic completed cohomology of unitary Shimura curves as the special case. 

\subsection{Main results for Shimura varieties of Hodge type}

Let $(G, X)$ be a Shimura datum of Hodge type. 
For compact open subgroups $K_p \subset G(\mathbb{Q}_p)$ and $K^p \subset G(\mathbb{A}^{\infty, p})$, let $\mathcal{S}h_{K_p K^p}$ be the attached Shimura variety and $\mathcal{S}h \coloneqq \mathcal{S}h_{K_p K^p}^{\tor}$ be its toroidal compactification under a fixed cone decomposition, which are seen to be rigid spaces over $\mathbb{C}_p$. 
For a sufficiently small $K^p$, let $\mathcal{S}h_{K^p}^{\tor}$ be the perfectoid Shimura variety constructed in \cite[Proposition 6.1]{Lan22} and $\pi_{\HT}^{\tor} \colon \mathcal{S}h_{K^p}^{\tor} \to \mathscr{F}\ell$ be the Hodge-Tate period map to a flag variety $\mathscr{F}\ell$ (cf.\ \cite[Theorem IV.1.1]{Sch15}).

In this paper, we will study locally analytic subspaces of completed cohomology groups of Shimura varieties of Hodge type 
\begin{align*}
\widetilde{H}^i (K^p, E) \coloneqq \varprojlim_s \varinjlim_{K_p} H_{\et}^i (\mathcal{S}h_{K_p K^p}, \mathcal{O}_E) \otimes_{\mathcal{O}_E} E, \\ 
\widetilde{H}_c^i (K^p, E) \coloneqq \varprojlim_s \varinjlim_{K_p} H_{\et, c}^i (\mathcal{S}h_{K_p K^p}, \mathcal{O}_E) \otimes_{\mathcal{O}_E} E
\end{align*}
where $K_p$ runs over compact open subsets of $G(\mathbb{Q}_p)$.
Then we relate them to certain cohomology groups of flag varieties. 

Let $D \coloneqq \mathcal{S}h_{K_p K^p}^{\tor} \backslash \mathcal{S}h_{K_p K^p}$ be the normal crossing divisor in $\mathcal{S}h_{K_p K^p}^{\tor}$. 
We can write $D = \bigcup_{a \in I} D_{K_p, a}$ for a finite set $I$ such that each $D_{K_p, a}$ is irreducible and $D_{K_p, I'} \coloneqq \cap_{a \in I'} D_{K_p, a}$ is a smooth divisor for each subset $I' \subset I$. 
Let $\iota_{K_p, I'} \colon D_{K_p, I'} \hookrightarrow \mathcal{S}h_{K'_p K^p}^{\tor}$ denote the closed immersion. 
Write 
$\widehat{\mathcal{O}}_{I'}$ for the pro-Kummer \'etale sheaf of completed (bounded) functions on $D_{K_p, I'}$. 
We let 
\[
\widehat{\mathcal{I}}_{\mathcal{S}h} \coloneqq \ker \left( \widehat{\mathcal{O}}_{\mathcal{S}h} \to \bigoplus_{a \in I} \iota_{K_p, a, \ast} \widehat{\mathcal{O}}_a \right). 
\]

Let $\mathcal{S}h_{K^p, \an}^{\tor}$ be the analytic site on the perfectoid Shimura variety $\mathcal{S}h_{K^p}^{\tor}$. 
Let $\mathcal{O}_{K^p} \coloneqq \pi_{\HT, \ast}^{\tor} (\widehat{\mathcal{O}}_{\mathcal{S}h}|_{\mathcal{S}h_{K^p, \an}^{\tor}})$ and $\mathcal{I}_{K^p} \coloneqq \pi_{\HT, \ast}^{\tor} (\widehat{\mathcal{I}}_{\mathcal{S}h}|_{\mathcal{S}h_{K^p, \an}^{\tor}})$ be sheaves of topological algebras on the flag variety $\mathscr{F}\ell$ in analytic topology. 
Let $\mathcal{O}_{K^p}^{\la} \subset \mathcal{O}_{K^p}$ and $\mathcal{I}_{K^p}^{\la} \subset \mathcal{I}_{K^p}$ be subsheaves of locally analytic sections. 

\begin{theorem}[{Theorem \ref{main_theorem_la}}]
\label{intro_main_theorem_la}
For any $i \geq 0$ and a compact open subgroup $K^p \subset G(\mathbb{A}^{\infty, p})$, there are natural $G(\mathbb{Q}_p)$-equivariant isomorphisms 
\begin{align*}
(\widetilde{H}^i (K^p, \mathbb{Q}_p) \widehat{\otimes}_{\mathbb{Q}_p} \mathbb{C}_p)^{\la} &\cong H^i (\mathscr{F}\ell_{\mu}, \mathcal{O}_{K^p}^{\la}), \\ 
(\widetilde{H}_c^i (K^p, \mathbb{Q}_p) \widehat{\otimes}_{\mathbb{Q}_p} \mathbb{C}_p)^{\la} &\cong H^i (\mathscr{F}\ell_{\mu}, \mathcal{I}_{K^p}^{\la})
\end{align*}
\end{theorem}

\begin{remark}
Theorem \ref{intro_main_theorem_la} has been proved in the case of the modular curve in \cite[Theorem 4.4.6]{Pan20} and in the case of unitary Shimura curves in \cite[\S 3.5]{QS25}. 
\end{remark}

\subsection{Main results for unitary Shimura varieties}

Next, suppose that $(G, X)$ is a Shimura datum which gives a unitary Shimura curve defined over a totally real number field $F \neq \mathbb{Q}$. 
Let $\wp$ be a fixed place on $F$ over $p$. 
Let $E/ \mathbb{Q}_p$ be a finite extension such that each set $\Sigma_{\wp_i} \coloneqq \Hom_{\mathbb{Q}_p} (F_{\wp_i}, E)$ where $\wp_i$ is a place of $F$ over $p$ has $[F_{\wp_i} : \mathbb{Q}_p]$ elements. 
Fix a place $\wp = \wp_1$ of $F$ over $p$ and an embedding $\tau \colon F_{\wp} \hookrightarrow E$ in $\Sigma_{\wp}$. 
We can write $G(\mathbb{Q}_p) = \GL_2 (F_{\wp}) \times G(\mathbb{Q}_p^{\wp})$ where $G(\mathbb{Q}_p^{\wp})$ be a $p$-adic Lie group. 
We let 
\[
K = K_{\wp} K_p^{\wp} K^p \subset \GL_2 (F_{\wp}) \times G(\mathbb{Q}_p^{\wp}) \times G(\mathbb{A}^{\infty, p})
\]
be a compact open subgroup. 

Similarly we define completed cohomology of unitary Shimura curves as 
\[
\widetilde{H}^i (K^{\wp}, E) \coloneqq \varprojlim_s \varinjlim_{K_{\wp}} H_{\et}^i (\mathcal{S}h_{K_{\wp} K^{\wp}}, \mathcal{O}_E) \otimes_{\mathcal{O}_E} E, 
\]
on which $\GL_2 (F_{\wp})$ acts. 

We can also consider completed cohomology $\widetilde{H}^i (K^p, E)$ of $K^p$-level. 
Write $\mathcal{S}h_{K^p}$ for the infinite level perfectoid unitary Shimura curve and $\pi_{\HT} \colon \mathcal{S}h_{K^p} \to \mathbb{P}_{\mathbb{C}_p}^1$ be the Hodge-Tate period map. 
Let $\mathcal{O}_{K^p, E} \coloneqq \pi_{\HT, \ast} \mathcal{O}_{\mathcal{S}h_{K^p}, E}$ and $\mathcal{O}_{K^p, E}^{\tau-\la}$ be the $\tau$-locally analytic subsheaf of $\mathcal{O}_{K^p, E}$. 
We can similarly define $\mathcal{O}_{K^{\wp}, E}^{\tau-\la}$. 

\begin{theorem}[{Theorem \ref{main_theorem}}]
\label{into_main_theorem}
There is natural $G(\mathbb{Q}_p)$-equivariant isomorphisms 
\[
(\widetilde{H}^1 (K^{\wp}, E) \widehat{\otimes}_{\mathbb{Q}_p} \mathbb{C}_p)^{\tau-\la} \cong H^1 (\mathscr{F}\ell_{\mu}, \mathcal{O}_{K^p, E}^{\tau-\la})^{K_p^{\wp}} \cong H^1 (\mathscr{F}\ell_{\mu}, \mathcal{O}_{K^{\wp}, E}^{\tau-\la}). 
\]
\end{theorem}

\begin{remark}
We note that Theorem \ref{into_main_theorem} is proved in \cite[\S 3.5]{QS25} in more general settings, but we give another proof of this theorem. 
There are analogous results \cite[Theorem 4.4.6]{Pan20} for the modular curve, \cite[Theorem 6.2.6]{Cam22b} for analytic cohomology of general infinite level Shimura varieties. 
\end{remark}

\subsection{Outline of the proof}
Let $\mu$ be the Hodge cocharacter associated with a fixed element $h \colon \Res_{\mathbb{C}/\mathbb{R}} \mathbb{G}_{m, \mathbb{C}} \to G(\mathbb{R})$ in $X$. 
A parabolic Lie subalgebra $\mathfrak{p}_{\mu} \subset \mathfrak{g} \coloneqq \Lie (G_{\mathbb{Q}_p})$ is naturally attached to $\mu$. 

Our argument uses the embeddings of perfectoid Shimura varieties of Hodge type to perfectoid Siegel modular varieties. 
Since the anticanonical tower of Siegel modular varieties has Tate's normalized traces by \cite[Theorem III.2.36]{Sch15}, we will verify that the following proposition in the first part of Section \ref{sec_la_comp_coh}. 

Let $(\GSp_{2g}(\mathbb{Q}_p), D)$ be a Siegel Shimura datum and fix a closed immersion of Shimura data of Hodge type $(G, X) \hookrightarrow (\GSp_{2g}(\mathbb{Q}_p), D)$. 
Fix a sufficiently small compact open subgroup $K'_p {K^p}' \subset \GSp_{2g} (\mathbb{A}^{\infty})$ such that $K'_p {K^p}' \cap G(\mathbb{A}^{\infty}) = K_p K^p$. 
Let $\mathcal{Y} \coloneqq \mathcal{Y}_{K'_p {K^p}'}^{\tor}$ be a toroidal compactification of the associated Siegel modular variety as a rigid space over $\mathbb{C}_p$. 
For a sufficiently small open subspace $U' \subset \mathcal{Y}$, we can attach an open affinoid perfectoid subspace $\mathcal{S}h_{K^p}^{\tor}(U') \subset \mathcal{S}h_{K^p}^{\tor}$. 

We will compare local sections on perfectoid Siegel modular varieties and on perfectoid Shimura varieties of Hodge type in Section \ref{sec_la_comp_coh}. 
In the section, we use the detailed argument in geometric Sen theory in \cite{Cam22a}. 

We will also show a relation between completed cohomology of Shmura curves of $K^p$-level and those of $K^{\wp}$-level, making use of the proof of \cite[Theorem 6.2.6]{Cam22b} and certain spectral sequences proved in the paper (Corollary \ref{spec_seq_Rlad}). 



\subsection{Organization of the paper}
In Section \ref{s_Shimura_curves}, we recall the notation of Shimura varieties of Hodge type and their special case, unitary Shimura varieties. 
In Section \ref{sec_HT_period_maps}, we recall the construction of the Hodge-Tate period map associated with general Shimura varieties. 
In Section \ref{sec_comparison_theorem}, we slightly generalize the result of Camargo \cite[Theorem 1.1.8]{Cam22b} in the case of $J$-locally analytic cohomology with coefficient of local systems attached to $p$-adic algebraic representations of reductive groups. 
In Section \ref{sec_la_comp_coh}, we will compare the local sections on perfectoid Siegel modular varieties and on perfectoid Shimura varieties of Hodge type. We also prove the locally analyticity of perfectoid covers of general Shimura varieties of Hodge type generalizing the method of Pan \cite[\S 3.5, \S 3.6, \S 4.2]{Pan20}. From them, we will deduce the main theorem. 

\subsection{Notation}


For a ring $A$, write $M_n (A)$ for the ring of $n \times n$-matrix and $A^{\times}$ for the subgroup of invertible elements in $A$. 
For a matrix $M$, let $M^t$ be the transpose of $M$. 
Let $\widehat{\mathbb{Z}} \coloneqq \prod_{\ell : \text{prime}} \mathbb{Z}_{\ell}$ and $\widehat{\mathbb{Z}}^p \coloneqq \prod_{\ell \neq p} \mathbb{Z}_{\ell}$
Let $\mathbb{A}$ be the ring of adeles and $\mathbb{A}^S$ be the ring of adeles outside of $S$ for a finite set $S$ of places on $\mathbb{Q}$.

For a finite extension $L/\mathbb{Q}_p$, let $\mathcal{O}_L$ be the ring of integers in $L$ and $\varpi_L$ be a fixed uniformizer in $\mathcal{O}_L$. 
Let $\mathbb{C}$ be the field of complex numbers and $\mathbb{C}_p$ be the completion of the algebraic closure of $\mathbb{Q}_p$. 
For a global field $F$ and a place $v$ on $F$, let $F_v$ be the completion of $F$ at $v$.
In this case, write $\mathcal{O}_v \coloneqq \mathcal{O}_{F_v}$ and $\varpi_v \coloneqq \varpi_{F_v}$. 

For a field extension $L/K$ and a scheme $X$ over $\Spec K$, let $X_L \coloneqq X \times_{\Spec K} \Spec L$ be the base change. 
Let $\GL_n$, $\SL_n$ be the general linear group and the special linear group respectively. 
They are algebraic groups over $\Spec \mathbb{Z}$. 
Let $\Set$, (resp.\ $(\Sch / S)$, $\Rep_F (G)$) be the category of sets (resp.\ schemes over $S$ for a scheme $S$, finite dimensional algebraic $F$-representations of $G$ for a field $F$ and an algebraic group $G$ over $F$). 
For a representation $W$ of a group $G$, let $W^{\vee}$ be the contragredient of $W$. 


\subsection*{Acknowledgement}
The author thanks his advisor Tetsushi Ito for valuable comments on the draft of the paper.

\section{Shimura varieties of Hogde type and unitary Shimura curves}
\label{s_Shimura_curves}

\subsection{Generalities of Shimura varieties}
\label{ss_Shimura_varieties}

Let $G$ be a reductive group over $\mathbb{Q}$ and $(G, X)$ a Shimura datum (e.g.\ see \cite[Definition 5.5]{Mil05}). 
Let $\mathcal{F}/\mathbb{Q}$ be the reflex field of $(G, X)$. 

A natural projective system $\Sh_K (\mathbb{C})$ of complex manifolds over $\mathbb{C}$ indexed by open compact subgroups $K \subset G(\mathbb{A}^{\infty})$ is associated with the datum $(G, X)$, which is defined by 
\[
\Sh_K (\mathbb{C}) \coloneqq G(\mathbb{Q}) \backslash (X \times (G(\mathbb{A}^{\infty})/K)). 
\]
Let $\Sh_K$ denote the canonical model of $\Sh_K (\mathbb{C})$ over $\Spec \mathcal{F}$. 

Let $Z$ be the center of $G$ and $Z_c \subset Z$ its maximal $\mathbb{Q}$-anisotropic torus which is $\mathbb{R}$-split. 
Let $G^c \coloneqq G / Z_c$. 
For any subgroup $H$ of $G(\mathbb{A})$, let $H^c$ denote the image of $H$ in $G^c$. 

Given neat open compact subgroups $K_1 \subset K_2 \subset G(\mathbb{A}^{\infty})$ where $K_1$ is a normal subgroup of $K_2$, the induced map $\Sh_{K_1} \to \Sh_{K_2}$ of Shimura varieties is Galois finite \'etale with Galois group $K_2^c / K_1^c$. 

For a Shimura variety $\Sh_K$ and an auxiliary $K$-admissible cone decomposition $\Sigma$, let $\Sh_{K, \Sigma}^{\tor}$ denote the toroidal compactification as in \cite{Pin89} (see also \cite{FC90} for an algebraic construction in the Siegel modular varieties). 
When the choice of the cone decomposition will not be important, we write $\Sh_K^{\tor} = \Sh_{K, \Sigma}^{\tor}$. 
A cone decomposition $\Sigma$ can be taken as follows and assume the conditions are always satisfied: 
\begin{enumerate}
\item $\Sigma$ is a smooth cone decomposition relative to $K$ so that the toroidal compactification $\Sh_K^{\tor}$ is smooth, projective with boundary divisor given by normal crossings (\cite[Theorem 9.21]{Pin89}). 
\item The divisor of boundary $D_K \coloneqq \Sh_K^{\tor} \backslash \Sh_K$ is a union $D_K = \cup_{a \in I} D_{K, a}$ of finite irreducible smooth divisors and $\cup_{a \in I'} D_{K, a}$ is also a smooth divisor for any subset $I' \subset I$ (\cite[Proposition 9.20]{Pin89}). 
\end{enumerate}

\subsection{Shimura varieties of Hodge type}
\label{ss_Shimura_Hodge}

We introduce the notation of the \emph{Siegel modular varieties} and \emph{Shimura varieties of Hodge type} (cf.\ \cite{Sch15}, etc.). 
Fix an integer $g \geq 1$. Let $(V, \psi)$ be a symplectic space of dimension $2g$ over $\mathbb{Q}$, that is, $V$ is a vector space of dimension $2g$ over $\mathbb{Q}$ and $\psi \colon V \times V \to \mathbb{Q}$ is a non-degenerate skew-symmetric bilinear form. 
Note that $(V, \psi)$ is (non-canonically) isomorphic to the standard symplectic space $(V_{2g}, \psi_{2g})$, which is, $V_{2g} \coloneqq \mathbb{Q}^{2g}$ and $\psi$ is the standard symplectic pairing defined by 
\[
\psi ((a_1, \ldots, a_g, b_1, \ldots, b_g), (a'_1, \ldots, a'_g, b'_1, \ldots, b'_g)) = \sum_{i = 1}^g (a_i b'_i - a'_i b_i). 
\]

We fix the self-dual lattice $V_{2g, \mathbb{Z}} \coloneqq \mathbb{Z}^{2g}$.  
Let $\GSp_{2g} / \mathbb{Z}$ be the group of symplectic similitudes of $V_{2g, \mathbb{Z}}$. 
For an integer $m \geq 1$, we define the following congruence subgroups of $\GSp_{2g} (\mathbb{Z}_p)$
\begin{align*}
\Gamma_0 (p^m) &\coloneqq \left\{ \gamma \in \GSp_{2g} (\mathbb{Z}_p) \ | \ \gamma \equiv \begin{pmatrix} \ast & \ast \\ 0 & \ast \\ \end{pmatrix} \mathrm{mod} \ p^m, \ \det \gamma \equiv 1 \ \mathrm{mod} \ p^m \right\}, \\ 
\Gamma_1 (p^m) &\coloneqq \left\{ \gamma \in \GSp_{2g} (\mathbb{Z}_p) \ | \ \gamma \equiv \begin{pmatrix} 1 & \ast \\ 0 & 1 \\ \end{pmatrix} \mathrm{mod} \ p^m \right\}, \\ 
\Gamma (p^m) &\coloneqq \left\{ \gamma \in \GSp_{2g} (\mathbb{Z}_p) \ | \ \gamma \equiv \begin{pmatrix} 1 & 0 \\ 0 & 1 \\ \end{pmatrix} \mathrm{mod} \ p^m \right\}, 
\end{align*}
where the all blocks has size $g \times g$.

A Shimura datum $(G, D)$ is said to be \emph{of Hodge type} if it admits a closed embedding $(G, D) \hookrightarrow (\GSp_{2g}, D_{\GSp_{2g}})$ into the group of symplectic similitude group $\GSp_{2g}$, with $D_{\GSp_{2g}}$ given by the Siegel upper half space. 
The associated variety $\Sh_K$ with $(G, D)$ of Hodge type and a compact open subgroup $K \subset G(\mathbb{A}^{\infty})$ is called a \emph{Shimura variety of Hodge type}. 

\subsection{Unitary Shimura curves}
\label{ss_Shimura_curves}

We introduce the notation of the \emph{unitary Shimura curves} handled within the paper (cf.\ \cite{Del71}, \cite{Car86}, \cite{Din17}). 
Let $F$ be a totally real number field of finite degree $d_F \coloneqq [F \colon \mathbb{Q}] > 1$. 
Let $B$ be a quaternion algebra over $F$. 
We assume that there is a unique real place $\tau_{\infty}$ where $B$ splits. 

Let $\mathcal{E}$ be a quadratic imaginary extension of $\mathbb{Q}$ and $\mathcal{F} \coloneqq F \cdot \mathcal{E}$. 
Let $c \colon \mathcal{F} \to \mathcal{F}$ be the conjugation induced by that of $\mathcal{E}$. 
Fix an embedding $u_{\infty} \colon \mathcal{E} \hookrightarrow \mathbb{C}$ and write $u_{\mathbb{C}} \colon \mathcal{F} \to \mathbb{C}$ for the unique embedding such that $u_{\mathbb{C}} |_F = \tau_{\infty}$ and $u_{\mathbb{C}} |_{\mathcal{E}} = u_{\infty}$. 
Let $D \coloneqq B \otimes_F \mathcal{F}$. 
Let $l \mapsto \overline{l}$ denote the tensor product of the canonical involution of $B$ and the conjugation $c \colon \mathcal{F} \to \mathcal{F}$. 
Let $\delta \in D$ be a invertible and symmetric (i.e.\ $\overline{\delta} = \delta$) element and set $l^{\ast} \coloneqq \delta^{-1} \overline{l} \delta$. 
The map $l \mapsto l^{\ast}$ is an involution of second kind on $D$. 

Write $V$ for the underlying $\mathbb{Q}$-vector space of $D$, which is equipped with the action of $D \otimes_{\mathcal{F}} D^{\op}$ defined by $(d_1 \otimes d_2) \cdot v = d_1 v d_2$ for any $d_1, d_2 \in D$ and $v \in V$. 
Choose an element $\alpha \in \mathcal{F}$ that is non-zero and imaginary (i.e.\ $\overline{\alpha} = -\alpha$). 
We define a pairing $\psi \colon V \times V \to \mathcal{F}$ as 
\[
\psi(v, w) \coloneqq \trace_{\mathcal{F}/\mathbb{Q}}(\alpha \trace_{D/\mathcal{F}} (v \delta w^{\ast}))
\]
for $v, w \in V$. It is an alternating and non-degenerate form and the following equation is verified:
\[
\psi (lv, w) = \psi(v, l^{\ast}w). 
\]

For any $\mathbb{Q}$-algebra $R$, regard elements $g \in D^{\op} \otimes_{\mathbb{Q}} R$ as $D$-linear $R$-endomorphisms $v \mapsto vg$ of $V \otimes_{\mathbb{Q}} R$. 
Let $G$ be the algebraic group such that $G(R)$ is the subgroup of the multiplicative group $(D^{\op} \otimes_{\mathbb{Q}} R)^{\times}$ which satisfies $g \in G(R)$ if and only if there exists an element $r \in R^{\times}$ such that 
\[
\psi (vg, wg) = r\psi (v, w)
\]
for any $v, w \in V \otimes_{\mathbb{Q}} R$. 

Let $\mathbb{S} \coloneqq \Res_{\mathbb{C}/\mathbb{R}} (\mathbb{G}_m)$. 
Let $\mu \colon \mathbb{S} \to G_{\mathbb{R}}$ be the homomorphism defined in \cite[\S 2.1.3]{Car86}. 
The map $\mu$ defines the $G(\mathbb{R})$-conjugation class $X$ of homomorphisms $\mathbb{S} \to G_{\mathbb{R}}$, which is identified with the Poincar\'e upper half plane. 
Then the composition of $h$ and the embedding $G_{\mathbb{R}} \hookrightarrow \GL (V_{\mathbb{R}})$ defines a Hodge structure of type $\{ (-1, 0), (0, -1) \}$ (cf.\ \cite[\S 2.2.4]{Car86}). 
In other words, the map induced a closed immersion $(G, X) \hookrightarrow (\GSp (V, \psi), D)$ where $D$ is given by the Siegel upper half space and so $(G, X)$ is a Shimura data of Hodge type. 

A natural projective system $\Sh_K (\mathbb{C})$ of curves over $\mathbb{C}$ indexed by open compact subgroups $K \subset G(\mathbb{A}^{\infty})$ is associated with the data $(G, X)$ as in Section \ref{ss_Shimura_varieties}. 
The curve $M_K (\mathbb{C})$ admits a proper smooth canonical model over the reflex field $\mathcal{F}$ via the embedding $u_{\mathbb{C}} \colon \mathcal{F} \hookrightarrow \mathbb{C}$. 

Fix a place $\wp$ of $F$ over $p$. 
Write $\wp_1 = \wp, \wp_2, \ldots, \wp_s$ for the places of $F$ over $p$. 
Assume that $p$ decomposes in $\mathcal{E}$ and $B$ splits at $\wp_i$ for all place $\wp_i$ of $F$ over $p$. 
Fix isomorphisms $D \otimes_{\mathbb{Q}} F_{\wp_i} \xrightarrow{\sim} M_2 (F_{\wp_i})$ for each place $\wp_i$ of $F$ over $p$ and an embedding $\lambda \colon \mathcal{E} \hookrightarrow \mathbb{Q}_p$, which corresponds to a place of $\mathcal{E}$ over $p$. 
Let $\mathcal{O}_D$ be an order of $D$ which is stable by the involution $l \mapsto l^{\ast}$. 
We have the following decomposition
\begin{align*}
D \otimes_{\mathbb{Q}} \mathbb{Q}_p &\cong (D \otimes_{\mathcal{E}, \lambda} \mathbb{Q}_p) \oplus (D \otimes_{\mathcal{E}, \lambda \circ c} \mathbb{Q}_p) \\ 
&\cong \Big( \prod_{\wp_i | p} M_2 (F_{\wp_i}) \Big) \oplus \Big( \prod_{\wp_i | p} M_2 (F_{\wp_i}) \Big) \eqqcolon \Big( \prod_{\wp_i | p} D_{\wp_i}^+ \Big) \oplus \Big( \prod_{\wp_i | p} D_{\wp_i}^- \Big)
\end{align*}
and similarly we have 
\[
\mathcal{O}_D \otimes_{\mathbb{Z}} \mathbb{Z}_p \cong \prod_{\wp_i | p} (\mathcal{O}_{D_{\wp_i}}^+ \oplus \mathcal{O}_{D_{\wp_i}}^-). 
\]

We assume the following condition on $\lambda$:

\begin{itemize}
\item The ring $\mathcal{O}_{D_{\wp_i}^+}$ (resp.\ $\mathcal{O}_{D_{\wp_i}^-}$) is a maximal order in $D_{\wp_i}^+$ (resp.\ $D_{\wp_i}^-$) and the ring $\mathcal{O}_{D_{\wp_i}^-}$ is identified with $M_2 (\mathcal{O}_{F_{\wp_i}})$ via the composite $\mathcal{O}_{D_{\wp_i}^-} \hookrightarrow D_{\wp_i}^- \cong M_2 (F_{\wp_i})$ for all places $\wp_i | p$. 
\end{itemize}

Let $e_{\wp_i}^{-, 1} \in \mathcal{O}_{D_{\wp_i}^-}$ (resp.\ $e_{\wp_i}^{-, 2} \in \mathcal{O}_{D_{\wp_i}^-}$) be the element which is identified with $\begin{pmatrix} 1 & 0 \\ 0 & 0 \\ \end{pmatrix}$ (resp.\ $\begin{pmatrix} 0 & 0 \\ 0 & 1 \\ \end{pmatrix}$) via the isomorphism $\mathcal{O}_{D_{\wp_i}^-} \xrightarrow{\sim} M_2 (\mathcal{O}_{F_{\wp_i}})$. 
Set $e_{\wp_i}^{+, k} \coloneqq (e_{\wp_i}^{-, k})^{\ast}$ for $k = 1, 2$. 
For any $\mathcal{O}_D \otimes_{\mathbb{Z}} \mathbb{Z}_p$-module $\Lambda$, we have the following decomposition 
\begin{align*}
\Lambda \cong \prod_{\wp_i | p} (\Lambda_{\wp_i}^+ \oplus \Lambda_{\wp_i}^-) \cong &\prod_{\wp_i | p} (\Lambda_{\wp_i}^{+, 1} \oplus \Lambda_{\wp_i}^{+, 2} \oplus \Lambda_{\wp_i}^{-, 1} \oplus \Lambda_{\wp_i}^{-, 2}) \\ 
\cong &\prod_{\substack{\wp_i | p \\ \sigma \in \Sigma_{\wp_i}}} (\Lambda_{\wp_i, \sigma}^{+, 1} \oplus \Lambda_{\wp_i, \sigma}^{+, 2} \oplus \Lambda_{\wp_i, \sigma}^{-, 1} \oplus \Lambda_{\wp_i, \sigma}^{-, 2})
\end{align*}
where $\Lambda_{\wp_i}^{\pm, k} \coloneqq e_{\wp_i}^{\pm, k} \Lambda$ is an $\mathcal{O}_{F_{\wp_i}}$-module. 

For any $\mathbb{Q}$-algebra $R$, we have the isomorphisms
\begin{align*}
G(R) &\xrightarrow{\sim} \{ (r, g) \in R^{\times} \times (D^{\op} \otimes_{\mathbb{Q}} R)^{\times} \ | \ gg^{\ast} = r \} \\ 
&\xrightarrow{\sim} \{ (r, g) \in R^{\times} \times (D \otimes_{\mathbb{Q}} R)^{\times} \ | \ gg^{\ast} = r \}
\end{align*}
functorial in $R$, where the first map is the natural isomorphism by the definition of $G$ and the second map is defined by 
\begin{align*}
R^{\times} \times (D^{\op} \otimes_{\mathbb{Q}} R)^{\times} &\xrightarrow{\sim} R^{\times} \times (D \otimes_{\mathbb{Q}} R)^{\times}, \\ 
(r, g) &\mapsto (r^{-1}, g^{-1}). 
\end{align*}
Thus, we have an isomorphism 
\begin{equation}
\label{curve_G_decomposition}
G_{\mathbb{Q}_p} \xrightarrow{\sim} \mathbb{G}_{m, \mathbb{Q}_p} \times \prod_{\wp_i | p} \Res_{F_{\wp_i} / \mathbb{Q}_p} \GL_{2, F_{\wp_i}}
\end{equation}
of algebraic groups over $\mathbb{Q}_p$ by the above condition on $\lambda$. We also set 
\begin{align*}
&G(\mathbb{Q}_p^{\wp}) \coloneqq \mathbb{Q}_p^{\times} \times \prod_{\wp_i \neq \wp} \GL_2 (F_{\wp_i}) \subset G(\mathbb{Q}_p), \\ 
&G(\mathbb{A}^{\infty, \wp}) \coloneqq G(\mathbb{Q}_p^{\wp}) \times G(\mathbb{A}^{\infty, p}) \subset G(\mathbb{A}^{\infty}). 
\end{align*}
We fix a finite extension $E/\mathbb{Q}_p$ such that $| \Hom_{\mathbb{Q}_p} (F_{\wp_i}, E) | = [F_{\wp_i} : \mathbb{Q}_p]$ for all $i \in [1, s]$ and write $\Sigma_{\wp_i} \coloneqq \Hom_{\mathbb{Q}_p} (F_{\wp_i}, E)$. 

We fix an embedding $\tau \colon F_{\wp} \hookrightarrow E$ in $\Sigma_{\wp}$ and a lattice $V_{\mathbb{Z}} \subset V$ where $\psi$ induces a perfect duality. 
Let $V_{\widehat{\mathbb{Z}}} \coloneqq V_{\mathbb{Z}} \otimes_{\mathbb{Z}} \widehat{\mathbb{Z}}$. 

\begin{theorem}
\label{Shimura_curves_moduli}
Suppose that $K \subset G(\mathbb{A}^{\infty})$ is a neat subgroup. Then the curve $\Sh_{K, \wp} \coloneqq \Sh_K \times_{\mathcal{F}} F_{\wp}$ is represented by a functor $(\Sch/\Spec F_{\wp}) \to \Set$, which associates each scheme $S$ over $\Spec F_{\wp}$ with the set of isomorphism classes of the data $(A, \iota, \theta, \overline{\alpha})$ where 
\begin{enumerate}
\item $A$ is an abelian scheme of relative dimension $4 d_F$ over $S$, with an action of $\mathcal{O}_D$ via $\iota \colon \mathcal{O}_D \hookrightarrow \End_S (A)$ (Note that $\iota$ induces a natural $\mathcal{O}_D \otimes_{\mathbb{Z}} \mathbb{Q}_p$-module structure on $\Lie(A)$) which satisfies the following two conditions:
\begin{enumerate}
\item Here $\Lie(A)_{\wp}^{-, 1}$ is a locally free $\mathcal{O}_S$-module of rank $1$, where $F_{\wp}$ acts via the structural morphism. 
\item $\theta \colon A \to A^{\vee}$ is a polarization of $A$ where the associated Rosati involution sends $\iota(\gamma)$ to $\iota(\gamma^{\ast})$. 
\end{enumerate}
\item Let $\widehat{T}(A) \coloneqq (\varprojlim_n A[n]) \otimes_{\mathbb{Z}} \mathbb{Q}$. Here $\overline{\alpha}$ is a modulo $K$ class of 
\[
\alpha \colon V_{\widehat{\mathbb{Z}}} \xrightarrow{\sim} \widehat{T}(A)
\]
which are $\mathcal{O}_D$-linear symplectic isomorphisms local on the \'etale topology on $S$. 
\end{enumerate}
\end{theorem}

\begin{proof}
See \cite[Th\'eor\`eme 4.19]{Del71}, \cite[\S 2.3]{Car86}, and also \cite[\S 3.1.1]{Din17}. 
\end{proof}

\section{The Hodge-Tate period map}
\label{sec_HT_period_maps}
We fix an isomorphism $\mathbb{C} \cong \mathbb{C}_p$. In this section, we recall the construction of the Hodge-Tate period maps of general Shimura varieties using (logarithmic) $p$-adic Riemann-Hilbert correspondence. 

\subsection{Log adic spaces}
\label{ss_log_adic_spaces}

We recall the notation of Kummer-\'etale sites on log adic spaces (\cite{DLLZ18}, \cite{DLLZ19}). 

Let $X$ be a log smooth adic space (\cite[Definition 3.1.1]{DLLZ19}) over $\Spa (k, k^+)$, where $(k, k^+)$ is a non-archimedean extension of $(\mathbb{Q}_p, \mathbb{Z}_p)$. 
We endow $X$ with the log structure induced by a normal crossing divisor $D \subset X$ (\cite[Example 2.3.17]{DLLZ19}). 
Let $X_{\an}$, $X_{\ket}$, $X_{\proket}$ denote the analytic, Kummer-\'etale and pro-Kummer \'etale sites of $X$ respectively (\cite[Definition 4.1.16, Definition 5.1.2]{DLLZ19}). 
If $D$ is empty, $X_{\ket}$ and $X_{\proket}$ are the same as the (pro-)\'etale sites $X_{\et}$ and $X_{\proet}$ of $X$ as in \cite{Sch13} respectively. 
Log perfectoid affinoids (\cite[Definition 5.3.1]{DLLZ19}) forms a basis of $X_{\proket}$ by \cite[Proposition 5.3.12]{DLLZ19}. 
We also write $\mathcal{O}_X$ 
for the inverse images to $X_{\proket}$ of the structure sheaf. 

For a topological abelian group $V$, let $\underline{V}$ denote the pro-Kummer \'etale sheaf defined as 
\[
\underline{V} ((\Spa (R, R^+), \mathcal{M})) \coloneqq \Hom_{\cont} (| \Spa (R, R^+) |, V)
\]
for each log affinoid perfectoid object $(\Spa (R, R^+), \mathcal{M})$ in $X_{\proket}$ (\cite[Definition 5.3.1]{DLLZ19}), where $| \Spa (R, R^+) |$ denotes the underlying topological space of $\Spa (R, R^+)$. 
When $V = \mathbb{Z}_p$ or $\mathbb{Q}_p$, write $\widehat{\mathbb{Z}}_p \coloneqq \underline{\mathbb{Z}}_p$ and $\widehat{\mathbb{Q}}_p \coloneqq \underline{\mathbb{Q}}_p$. 

Let $\widehat{\mathcal{O}}_X^+$, $\widehat{\mathcal{O}}_X$, $\mathbb{B}_{\dR, X}^+$, $\mathbb{B}_{\dR, X}$, $\mathcal{O}\mathbb{B}_{\dR, \log, X}^+$, $\mathcal{O}\mathbb{B}_{\dR, \log, X}$, etc.\ be the period sheaves on $X_{\proet}$ or $X_{\proket}$ as in \cite[\S 2.2]{DLLZ18}. 

\subsection{The Hodge-Tate period map for Shimura varieties of Hodge type}
\label{ss_HT_period_maps}

In this subsection, we use the notation in Section \ref{ss_Shimura_varieties}. 
Fix a Shimura datum $(G, X)$ of Hodge type. 
Let $\mu$ be the $G(\mathbb{C})$-conjugacy class of Hodge cocharacters (defined over $\mathcal{F}$) attached to $X$, which gives rise to associated subgroups $P_{\mu}$, $N_{\mu}$, $M_{\mu}$ of $G_{\mathcal{F}}$ and a flag variety $\Fl_{\mu} = G / P_{\mu}$ over $\Spec \mathcal{F}$ (cf.\ \cite[\S 2.1]{CS15}, \cite[\S 3.1]{Cam22b}). 
Let $\mathscr{F}\ell_{\mu}$ denote the adic space attached to $\Fl_{\mu} \times_{\Spec \mathcal{F}} \Spec \mathbb{C}_p$. 



We recall the notation of $G_{\mathbb{Q}_p}$-equivariant Lie algebroids (cf.\ \cite[\S 3]{BB83}, \cite[\S 4.2.6]{Pan20}, \cite[\S 3.2]{Cam22b}). 
Let $\mathfrak{g}$, $\mathfrak{g}^c$, $\mathfrak{p}_{\mu}$, $\mathfrak{n}_{\mu}$ and $\mathfrak{m}_{\mu}$ be the Lie algebras of $G_{\mathbb{Q}_p}$, $G_{\mathbb{Q}_p}^c$, $P_{\mu, {\mathbb{Q}_p}}$, $N_{\mu, {\mathbb{Q}_p}}$ and $M_{\mu, {\mathbb{Q}_p}}$ respectively. 
Let $\mathfrak{g}^0 \coloneqq \mathcal{O}_{\mathscr{F}\ell_{\mu}} \otimes_{\mathbb{Q}_p} \mathfrak{g}$. 
We also have a natural filtration $\mathfrak{n}_{\mu}^0 \subset \mathfrak{p}_{\mu}^0 \subset \mathfrak{m}_{\mu}^0$ depending on $\mu$ and set $\mathfrak{m}_{\mu}^0 \coloneqq \mathfrak{p}_{\mu}^0 / \mathfrak{n}_{\mu}^0$. 


Let $\mathcal{S}h_K = \mathcal{S}h_{K, \mathbb{C}_p}$ denote the adic space over $\Spa(\mathbb{C}_p, \mathcal{O}_{\mathbb{C}_p})$ attached to $\Sh_{K, \mathbb{C}_p} \coloneqq \Sh_K \times_{\Spec \mathcal{F}} \Spec \mathbb{C}_p$, (cf.\ \cite{Hub96}). 
We also let $\mathcal{S}h_K^{\tor} = \mathcal{S}h_{K, \mathbb{C}_p}^{\tor}$ denote the adic space attached to the toroidal compactification. 
It is seen as a log adic space with log structure defined by the normal crossing divisor of the boundary. 

Suppose that $K = K_p K^p$ for some open compact subgroups $K_p \subset G(\mathbb{Q}_p)$ and $K^p \subset G(\mathbb{A}^{\infty, p})$. 
Let $\varprojlim_{K_p} \mathcal{S}h_{K_p K^p}$, where $K_p \subset G(\mathbb{Q}_p)$ lies on open compact subgroups, be the infinite level Shimura variety considered as an object in $\mathcal{S}k_{K_p K^p, \proet}$. 
Assume $K$ is neat. 
Then we can take the same smooth cone decomposition $\Sigma$ relative to $K = K'_p K^p$ for all sufficiently small $K'_p \subset K_p$ (\cite[Corollary 9.33]{Pin89}) which is also compatible with an closed embedding $\iota \colon (G, X) \hookrightarrow (\GSp_{2g}, D_{\GSp_{2g}})$ of the Shimura data (\cite[Proposition 4.10]{Lan22}). 
Assume that such $\Sigma$ is taken. 
Then consider also $\varprojlim_{K_p} \mathcal{S}h_{K_p K^p}^{\tor}$ to be an object in $\mathcal{S}h_{K_p K^p, \proket}^{\tor}$. 

\begin{theorem}
\label{perfd_Hodge_type}
For any neat compact open subgroup $K^p \subset G(\mathbb{A}^{\infty, p}) \subset G(\mathbb{A}^{\infty})$ contained in the level-$N$-subgroup 
\[
\{ \gamma \in \GSp_{2g} (\widehat{\mathbb{Z}}^p) \ | \ \gamma \equiv 1 \mod N \}
\]
of $\GSp_{2g} (\widehat{\mathbb{Z}}^p)$ for some $N \geq 3$ prime to $p$, there exist perfectoid spaces $\mathcal{S}h_{K^p}$ and $\mathcal{S}h_{K^p}^{\tor}$ such that 
\[
\mathcal{S}h_{K^p} \sim \varprojlim_{K_p} \mathcal{S}h_{K_p K^p} \quad \text{and} \quad \mathcal{S}h_{K^p}^{\tor} \sim \varprojlim_{K_p} \mathcal{S}h_{K_p K^p}^{\tor}, 
\]
where ``$\sim$'' is referred to \cite[Definition 2.4.1]{SW13}. 
Moreover, the closed embedding of Shimura data $\iota$ induces natural closed immersions of perfectoid spaces 
\[
\mathcal{S}h_{K^p} \hookrightarrow \mathcal{Y}_{{K'}^p}, \quad \text{and} \quad \mathcal{S}h_{K^p}^{\tor} \hookrightarrow \mathcal{Y}_{{K'}^p}^{\tor}
\]
where ${K'}^p \subset \GSp_{2g} (\mathbb{A}^{\infty})$ is a neat compact open subgroup such that ${K'}^p \cap G(\mathbb{A}^{\infty}) = K^p$ and $\mathcal{Y}_{{K'}^p}$, $\mathcal{Y}_{{K}^p}^{\tor}$ be the perfectoid Siegel modular variety of level ${K'}^p$ and its toroidal compactification. 
\end{theorem}

\begin{proof}
The proof is similar to that of \cite[Theorem IV.1.1]{Sch15} and \cite[Proposition 6.1]{Lan22}. 
\end{proof}



Write $\mathcal{S}h \coloneqq \mathcal{S}h_{K_p K^p}^{\tor}$. 
Let $V$ be an object in $\Rep_{\mathbb{Q}_p} (G^c)$. 
The object $V$ naturally induces an \'etale local system $V_{\et}$ (resp.\ Kummer \'etale local system $V_{\ket}$ on $\mathcal{S}h$ (resp.\ $\mathcal{S}h^{\tor}$) (see \cite[\S 4.2]{LZ16} and \cite[\S 5.2]{DLLZ18}). 
Let $V_{\proket}$ denote the associated pro-Kummer \'etale local system by \cite[Lemma 6.3.3]{DLLZ19}. 
It is de Rham in the sense of \cite[Definition 2.2.1]{Cam22b} by \cite[Theorem 5.3.1]{DLLZ18}. 
We note that the pullback of $V_{\proket}$ by the natural map of sites $\mathcal{S}h_{K^p, \proket} \to \mathcal{S}h_{K_p K^p, \proket}$ is the constant pro-Kummer \'etale sheaf $\underline{V}$ defined in Section \ref{ss_log_adic_spaces}. 

\begin{proposition}
\label{HT_period_maps}
\begin{enumerate}
\item For $K^p \subset G(\mathbb{A}^{\infty, p})$ which satisfies the condition in Theorem \ref{perfd_Hodge_type}, 
there exists a natural $G(\mathbb{Q}_p)$-equivariant map of adic spaces 
\[
\pi_{\HT}^{\tor} \colon \mathcal{S}h_{K^p}^{\tor} \to \mathscr{F}\ell_{\mu}, 
\]
which is functorial in Shimura data. 
\item There exists a basis $\mathfrak{B}$ of open affinoid subsets of $\mathscr{F}\ell$ stable under finite intersections and each $V \in \mathfrak{B}$ has the following properties: 
\begin{enumerate}
\item Its preimage $U_{\infty} \coloneqq \pi_{\HT}^{\tor, -1} (V)$ is affinoid perfectoid. 
\item $U_{\infty}$ is the preimage of an affinoid subset $U_{K'_p} \subset \mathcal{S}h_{K'_p K^p}^{\tor}$ for sufficiently small compact open subgroup $K'_p \subset G(\mathbb{A}^{\infty})$. 
\item The map $\varinjlim_{K'_p} H^0 (U_{K'_p}, \mathcal{O}_{\mathcal{S}h_{K'_p K^p}^{\tor}}) \to H^0 (U_{\infty}, \mathcal{O}_{\mathcal{S}h_{K^p}^{\tor}})$ has dense image. 
\end{enumerate}
\end{enumerate}
\end{proposition}

\begin{proof}

The proof of (1) is similar to \cite[Theorem 4.40]{BP21} and \cite[Theorem 4.2.1]{Cam22b}. We explain the detailed argument. 
For each object $U = \Spa (R, R^+)$ in $\mathcal{U}_{\infty}$ and an object $V$ in $\Rep_{\mathbb{Q}_p} (G^c)$, we have a natural identification 
\[
H^0 (U, V_{\proket} \otimes_{\mathbb{Q}_p} \widehat{\mathcal{O}}_{\mathcal{S}h}) = V \otimes_{\mathbb{Q}_p} R, 
\]
where the tensor product of the left hand side is taken in the category of pro-Kummer \'etale presheaves of rings over the presheaf associated with $\mathbb{Q}_p$. 

On the other hand, $V_{\proket}$ is de Rham, i.e., there is a filtered log-connection $(V_{\dR, \log}, \nabla_{\log}, \Fil^{\bullet})$ and an isomorphism of pro-Kummer \'etale sheaves 
\[
V_{\proket} \otimes_{\mathbb{Q}_p} \mathcal{O}\mathbb{B}_{\dR, \log, \mathcal{S}h} \cong V_{\dR, \log} \otimes_{\mathcal{O}_{\mathcal{S}h}} \mathcal{O}\mathbb{B}_{\dR, \log, \mathcal{S}h}
\]
on $\mathcal{S}h$ by \cite[Theorem 5.3.1]{DLLZ18} compatible with connections and filtrations. 

Then the Hodge-Tate filtration defined by 
\[
\Fil_{-j} (V_{\proket} \otimes_{\mathbb{Q}_p} \widehat{\mathcal{O}}_{\mathcal{S}h}) \coloneqq (\mathbb{M} \cap \Fil^j \mathbb{M}^0) / (\Fil^1 \mathbb{M} \cap \Fil^j \mathbb{M}^0)
\]
of $V_{\proket} \otimes_{\mathbb{Q}_p} \widehat{\mathcal{O}}_{\mathcal{S}h}$, where 
\[
\mathbb{M} \coloneqq V_{\proket} \otimes_{\mathbb{Z}_p} \mathbb{B}_{\dR, \mathcal{S}h}^+ \quad \text{and} \quad \mathbb{M}^0 \coloneqq (V_{\dR, \log} \otimes_{\mathcal{O}_{\mathcal{S}h}} \mathcal{O}\mathbb{B}_{\dR, \log, \mathcal{S}h}^+)^{\nabla_{\log} = 0}
\]
is compatible with the Hodge filtration $\Fil^{\bullet}$ on $V_{\dR, \log}$ by \cite[Lemma 2.2.3]{Cam22b} (see also \cite[\S 4.4.38]{BP21}). 
So we obtain an increasing $\mu$-filtration on $V \otimes_{\mathbb{Q}_p} R$ and this induces the map $\pi_{\HT}^{\tor}$ in this proposition. 

The all construction in the proof is $G(\mathbb{Q}_p)$-equivariant by the moduli interpretation of Siegel modular varieties (cf.\ \cite[\S 1]{FC90}, \cite[\S III.3]{Sch15}). 
The functoriality of $\pi_{\HT}^{\tor}$ in Shimura data follows by the construction and Tannakian formalism. 
Then we have proved (1). 

For (2), see \cite[Theorem IV.1.1]{Sch15}. 
\end{proof}

\begin{remark}
By considering pro-Kummer \'etale sites on infinite level Shimura varieties instead of perfectoid Shimura varieties of Hodge type in Theorem \ref{perfd_Hodge_type}, it is possible to construct Hodge-Tate period maps for general Shimura varieties (\cite[Theorem 4.2.1]{Cam22b}).
\end{remark}

We call the map $\pi_{\HT}^{\tor}$ the \emph{Hodge-Tate period map}. 
We write $\pi_{\HT} \colon \mathcal{S}h_{K^p} \to \mathscr{F}\ell_{\mu}$ for the restriction of $\pi_{\HT}^{\tor}$ to $\mathcal{S}h_{K^p}$. 
Note that the construction produces the same Hodge-Tate period map as the previously constructed one for Shimura varieties of Hodge type using minimal compactifications instead of toroidal compactifications (\cite[Theorem IV.1.1]{Sch15}, \cite[Theorem 2.1.3]{CS15}). 

\subsection{Flag varieties associated with unitary Shimura curves}
\label{ss_flag_varieties}
At the end of this section, we make an observation about flag varieties associated with unitary Shimura curves. 
Let $(G, X)$ be as in Section \ref{ss_Shimura_curves} and $\Sh_K$ is the associated unitary Shimura curve of level $K = K_p K^p \subset G(\mathbb{Q}_p) \times G(\mathbb{A}^{\infty, p})$. 
Fix an embedding $\tau \colon F_{\wp} \hookrightarrow E$. 
We let 
\[
\varepsilon \colon \mathcal{A} \to (\Sh_K)_{\tau, E} \coloneqq \Sh_K \times_{\Spec \mathcal{F}} \Spec F_{\wp} \times_{\Spec F_{\wp}, \tau} \Spec E
\]
be the universal abelian variety for the moduli interpretation of the unitary Shimura curve in Theorem \ref{Shimura_curves_moduli}. 
Then $R^1 \varepsilon_{\ast} E$ defines a rank two $E$-local system on $\Sh_{K, \et}$, which is isomorphic to the $E$-local system $V_{\std, \et}$ corresponds to the standard representation $\GL_2 (F_{\wp}) \to \GL(V_{\std})$ over $E$ via the embedding $\tau$. 
The local system $R^1 \varepsilon_{\ast} E$ is de Rham and the associated filtered module with integrable connection is given by $R^1 \varepsilon_{\ast} (\Omega_{\mathcal{A} / \Sh_K}^{\bullet} \otimes_{\mathbb{Q}_p} E)$ with the Hodge filtration and the Gauss-Manin connection by \cite[Theorem 1.10]{Sch13}. 

Consider the exact sequence 
\begin{equation}
\label{Hodge_filt_Sh_curves}
0 \to \varepsilon_{\ast} \Omega_{\mathcal{A} / (\Sh_K)_{\tau, E}}^1 \to R^1 \varepsilon_{\ast} \Omega_{\mathcal{A} / (\Sh_K)_{\tau, E}}^{\bullet} \to R^1 \varepsilon_{\ast} \mathcal{O}_{\mathcal{A}} \to 0, 
\end{equation}
which defines the Hodge-Tate filtration of the rank two vector bundle $R^1 \varepsilon_{\ast} \Omega_{\mathcal{A} / (\Sh_K)_{\tau, E}}^{\bullet}$ over $(\Sh_K)_{\tau, E}$. 

\begin{proposition}
\label{Sh_curve_HT_filt_-1}
Let $\wp_i | p$ and $\sigma \in \Sigma_{\wp_i}$. 
\begin{enumerate}
\item Unless $\wp_i = \wp$ and $\sigma = \tau$, the $\mathcal{O}_{(\Sh_K)_{\tau, E}}$-module $(R^1 \varepsilon_{\ast} \Omega_{\mathcal{A}/ (\Sh_K)_{\tau, E}}^{\bullet})_{\wp_i, \sigma}^{-, 1}$ is locally free of rank two and naturally isomorphic to the $\mathcal{O}_{(\Sh_K)_{\tau, E}}$-module $(R^1 \varepsilon_{\ast} \mathcal{O}_{\mathcal{A}})_{\wp_i, \sigma}^{-, 1}$. 
\item If $\wp_i = \wp$ and $\sigma = \tau$, the induced maps 
\[
0 \to (\varepsilon_{\ast} \Omega_{\mathcal{A} / (\Sh_K)_{\tau, E}}^1)_{\wp, \tau}^{-, 1} \to (R^1 \varepsilon_{\ast} \Omega_{\mathcal{A} / (\Sh_K)_{\tau, E}}^{\bullet})_{\wp, \tau}^{-, 1} \to (R^1 \varepsilon_{\ast} \mathcal{O}_{\mathcal{A}})_{\wp, \tau}^{-, 1} \to 0
\]
is exact sequence of locally free $\mathcal{O}_{(\Sh_K)_{\tau, E}}$-modules. The rank of second and fourth term is one and the third term is rank two. 
\end{enumerate}
\end{proposition}

\begin{proof}
See \cite[Proposition 3.3.2]{Din17}. 
\end{proof}

Recall that a closed immersion $(G, X) \hookrightarrow (\GSp(V, \psi), D)$ to a Siegel data is given by the element $h \colon \mathbb{S} \to G_{\mathbb{R}}$ in $X$. 
Then it follows the exact sequence \eqref{Hodge_filt_Sh_curves} naturally comes from the exact sequences from de Rham cohomology of the universal abelian variety of the Siegel modular variety attached with $(\GSp(V, \psi), D)$. 

\begin{proposition}
The flag variety $\Fl_{G, E}$ over $E$ associated with the unitary Shimura curve is isomorphic to $\GL_{2, E} / B_E \cong \mathbb{P}_E^1$, where $B_E$ is the standard Borel subgroup of $\GL_{2, E}$. 
\end{proposition}

\begin{proof}
Since the completed colimit along $K_p \to 1$ gives the Hodge-Tate filtration defining the Hodge-Tate period map of the infinite level perfectoid Siegel modular variety, it follows that the subgroup 
\[
\mathbb{G}_{m, E} \times_{\Spec E} (\prod_{\wp_i \neq \wp} \Res_{F_{\wp_i} / \mathbb{Q}_p} \GL_2 (F_{\wp_i}))_E \times_{\Spec E} \prod_{\sigma \neq \tau} \GL_{2, E} \subset G_E
\]
stabilizes the Hodge-Tate filtration by Proposition \ref{Sh_curve_HT_filt_-1} (1). 
On the other hand, the $E$-vector space $V_{\wp, \tau}^{-, 1}$ is identified with the contragredient of the standard representation of $\GL_2 (E)$. 
Let $B \subset \GL_2 (E)$ be the Borel subgroup which stabilizes the standard representation. 
So a subgroup of $B$ stabilizes the Hodge-Tate filtration by Proposition \ref{Sh_curve_HT_filt_-1} (2). 
The closed immersion of the Shimura data induces a closed immersion $G_E / P \hookrightarrow \GSp(V, \psi)_E / \widetilde{P}$, where $P \subset G_E$ is a parabolic subgroup and $\widetilde{P} \subset \GSp(V, \psi)_{\mathbb{Q}_p}$ is the Siegel parabolic subgroup. 
In particular, $G_E / P$ is proper over $E$, so just the Borel subgroup $B \subset \GL_2 (E)$ is the stabilizer of the whole Hodge-Tate filtration. 
\end{proof}

The exact sequence obtained by tensoring \eqref{Hodge_filt_Sh_curves} with $\widehat{\mathcal{O}}_{(\Sh_K)_{\tau, E}}$ produces a Hodge-Tate period map $\pi_{\HT, G} \colon \mathcal{S}h_{K^p} \to \mathscr{F}\ell_{\mu}$, which is compatible with the Hodge-Tate period map $\pi_{\HT, \GSp(V, \psi)}^{\tor}$ of the Siegel modular variety via the closed immersion $(G, X) \hookrightarrow (\GSp(V, \psi), D)$ of the Shimura data by Proposition \ref{HT_period_maps}. 

\section{$J$-locally analytic completed cohomology of unitary Shimura curves}
\label{sec_comparison_theorem}

We recall the notion of $J$-locally analytic actions of $p$-adic Lie groups and slightly generalize the results in \cite{Pan20} and \cite{Cam22b} to the case of $J$-locally analytic vectors in the completed cohomology of unitary Shimura curves. 

\subsection{$J$-Locally analytic vectors}
\label{J-la_vectors}

Let $\Sigma_{\wp} \coloneqq \Hom_{\mathbb{Q}_p} (F_{\wp}, E)$ and $J \subset \Sigma_{\wp}$. 
Let $\mathcal{G}$ be a $p$-adic Lie group of dimension $d$. 
Then there exists a compact open subgroup $\mathcal{G}_0 \subset \mathcal{G}$ which is a uniform pro-$p$ group (cf.\ e.g., \cite[\S 4.1.1]{JC21}). 
Consequently, there exist $g_1, \ldots, g_d \in \mathcal{G}_0$ such that the map 
\[
c \colon \mathbb{Z}_p^d \to \mathcal{G}_0, \qquad (x_1, \ldots, x_d) \mapsto g_1^{x_1} \cdots g_d^{x_d}
\]
is a homeomorphism. 
Fix such $c$. 
Let $\mathcal{G}_n \coloneqq \mathcal{G}_0^{p^n} = \{ g^{p^n} \ | \ g \in \mathcal{G}_0 \}$ for $n \geq 1$. 
It follows that $\{ \mathcal{G}_n \}_n$ forms a fundamental system of open compact subgroups of the identity element in $\mathcal{G}$.

Let $F_{\wp}$ be the $p$-adic field as in Section \ref{ss_Shimura_curves} and $\widetilde{\mathcal{G}}$ an algebraic group over $\mathcal{O}_{\wp} = \mathcal{O}_{F_{\wp}}$. 
From here, let $\mathcal{G} \coloneqq \Res_{\mathcal{O}_{\wp} / \mathbb{Z}_p} \widetilde{\mathcal{G}}$ and we also see $\mathcal{G}$ as a $p$-adic Lie group.
Let $\widetilde{\mathfrak{g}}$ be the Lie algebra of $\widetilde{\mathcal{G}}$ over $F_{\wp}$, which can be also naturally seen to be the Lie algebra of $\mathcal{G}_0$ over $\mathbb{Q}_p$. 
We have a decomposition 
\[
\mathcal{G} \times_{\Spec \mathbb{Z}_p} \Spec \mathcal{O}_E \xrightarrow{\sim} \bigsqcup_{\sigma \in \Sigma_{\wp}} \mathcal{G}_{\sigma}, \quad \text{and} \quad \mathfrak{g} \otimes_{\mathbb{Q}_p} E \xrightarrow{\sim} \prod_{\sigma \in \Sigma_{\wp}} \mathfrak{g}_{\sigma}, 
\]
where $\mathcal{G}_{\sigma} \coloneqq \widetilde{\mathcal{G}} \times_{\Spec \mathcal{O}_{\wp}} \Spec \mathcal{O}_E$ and $\mathfrak{g}_{\sigma} \coloneqq \widetilde{\mathfrak{g}} \otimes_{F_{\wp}} E$. 
Write $\mathcal{G}_J \coloneqq \prod_{\sigma \in J} \mathcal{G}_{\sigma}$, $\mathfrak{g}_J \coloneqq \prod_{\sigma \in J} \mathfrak{g}_{\sigma}$, and $\mathfrak{g}^J \coloneqq \prod_{\sigma \notin J} \mathfrak{g}_{\sigma}$. 

Let $B$ be a Banach space over $E$. 
Let $C^{\an} (\mathcal{G}_n, B)$ denote the $E$-vector space of \emph{($\mathbb{Q}_p$-)analytic} functions from $\mathcal{G}_n$ to $B$ (cf.\ \cite[\S 2]{ST02}, in which $B$-valued locally ($\mathbb{Q}_p$-)analytic functions are defined). 
We equip the space $C^{\an} (\mathcal{G}_n, B)$ with a norm $|| - ||_{\mathcal{G}_n}$ defined by $||f||_{\mathcal{G}_n} \coloneqq \sup_{\mathbf{k} \in \mathbb{N}^d} || p^{n |\mathbf{k}|} b_{\mathbf{k}} ||$, where $b_{\mathbf{k}}$ are coefficient when the element $f \in C^{\an} (\mathcal{G}_n, B)$ is expanded by the coordinates $x_1, \ldots, x_d$ as 
\[
(x_1, \ldots, x_d) \mapsto \sum_{\mathbf{k} = (k_1, \ldots, k_d) \in \mathbb{N}^d} b_{\mathbf{k}} x_1^{k_1} \cdots x_d^{k_d}
\]
for some $b_{\mathbf{k}} \in B$ such that $p^{n | \mathbf{k} |} b_{\mathbf{k}} \to 0$ as $| \mathbf{k} | \to 0$. 

Let $I$, $J$ be a subset of $\Sigma_{\wp}$ such that $I \cap J = \emptyset$. 
A function $f \in C^{\an} (\mathcal{G}_n, B)$ is said to be \emph{$\mathcal{G}_{n, I}$-smooth $\mathcal{G}_{n, J}$-analytic} if it is invariant under the action of $\mathcal{G}_{n, I}$ and the action of $\mathfrak{g} \otimes_{\mathbb{Q}_p} E$ factors through $\mathfrak{g}_J$ (cf.\ \cite[D\'efinition, 2.1]{Sch10}). 
If $I = \emptyset$, we say \emph{$\mathcal{G}_{n, J}$-analytic} instead of $\mathcal{G}_{n, \emptyset}$-smooth $\mathcal{G}_{n, J}$-analytic. 

Let $C^{\mathcal{G}_{n, J}-\an} (\mathcal{G}_n, B)$ (resp.\ $C^{\mathcal{G}_{n, I} -\sm, \mathcal{G}_{n, J}-\an} (\mathcal{G}_n, B)$) denote the $E$-vector subspace of $C^{\an} (\mathcal{G}_n, B)$ generated by $\mathcal{G}_{n, J}$-analytic functions (resp.\ $\mathcal{G}_{n, I}$-smooth $\mathcal{G}_{n, J}$-analytic functions). 
The subspace are closed (cf.\ \cite[Lemme 2.5]{Sch10}), so it has the natural structure of Banach subspace induced from $C^{\an} (\mathcal{G}_n, B)$. 
There is a natural isomorphism of Banach spaces 
\[
C^{\mathcal{G}_{n, I} -\sm, \mathcal{G}_{n, J}-\an} (\mathcal{G}_n, E) \widehat{\otimes}_E B \cong C^{\mathcal{G}_{n, I} -\sm, \mathcal{G}_{n, J}-\an} (\mathcal{G}_n, B), 
\]
which preserves the norms. 

We extend the definition of the spaces of ($\mathbb{Q}_p$-)locally analytic vectors (\cite[\S 2]{ST02}) and its $\mathfrak{L}\mathfrak{A}$-acyclicity (cf.\ \cite[Definition 2.2.1]{Pan20} and also \cite[Definition 4.40]{JC21}) to $J$-locally analytic case. 

\begin{definition}
\label{def_sm_la}
For a positive integer $n \geq 1$, a Banach space $B$ equipped with a continuous action of $\mathcal{G}$, and a subset $J \subset \Sigma_{\wp}$, 
\begin{enumerate}
\item An element $b \in B$ is said to be \emph{$\mathcal{G}_{n, I}$-smooth $\mathcal{G}_{n, J}$-analytic} if $b$ is in the image of the evaluation map at the identity element $C^{\mathcal{G}_{n, I} -\sm, \mathcal{G}_{n, J}-\an} (\mathcal{G}_n, B)^{\mathcal{G}_n} \to B$. 
Let $B^{\mathcal{G}_{n, I} -\sm, \mathcal{G}_{n, J} -\an} \subset B$ denote the subspace of $\mathcal{G}_{n, I}$-smooth $\mathcal{G}_{n, J}$-analytic vectors. 
\item The space of \emph{$J$-locally analytic} vectors in $B$ is the subspace $B^{J-\la} \coloneqq \cup_{N \geq n} B^{\mathcal{G}_{N, J} -\an}$. 
It is stable under the $\mathcal{G}_n$-action and independent of the choice of $\mathcal{G}_0$ and $n$. 
\item For $i \geq 0$, the $i$-th $J$-locally analytic cohomology of $B$ is defined by 
\[
R^i \mathfrak{L} \mathfrak{A}_J (B) \coloneqq \varinjlim_N H_{\cont}^i (\mathcal{G}_n, B \widehat{\otimes}_E C^{\mathcal{G}_{N, J} -\an} (\mathcal{G}_n, E)). 
\]
\end{enumerate}
We say $B$ is \emph{$\mathfrak{L}\mathfrak{A}_J$-acyclic} if $R^i \mathfrak{L}\mathfrak{A}_J (B) = 0$ for $i \geq 1$. 
We often say $J = \mathbb{Q}_p$ instead of $J = \Sigma_{\wp}$. 
If $J = \mathbb{Q}_p$, we say \emph{$\mathfrak{L}\mathfrak{A}$-acyclic} instead of $\mathfrak{L}\mathfrak{A}_{\Sigma_{\wp}}$-acyclic. 
\end{definition}


In the case of $J = \mathbb{Q}_p$ (in other words, if ``$J-\an$'' = ``$\an$''), the following theorem is well-known. 

\begin{theorem}
\label{admissible_LAacyclic}
Any admissible $\mathbb{Q}_p$-Banach representation $B$ of a $p$-adic Lie group is $\mathfrak{L}\mathfrak{A}$-acyclic (in the sense of \cite[Definition 2.2.1]{Pan20}). 
\end{theorem}

\begin{proof}
See \cite[Theorem 7.1]{ST03} (and also \cite[Theorem 2.2.3]{Pan20}, \cite[Proposition 2.3.1]{Cam22b}). 
\end{proof}


For a later purpose, we need to introduce the derived version of $\mathfrak{L}\mathfrak{A}$, using the condensed settings as in \cite{JC21}. 
We use the notation of \cite[\S 4]{JC21}. 
For $h \geq 0$, we can attach a rigid analytic group $\mathbb{G}^{(h)}$ such that $\mathcal{G}_0 \subset \mathbb{G}^{(h)} (\mathbb{Q}_p)$ as in \cite[Definition 4.7]{JC21} or \cite[Definition 2.1.2]{JC23}. 
Note that $\mathbb{G}_I^{(h)}$-smooth $\mathbb{G}_J^{(h)}$-analytic functions on $\mathbb{G}^{(h)}$ can be similarly defined in the case of $p$-adic Lie groups and let 
\[
C^{\mathbb{G}_I^{(h)}-\sm, \mathbb{G}_J^{(h)}-\an} (\mathbb{G}^{(h)}, E) \subset C^{\mathbb{G}_J^{(h)}-\an} (\mathbb{G}^{(h)}, E)
\]
denote the subspace of $\mathbb{G}_I^{(h)}$-smooth $\mathbb{G}_J^{(h)}$-analytic functions. 
We can consider the subalgebra of $\mathbb{G}_I^{(h)}$-smooth $\mathbb{G}_J^{(h)}$-analytic functions (cf.\ \cite[Example 2.3]{Sch10})
\[
C^{\mathbb{G}_I^{(h)}-\sm, \mathbb{G}_J^{(h)}-\an} (\mathbb{G}^{(h)}, E) \subset C^{\mathbb{Q}_p -\an} (\mathbb{G}^{(h)}, E) \coloneqq C^{\an} (\mathbb{G}^{(h)}, E). 
\]
The subalgebra $C^{\mathbb{G}_I^{(h)}-\sm, \mathbb{G}_J^{(h)}-\an} (\mathbb{G}^{(h)}, E)$ is also an affinoid ring, and so it provides a natural analytic ring that we write $C^{\mathbb{G}_I^{(h)}-\sm, \mathbb{G}_J^{(h)}-\an} (\mathbb{G}^{(h)}, E)_{\blacksquare}$. 
For an object $C$ in the category $D(E_{\blacksquare} [\mathcal{G}_0])$ of solid $\mathcal{G}_0$-modules, we define the \emph{derived $\mathbb{G}_I^{(h)}$-smooth $\mathbb{G}_J^{(h)}$-analytic vectors} of $C$ 
\[
C^{R\mathbb{G}_I^{(h)}-\sm, \mathbb{G}_J^{(h)} -\an} \coloneqq R\underline{\Hom}_{E_{\blacksquare}[G]} (E, (C^{R\mathbb{G}_I^{(h)}-\sm, \mathbb{G}_J^{(h)} -\an} (\mathbb{G}^{(h)}, E)_{\blacksquare} \otimes_{E_{\blacksquare}} C)_{\star_{1, 3}})
\]
as an object in $D(E_{\blacksquare} [\mathcal{G}])$, where the action of $G$ is induced by the $\star_2$-action. 
See \cite[Remark 4.23, Definition 4.24]{JC21} for the notation of $(-)_{\star_I}$ for $I \subset \{1, 2, 3\}$. 
We set 
\[
C^{R\mathbb{G}_I^{(h)}-\sm, \mathbb{G}_J^{(h^+)}-\an} \coloneqq R\varprojlim_{h' < h} C^{R\mathbb{G}_I^{(h)}-\sm, \mathbb{G}_J^{(h')}-\an}. 
\]
We also define the \emph{derived $J$-locally analytic vectors} of $C$ 
\[
C^{RJ-\la} \coloneqq \mathop{\mathrm{hocolim}}_{h \to \infty} C^{R\mathbb{G}_J^{(h)}-\an} = \mathop{\mathrm{hocolim}}_{h \to \infty} C^{R\mathbb{G}_J^{(h^+)}-\an}
\]
as a homotopy colimit in $D(E_{\blacksquare} [\mathcal{G}_0])$. 
We also set 
\begin{equation}
\label{CRG0Ism_Jla}
C^{R\mathcal{G}_{0, I}-\sm, J-\la} \coloneqq \mathop{\mathrm{hocolim}}_{h \to \infty} C^{R\mathbb{G}_I^{(h)}-\sm, R\mathbb{G}_J^{(h)}-\an} = \mathop{\mathrm{hocolim}}_{h \to \infty} C^{R\mathbb{G}_I^{(h^+)}-\sm, R\mathbb{G}_J^{(h^+)}-\an}. 
\end{equation}

We note that if the underlying topological module of $B = B[0]$ in $D(E_{\blacksquare} [\mathcal{G}_0])$ is a Banach space, we have 
\[
H^0 (B^{RJ-\la}) = R^0 \mathfrak{L}\mathfrak{A}_J (B) = B^{J-\la}. 
\]

In the remaining of this subsection, suppose that $\mathcal{G}$ be a $p$-adic Lie group such that $\mathcal{G} = \mathcal{G}' \times \mathcal{G}''$ for $p$-adic Lie groups $\mathcal{G}', \mathcal{G}''$. 
For a Banach $E$-representation $B$ of $\mathcal{G}$, Let us define a subspace of \emph{$\mathcal{G}'$-locally analytic} vectors (resp.\ \emph{$\mathcal{G}''$-locally analytic} vectors) $B^{\mathcal{G}'-\la}$ (resp.\ $B^{\mathcal{G}''-\la}$) as the subspace of elements in $B^{\la}$ killed by the Lie algebra $\Lie (\mathcal{G}'')$ (resp.\ $\Lie (\mathcal{G}')$). 

For $h > 0$, let ${\mathbb{G}'}^{(h)}$ and ${\mathbb{G}'}^{(h^+)}$ be the rigid analytic groups associated with $\mathcal{G}'_0$, and similarly define ${\mathbb{G}''}^{(h)}$ and ${\mathbb{G}''}^{(h^+)}$ as before. 
Let $\mathcal{G}'_0$, $\mathcal{G}''_0$ be uniform pro-$p$ subgroups of $\mathcal{G}'$, $\mathcal{G}''$ respectively, with $\mathcal{G}_0 \coloneqq \mathcal{G}'_0 \times \mathcal{G}''_0$. 
For $(\mathcal{H}, \mathcal{H}_0) = (\mathcal{G}', \mathcal{G}'_0)$ or $(\mathcal{G}'', \mathcal{G}''_0)$, let $\mathcal{D}_{(h)} (\mathcal{H}_0, \mathcal{O}_E)$ and $\mathcal{D}_{(h)} (\mathcal{H}_0, E)$ denote the (condensed) distribution algebras as defined in \cite[Definition 4.12]{JC21}. 
We also set $\mathcal{D}^{(h)}(\mathcal{H}_0, E) \coloneqq \varinjlim_{h' > h} \mathcal{D}^{(h')} (\mathcal{H}_0, E)$. 
We may also define the spaces of \emph{derived $\mathcal{G}'_0$-(locally) analytic vectors} $C^{R\an'}$, $C^{R\la'}$ (resp.\ \emph{\emph{derived $\mathcal{G}''_0$-(locally) analytic vectors} $C^{R\an''}$, $C^{R\la''}$}) of an object $C$ in $D(E_{\blacksquare} [\mathcal{G}_0])$. 

The following lemma is obtained by the same argument of the proof of \cite[Theorem 4.36 (1)]{JC21}. 

\begin{lemma}
\label{ladash_distribution}
Let $V$ be an object in $D(E_{\blacksquare})$ and $C$ be an object in $D(E_{\blacksquare} [\mathcal{G}_0])$. 
We naturally extend the action of $\mathcal{G}'_0$ on $\mathcal{D}^{(h^+)}(\mathcal{G}'_0, E)$ to $\mathcal{G}_0 = \mathcal{G}'_0 \times \mathcal{G}''_0$ so that $\mathcal{G}''_0$ acts trivially. 
There is a natural quasi-isomorphism in $D(E_{\blacksquare}[\mathcal{G}_0])$ 
\[
R\underline{\Hom}_{E_{\blacksquare}[\mathcal{G}_0]}(\mathcal{D}^{(h^+)}(\mathcal{G}'_0, E) \otimes_{K_{\blacksquare}}^L V, C) \cong R\underline{\Hom}_E (V, C^{R{\mathbb{G}''}^{(h^+)}-\sm, {\mathbb{G}'}^{(h^+)}-\an}). 
\]
\end{lemma}

We will use the following proposition in the later. 

\begin{proposition}
\label{Rlad_Rdd_Rla}
For an object $C$ in $D(E_{\blacksquare} [\mathcal{G}_0])$, there exists a natural quasi-isomorphism 
\[
R\underline{\Hom}_{E_{\blacksquare} [\mathcal{G}''_0]} (E, C^{R\la}) \xrightarrow{\sim} C^{R\mathcal{G}''_0 -\sm, \la'}. 
\]
where $R\mathcal{G}''_0 -\sm, \la'$ is defined as the equation \eqref{CRG0Ism_Jla}. 
\end{proposition}

\begin{proof}
The proof is similar to that of \cite[Theorem 5.3]{JC21}. 
By \cite[Remark 4.13]{JC21}, we have a natural isomorphism 
\begin{equation*}
\mathcal{D}_{(h)} (\mathcal{G}'_0, E) \otimes_{E_{\blacksquare}} \mathcal{D}_{(h)} (\mathcal{G}''_0, E) \cong \mathcal{D}_{(h)} (\mathcal{G}_0, E). 
\end{equation*}
By the isomorphism and taking $- \otimes_{E_{\blacksquare}} \mathcal{D}_{(h)} (\mathcal{G}'_0, E)$ to the quasi-isomorphism in Kohlhaase's theorem \cite[Theorem 5.8]{JC21}, we have a quasi-isomorphism
\begin{equation*}
\mathcal{O}_E \otimes_{\mathcal{O}_{E, \blacksquare} [\mathcal{G}''_0]}^L \mathcal{D}_{(h)} (\mathcal{G}_0, E) \cong \mathcal{D}_{(h)} (\mathcal{G}'_0, E). 
\end{equation*} 
By taking colimits, we also have 
\begin{equation}
\label{Kohlhaase_tensor}
\mathcal{O}_E \otimes_{\mathcal{O}_{E, \blacksquare} [\mathcal{G}''_0]}^L \mathcal{D}^{(h^+)} (\mathcal{G}_0, E) \cong \mathcal{D}^{(h^+)} (\mathcal{G}'_0, E). 
\end{equation}

We now compute 
\begin{align*}
&R\underline{\Hom}_{E_{\blacksquare}[\mathcal{G}''_0]} (E, C^{R\mathbb{G}^{(h^+)}-\an}) \\ 
&\cong R\underline{\Hom}_{E_{\blacksquare}[\mathcal{G}''_0]}(E, R\underline{\Hom}_{E_{\blacksquare}[\mathcal{G}]}(\mathcal{D}^{(h^+)}(\mathcal{G}_0, E), C)) & (\text{\cite[Theorem 4.36]{JC21}}) \\ 
&\cong R\underline{\Hom}_{E_{\blacksquare}[\mathcal{G}_0]}(E \otimes_{E_{\blacksquare}[\mathcal{G}''_0]}^L \mathcal{D}^{(h^+)}(\mathcal{G}_0, E), C) & (\text{tensor-Hom adjunction}) \\ 
&\cong R\underline{\Hom}_{E_{\blacksquare}[\mathcal{G}_0]}(\mathcal{D}^{(h^+)}(\mathcal{G}'_0, E), C) & (\text{\ref{Kohlhaase_tensor}}) \\ 
&\cong C^{R{\mathbb{G}''}^{(h^+)}-\sm, {\mathbb{G}'}^{(h^+)}-\an} & (\text{Lemma \ref{ladash_distribution}}). 
\end{align*}
Then the natural quasi-isomorphism of the lemma is obtained by taking the homotopy colimit of the obtained quasi-isomorphism along $h \to \infty$. 
\end{proof}

\begin{corollary}
\label{spec_seq_Rlad}
Let $V$ be a solid $E_{\blacksquare} [\mathcal{G}_0]$-module. 
There is a spectral sequence of solid $K$-vector spaces 
\[
E_2^{i, j} \coloneqq \underline{\Ext}_{K_{\blacksquare}[\mathcal{G}''_0]}^i (E, H^j (V^{R\la})) \Longrightarrow H^{i+j} (V^{R\mathcal{G}''_{0, I}-\sm, \la'}). 
\]
\end{corollary}

\begin{proof}
This follows from Proposition \ref{Rlad_Rdd_Rla}. 
\end{proof}

\subsection{Completed cohomology of unitary Shimura curves}

From here, we use the notation of $G$ in Section \ref{ss_Shimura_curves}. 

Let $K_{\wp, 0} \coloneqq \GL_2 (\mathcal{O}_{\wp}) \subset \GL_2 (F_{\wp})$ and $K_{\wp, m}$ be its congruence subgroup 
\[
K_{\wp, m} \coloneqq \left\{ \gamma \in \GL_2 (F_{\wp}) \ | \ \gamma \equiv \begin{pmatrix} 1 & 0 \\ 0 & 1 \\ \end{pmatrix} \mathrm{mod} \ \varpi^m \right\}
\]
for $m \geq 1$. 
Let $K_{\wp}$, $K^{\wp} = K_p^{\wp} K^p$ be a compact open subgroup of $\GL_2 (F_{\wp})$, $G(\mathbb{A}^{\infty, \wp}) = G(\mathbb{Q}_p^{\wp}) \times G(\mathbb{A}^{\infty, p})$ respectively. 

Recall that for an object $W$ in $\Rep_E (G^c)$, we have attached a $K_p$-equivariant \'etale (resp.\ Kummer \'etale) $E$-local system $W_{\et}$ (resp.\ $W_{\ket}$) on $\mathcal{S}h_{\et}$ (resp.\ $\mathcal{S}h_{\ket}^{\tor}$), base changing to $E/\mathbb{Q}_p$ the argument in Section \ref{ss_HT_period_maps}. 
Let $W_0$ be a $\mathcal{O}_E$-lattice of $W$ which is stable by $K_p^{\wp}$ and $W_{0, \et}$ (resp.\ $W_{0, \ket}$) be the associated \'etale (resp.\ pro-\'etale) sheaves. 
Let $j_{K_{\wp}, \ket} \colon \mathcal{S}h_{K_{\wp} K^{\wp}, \et} \to \mathcal{S}h_{K_{\wp} K^{\wp}, \proket}^{\tor}$ be the natural morphism of sites. 
We have natural isomorphisms $j_{K_{\wp}, \ket, \ast} W_{0, \et} \cong W_{0, \ket}$ and $j_{K_{\wp}, \ket}^{\ast} W_{0, \ket} \cong W_{0, \et}$ by the construction. 
Note that we also have a quasi-isomorphism 
\[
Rj_{K_{\wp}, \ket, \ast} (W_{0, \et} / \varpi_E^m) \cong (W_{0, \ket} / \varpi_E^m)[0]
\]
for $m \geq 1$ by \cite[Theorem 4.6.1]{DLLZ19}. 

Let us recall the (completed) \'etale cohomology of unitary Shimura curves. 
For each integer $i \geq 0$, let (cf.\ \cite[\S 2.1.1]{Eme06b}, \cite[Definition 6.1.1]{Cam22b})
\begin{align*}
H_?^i (K^{\wp}, W_0) &\coloneqq \varinjlim_{K_{\wp}} \varprojlim_s H_{\et, ?}^i (\mathcal{S}h_{K_{\wp} K^{\wp}}, W_{0, \et} / \varpi_E^s), \\ 
\widetilde{H}_?^i (K^{\wp}, W_0) &\coloneqq \varprojlim_s \varinjlim_{K_{\wp}} H_{\et, ?}^i (\mathcal{S}h_{K_{\wp} K^{\wp}}, W_{0, \et} / \varpi_E^s), \\ 
H_?^i (K^{\wp}, W_0)_E &\coloneqq H_?^i (K^{\wp}, W_0) \otimes_{\mathcal{O}_E} E, \\ 
\widetilde{H}_?^i (K^{\wp}, W_0)_E &\coloneqq \widetilde{H}_?^i (K^{\wp}, W_0) \otimes_{\mathcal{O}_E} E, 
\end{align*}
where $? \in \{ \emptyset, c \}$ and $K_{\wp}$ runs over compact open subgroups of $\GL_2 (F_{\wp})$ which stabilizes $W_0$. 
Let 
\begin{align*}
H_?^i (K^{\wp}, W) &\coloneqq \varinjlim_{W_0} H_?^i (K^{\wp}, W_0)_E, \\ 
\widetilde{H}_?^i (K^{\wp}, W) &\coloneqq \varinjlim_{W_0} \widetilde{H}_?^i (K^{\wp}, W_0)_E, 
\end{align*}
where $W_0$ runs over $\mathcal{O}_E$-lattices of $W$ stable by $K_p^{\wp}$. 
The $p$-adic Lie group $\GL_2 (F_{\wp})$ naturally acts on these modules. 


For a general $\mathcal{O}_E$-lattice $W_0 \subset W$, we can similarly define (completed) \'etale cohomology with $W$-coefficients $H^i (K^p, W_0)$, $\widetilde{H}^i (K^p, W_0)$, and so on, taking towers of Shimura varieties of level $K'_p \subset K_p$, instead of level $K'_{\wp} \subset K_{\wp}$. 
We may also consider the space of \emph{$J$-locally analytic} vectors of the completed cohomology where $J \subset \{ z \} \cup \bigsqcup_i \Sigma_{\wp_i}$ and $z$ is the symbol which corresponds to the continuous representation of the component of similitudes $\mathbb{G}_m \subset G(\mathbb{Q}_p)$. 
If $J$ is of the form $\bigsqcup_{i \in J' \subset [1, s]} \Sigma_{\wp_i}$ (resp.\ $\{ z \} \cup \bigsqcup_{i \in J' \subset [1, s]} \Sigma_{\wp_i}$), we may define the terminology ``\emph{$J'$-locally analytic}'' (resp.\ ``$(\{ z \} \cup J')$-locally analytic'') and (derived) $J$-locally analytic vectors as in Section \ref{J-la_vectors}. 
Moreover, if $J = \{ \sigma \}$ for an element $\sigma \in \{ z \} \cup \bigsqcup_{\wp_i | p} \Sigma_{\wp_i}$, we call \emph{$\sigma$-locally analytic} instead of $\{ \sigma \}$-locally analytic. 
Note that $\widetilde{H}^i (K^p, W_0)$ is ($\mathbb{Q}_p$-)locally analytic if and only if $(\{ z \} \cup \bigsqcup_i \Sigma_{\wp_i})$-locally analytic. 
For a while, we consider completed cohomology $\widetilde{H}_?^i (K^p, -)$ rather than $\widetilde{H}_?^i (K^{\wp}, -)$. 
The advantage of using them is that they can be embedded in the completed cohomology of Siegel modular varieties, this argument is treated in Section \ref{sec_la_comp_coh}.

By Corollary \ref{spec_seq_Rlad} and the proof of \cite[Theorem 1.1.13]{Eme06b} (see also \cite[Corollary 2.2.11]{QS25}), we have the analogous assertion for admissible Banach representations. 

\begin{proposition}
\label{spec_seq_RJlad}
For a subset $J \subset \{ z \} \cup \bigsqcup_{\wp_i | p} \Sigma_{\wp_i}$ and an admissible Banach representation $V$ of $G$ over $E$, we have 
\[
V^{RJ-\la} = R\Gamma(\mathfrak{g}^J, V^{R\la})
\]
\end{proposition}


\begin{definition}
Let $j \colon \mathcal{S}h_{K^p} \hookrightarrow \mathcal{S}h_{K^p}^{\tor}$ be an open immersion of perfectoid Shimura varieties. 
For a pro-Kummer \'etale $\mathcal{O}_E$-local system $\Lambda$ on $\mathcal{S}h_{K^p, \proet}$, let $j_! \Lambda \coloneqq R \varprojlim_s j_! (\Lambda / \varpi_E^s)$ be the (derived) $p$-adic completion of the extension by zero of the pro-Kummer \'etale local systems $\Lambda / \varpi_E^s$ over $\mathcal{S}h_{K^p, \proet}$. 
\end{definition}

\begin{proposition}
\label{proket_infty_W0_adm}
Let $\Lambda = W_{0, \proket}$ or $W_{0, \proket} / \varpi_E^m$. 
The objects in the derived category of $\mathcal{O}_E$-local systems 
\[
R\Gamma_{\proket} (\mathcal{S}h_{K^p}^{\tor}, \Lambda) \quad \text{and} \quad R\Gamma_{\proket} (\mathcal{S}h_{K^p}^{\tor}, j_! \Lambda)
\]
are represented by bounded complexes of admissible $K_{\wp}$-representations. 

In particular, 
\[
H_{\proket}^i (\mathcal{S}h_{K^p}^{\tor}, \Lambda) \quad \text{and} \quad H_{\proket}^i (\mathcal{S}h_{K^p}^{\tor}, j_! \Lambda)
\]
are admissible $K_p$-representations for $i \geq 0$. 
\end{proposition}

\begin{proof}
The pullback of the pro-Kummer \'etale sheaf $\Lambda$ by the map of sites $\mathcal{S}h_{K^p, \proket}^{\tor} \to \mathcal{S}h_{K_p K^p, \proket}^{\tor}$ is the constant pro-Kummer \'etale sheaf $\underline{W_0}$ associated with $W_0$, which is equipped with the $G$-action. 
So we have a quasi-isomorphism 
\[
R\Gamma_{\proket} (\mathcal{S}h_{K^p}^{\tor}, \Lambda) \xrightarrow{\sim} \varinjlim_{K'_p} R\Gamma_{\ket} (\mathcal{S}h_{K'_p K^p}^{\tor}, \underline{W_0})
\]
by \cite[Proposition 5.1.6]{DLLZ19}. 
Then, the proof is similar to \cite[Theorem 6.1.6]{Cam22b}. 
\end{proof}

\begin{corollary}
\label{infin_proket_comp_comparison}
For $i \geq 0$, we have natural $K_p$-equivariant isomorphisms 
\begin{align*}
H_{\proket}^i (\mathcal{S}h_{K^p}^{\tor}, W_{0, \proket}) &\cong \widetilde{H}^i (K^p, W_{0, \proket}), \\ 
H_{\proket}^i (\mathcal{S}h_{K^p}^{\tor}, j_! W_{0, \proket}) &\cong \widetilde{H}_c^i (K^p, W_{0, \proket}). 
\end{align*}
\end{corollary}

\begin{proof}
By Proposition \ref{proket_infty_W0_adm}, the proof is similar to \cite[Corollary 6.1.7]{Cam22b}. 
\end{proof}

\subsection{Comparison between $J$-locally analytic completed cohomology and the cohomology of perfectoid unitary Shimura curves}
\label{ss_comparison_theorem}

In this subsection and the next section, we will use the terminology of geometric Sen theory in \cite{Cam22a}. 
In this subsection, we suppose one of the following two conditions on $(G, X)$, $E$, and $J$:
\begin{condition}
\label{condition_comparison}
\begin{enumerate}
\item $(G, X)$ is a Shimura data as in Section \ref{ss_Shimura_curves}, $E / \mathbb{Q}_p$ is a finite extension such that $|\Hom_{\mathbb{Q}_p} (F_{\wp_i}, E) | = [F_{\wp_i} : \mathbb{Q}_p]$ for each $\wp_i | p$, and $J \in \{ z \} \cup \bigsqcup_{\wp_i | p} \Sigma_{\wp_i}$. 
\item $(G, X)$ is a Shimura data of Hodge type (Section \ref{ss_Shimura_Hodge}), $E = \mathbb{Q}_p$, and $J = \mathbb{Q}_p$. 
\end{enumerate}
\end{condition}
The following argument is based on \cite[\S 6.2]{Cam22b} (in which the case of general Shimura varieties and its $\mathbb{Q}_p$-locally analytic completed cohomology is treated, considering the pro-Kummer \'etale sites on infinite level Shimura varieties). 

Let $W$ be an object in $\Rep_E (G^c)$ and $W_0$ be its $\mathcal{O}_E$-lattice. 
As in the previous section, for each compact open subgroup $K_p K^p \subset G(\mathbb{Q}_p) \times G(\mathbb{A}^{\infty, p})$, let $\mathcal{S}h = \mathcal{S}h_{K_p K^p}$ denote the rigidification of the Shimura variety over $\mathbb{C}_p$ associated with the datum $(G, X)$. 
Also let $\mathcal{S}h_{K_p K^p}^{\tor}$ be a toroidal compactification of $\mathcal{S}h_{K_p K^p}$ over $\Spa(\mathbb{C}_p, \mathcal{O}_{\mathbb{C}_p})$. 
Note that in the case of unitary Shimura curves, they are compact and so we can assume that $\mathcal{S}h_{K_p K^p} = \mathcal{S}h_{K_p K^p}^{\tor}$. 

Let $D \coloneqq \mathcal{S}h_{K_p K^p}^{\tor} \backslash \mathcal{S}h_{K_p K^p}$ be the normal crossing divisor in $\mathcal{S}h_{K_p K^p}^{\tor}$. 
We can write $D = \bigcup_{a \in I} D_{K_p, a}$ for a finite set $I$ such that each $D_{K_p, a}$ is irreducible and $D_{K_p, I'} \coloneqq \cap_{a \in I'} D_{K_p, a}$ is a smooth divisor for each subset $I' \subset I$. 
For a level $K'_p \subset K_p$, let $D_{K'_p, I'} \coloneqq D_{K_p, I'} \times_{\mathcal{S}h_{K_p K_p}^{\tor}} \mathcal{S}h_{K'_p K_p}^{\tor}$ and let $\iota_{K'_p, I'} \colon D_{K'_p, I'} \hookrightarrow \mathcal{S}h_{K'_p K^p}^{\tor}$ denote the closed immersion. 
Write 
$\widehat{\mathcal{O}}_{I'}$ for the pro-Kummer \'etale sheaf of completed (bounded) functions on $D_{K_p, I'}$, endowed with the induced log structure. 
Let 
$\widehat{\mathcal{I}}_D$ be the kernel of the map of pro-Kummer \'etale sheaves 
\[
\widehat{\mathcal{I}}_{\mathcal{S}h} \coloneqq \ker \left( \widehat{\mathcal{O}}_{\mathcal{S}h} \to \bigoplus_{a \in I} \iota_{K_p, a, \ast} \widehat{\mathcal{O}}_a \right). 
\]


Let $\pi_{K_p} \colon \mathcal{S}h_{K^p} \to \mathcal{S}h_{K_p K^p}$ denote the natural projection. 
The map $\pi_{K_p}$ is a pro-\'etale Galois covering with Galois group 
\[
\widetilde{K}_p \coloneqq K_p K^p / (K_p K^p \cap \overline{Z(\mathbb{Q})} \cdot K^p). 
\]
Similarly let $\pi_{K_p}^{\tor} \colon \mathcal{S}h_{K^p}^{\tor} \to \mathcal{S}h_{K_p K^p}^{\tor}$ be the natural projection. 
Set $d \coloneqq \dim \mathcal{S}h_{K_p K^p}$. 
For $n \geq 1$ an integer, let 
\[
\mathbb{T}_n \coloneqq \Spa (\mathbb{C}_p \langle T^{\pm \frac{1}{n}} \rangle, \mathcal{O}_{\mathbb{C}_p} \langle T^{\pm \frac{1}{n}}) \rangle \quad \text{and} \quad \mathbb{D} \coloneqq \Spa (\mathbb{C}_p \langle S^{\frac{1}{n}} \rangle, \mathcal{O}_{\mathbb{C}_p} \langle S^{\frac{1}{n}} \rangle). 
\]
and $\mathbb{T} \coloneqq \mathbb{T}_1$ and $\mathbb{D} \coloneqq \mathbb{D}_1$. 
Write $\mathbb{T}_{\infty} \coloneqq \varprojlim_n \mathbb{T}_n$ and $\mathbb{D}_{\infty} \coloneqq \varprojlim_n \mathbb{D}_n$. 

We assume the following conditions on $(U, V)$:
\begin{condition}
\label{U_V_conditions}
For open subspaces $U \subset \mathcal{S}h_{K_p K^p}^{\tor}$ and $V \subset \mathscr{F}\ell_{\mu}$, 
\begin{enumerate}
\item The $\mathcal{O}_{\mathscr{F}\ell_{\mu}} (V)$-sheaves $\mathfrak{n}_{\mu}^0$ and $\mathfrak{g}^{c, 0} / \mathfrak{n}_{\mu}^0$, where $\mathfrak{g}^c \coloneqq \Lie G_{\mathbb{Q}_p}^c$, and $\mathfrak{g}^{c, 0}$, $\mathfrak{n}_{\mu}^0 = \mathfrak{n}_{\mu}^{c, 0}$ are constructed as in Section \ref{ss_HT_period_maps}, are finite free. 
\item The compact open subgroup $K_p \subset G(\mathbb{Q}_p)$ stabilizes $V$. 
\item $\pi_{K_p}^{\tor, -1} (U) \subset \pi_{\HT}^{\tor, -1} (V)$ is an affinoid perfectoid space. 
\item $U$ admits a toric chart, i.e., there is a map $U \to \mathbb{T}^e \times \mathbb{D}^{d-e}$ for some $e \in [0, d]$ that factors as a composite of a finite \'etale map and a rational localization.  
\item $U$ has the log structure obtained from $U \to \mathbb{T}^e \times \mathbb{D}^{d-e}$ by pullback. 
\end{enumerate}
\end{condition}
Note that such $U$ covers the whole $\mathcal{S}h_{K_p K^p}^{\tor}$ by Proposition \ref{HT_period_maps} (2) and \cite[Corollary 3.11]{DLLZ19}. 
For an object $U'$ in $U_{\proket}$, let $U'_{I'}$ be the space $D_{I'} \times_{\mathcal{S}h} U'$ and $\widetilde{U}' \coloneqq \pi_{K_p}^{\tor, -1} (U) \times_U U'$ with the induced log structures. 
Let $U_n \coloneqq U \times_{(\mathbb{T}^e \times \mathbb{D}^{d-e})} (\mathbb{T}_n^e \times \mathbb{D}_n^{d-e})$, $U_{\infty} \coloneqq \varprojlim_n U_n$, and $U_{p^{\infty}} \coloneqq \varprojlim_n U_{p^n}$. 
Let $\Gamma \cong \widehat{\mathbb{Z}}^d$ (resp.\ $\Gamma_J \cong \widehat{\mathbb{Z}}^{|J \cap \{ e+1, \ldots, d \} |}$) be the pro-Kummer \'etale Galois group of $\mathbb{T}_{\infty}^e \times \mathbb{D}_{\infty}^{d-e} \to \mathbb{T}^e \times \mathbb{D}^{d-e}$ (resp.\ $U_{\infty, I'} \to U_{I'}$). 
Write $\Gamma \coloneqq \Gamma_p \times \Gamma^p$ and $\Gamma_{I'} \coloneqq \Gamma_{I', p} \times \Gamma_{I'}^p$ as a product of its prime-to-$p$ part and its pro-$p$-Sylow subgroup. 
Note that $U_{p^{\infty}}$, $\widetilde{U}_{p^{\infty}}$, $U_{\infty}$, and $\widetilde{U}_{\infty}$ are log affinoid perfectoid in the sense of \cite[Definition 5.3.1]{DLLZ19}. 

Let $\widetilde{\mathfrak{g}} \coloneqq \Lie G^c$ with its dimension $g \coloneqq \dim_{\mathbb{Q}_p} \widetilde{\mathfrak{g}}$. 

\begin{lemma}
\label{n_cohomology_induce}
For $i \geq 0$, $I' \subset I$ and a pro-Kummer \'etale $\widehat{\mathcal{O}}_{I'}$-module $\mathcal{F}$ which is the completed tensor product $\widehat{\mathcal{O}}_{I'} \widehat{\otimes}_{{\mathbb{Q}}_p} V_{\proket}$, where $V_{\proket}$ is associated with a colimit of ON relative locally analytic $E$-representation $V$ of $G^c$, there is a natural isomorphism 
\[
H_{\proket}^i (U_{I'}, \widehat{O}_{I'} \widehat{\otimes}_{\mathbb{Q}_p} \mathcal{F}) \cong H^i (\mathfrak{n}_{\mu}^0, \widehat{\mathcal{O}}_{I'} (\widetilde{U}_{p^{\infty}, I'}) \widehat{\otimes}_{\mathbb{Q}_p} \mathcal{F})^{\widetilde{K}_p \times \Gamma_{I', p}}, 
\]
where $\mathfrak{n}_{\mu}^0$ acts on $J$-locally analytic pro-Kummer \'etale $E$-local systems on $\widehat{\mathcal{O}}_{\mathcal{S}h^{\tor}}$ via the natural embedding $\pi_{\HT}^{\tor, \ast} (\mathfrak{n}_{\mu}^0) \hookrightarrow \underline{\widetilde{\mathfrak{g}}}_{\proket} |_{\mathcal{S}h_{K^p}^{\tor}}$ of the pullback of the vector bundle $\mathfrak{n}_{\mu}^0$ on $\mathscr{F}\ell_{\mu}$ into the constant sheaf $\underline{\widetilde{\mathfrak{g}}}_{\proket} |_{\mathcal{S}h_{K^p}^{\tor}}$. 
\end{lemma}

\begin{proof}
We follow the argument of Step 1 and Step 2 of the proof in \cite[Proposition 6.2.8]{Cam22b}. 
There is an equivalence of topoi $\widetilde{U}_{\infty, I', \proet}^{\sim} \cong \widetilde{U}_{\infty, I', \proket}^{\sim}$ by \cite[Lemma 5.3.8]{DLLZ19}. 
So by \cite[Lemma 4.12]{Sch13} and the definition of ON relative locally analytic pro-Kummer \'etale $\widehat{\mathcal{O}}_{I'}$-module, we have 
\[
R\Gamma_{\proket} (\widetilde{U}_{\infty, I'}, \mathcal{F}) = \widehat{\mathcal{O}}_{I'} (\widetilde{U}_{\infty, I'}) \widehat{\otimes}_{\mathbb{Q}_p} \mathcal{F}. 
\]

By the above equality and that $\widetilde{U}_{\infty, I'} \to U_{I'}$ is a pro-Kummer \'etale $\widetilde{K}_p \times \Gamma_{I'}$-torsor, there is a quasi-isomorphism 
\[
R\Gamma_{\proket} (U_{I'}, \widehat{\mathcal{O}}_{I'} \widehat{\otimes}_{\mathbb{Q}_p} \mathcal{F}) \cong R\Gamma_{\cont} (\widetilde{K}_p \times \Gamma_{I'}, \widehat{\mathcal{O}}_{I'} (\widetilde{U}_{\infty, I'}) \widehat{\otimes}_{\mathbb{Q}_p} \mathcal{F} (\widetilde{U}_{\infty, I'})). 
\]

Since $\Gamma_{I'}^p$ is a prime-to-$p$ profinite group, we have quasi-isomorphisms 
\begin{align*}
&R\Gamma_{\cont} (\widetilde{K}_p \times \Gamma_{I'}, \widehat{\mathcal{O}}_{I'} (\widetilde{U}_{\infty, I'}) \widehat{\otimes}_{\mathbb{Q}_p} \mathcal{F} (\widetilde{U}_{\infty, I'}) \\  
&\cong R\Gamma_{\cont} (\widetilde{K}_p \times \Gamma_{I', p}, R\Gamma_{\cont} (\Gamma_{I'}^p, \widehat{\mathcal{O}}_{I'} (\widetilde{U}_{\infty, I'}) \widehat{\otimes}_{\mathbb{Q}_p} \mathcal{F} (\widetilde{U}_{\infty, I'}))) \\ 
&\cong R\Gamma_{\cont} (\widetilde{K}_p \times \Gamma_{I', p}, \widehat{\mathcal{O}}_{I'} (\widetilde{U}_{p^{\infty}, I'}) \widehat{\otimes}_{\mathbb{Q}_p} \mathcal{F}). 
\end{align*}

We use geometric Sen theory to translate cohomology of the obtained complex into Sen operator cohomology. 
The triple $(\widehat{\mathcal{O}}_{I'} (\widetilde{U}_{p^{\infty, I'}}), \widetilde{K}_p \times \Gamma_{I', p}, \mathrm{pr}_2)$, where $\mathrm{pr}_2 \coloneqq \widetilde{K}_p \times \Gamma_{I', p} \to \Gamma_{I', p}$ is the projection, is a (strongly decomposable) Sen theory in the sense of \cite[Definition 2.2.6]{Cam22a}. 

By the assumption, we can attach a (geometric) Sen operator 
\[
\theta_{\mathcal{F}} \colon \mathcal{F} \to \mathcal{F} \otimes_{\mathbb{Q}_p} \Lie (\widetilde{K}_p \times \Gamma_{I', p})^{\vee}
\]
for $\mathcal{F}$ (\cite[Definition 2.5.3]{Cam22a}) and the geometric Sen operator of $\mathcal{S}h_{K^p}$ (\cite[Proposition 3.2.6]{Cam22a})
\[
\theta_{\mathcal{S}h} \coloneqq \theta_{\mathcal{S}h_{K^p}} \colon \widehat{\mathcal{O}}_{\mathcal{S}h} \otimes_{\mathbb{Q}_p} (\Lie G)_{\proket}^{\vee} \to 
\widehat{\mathcal{O}}_{\mathcal{S}h} (-1) \otimes_{\mathcal{O}_{\mathcal{S}h}} \Omega_{\mathcal{S}h}^1 (\log). 
\]

Then for $i \geq 0$, we have 
\begin{align*}
H^i (\widetilde{K}_p \times \Gamma_{I', p}, \widehat{\mathcal{O}}_{I'} (\widetilde{U}_{p^{\infty}, I'}) \widehat{\otimes}_{\mathbb{Q}_p} \mathcal{F}) &\cong H^i (\theta_{\mathcal{F}}, \widehat{\mathcal{O}}_{I'} (\widetilde{U}_{p^{\infty}, I'}) \widehat{\otimes}_{\mathbb{Q}_p} \mathcal{F})^{\widetilde{K}_p \times \Gamma_{I', p}} \\ 
&\cong H^i (\theta_{\mathcal{S}h}, \widehat{\mathcal{O}}_{I'} (\widetilde{U}_{p^{\infty}, I'}) \widehat{\otimes}_{\mathbb{Q}_p} \mathcal{F})^{\widetilde{K}_p \times \Gamma_{I', p}} \\ 
&\cong H^i (\mathfrak{n}_{\mu}^0, \widehat{\mathcal{O}}_{I'} (\widetilde{U}_{p^{\infty}, I'}) \widehat{\otimes}_{\mathbb{Q}_p} \mathcal{F})^{\widetilde{K}_p \times \Gamma_{I', p}}, 
\end{align*}
where the first isomorphism is from \cite[Proposition 2.5.4, Remark 2.5.8]{Cam22a}, the second isomorphism is from \cite[Theorem 3.3.2 (5)]{Cam22a}, and the last isomorphism is from the computation of $\theta_{\mathcal{S}h}$ (\cite[Theorem 5.2.5]{Cam22b}). 
\end{proof}

Let $C^{J-\la} (\widetilde{\mathfrak{g}}, W) \coloneqq \varinjlim_{K_p} C^{J-\la} (\widetilde{K}_p, W)$. 
We recall \cite[Definition 2.1.2]{JC23}. 
Let $\mathfrak{H}_1, \ldots, \mathfrak{H}_g$ be a basis of $\widetilde{\mathfrak{g}}$ over $\mathbb{Q}_p$. 
For $h \gg 0$, the exponential of the $\mathbb{Z}_p$-lattice $\mathfrak{K}_h$ generated by $\{ p^h \mathfrak{H}_k \}_k$ defines an affinoid group $\mathbb{G}_h$ isomorphic to a closed polydisc of dimension $g$, and let $\mathring{\mathbb{G}}_h \coloneqq \bigcup_{h' > h} \mathbb{G}_{h'}$. 
Let $\mathbb{G}^{(h)}$ be the rigid analytic group containing $\widetilde{K}_p$ as defined in loc.\ cit.\ (applying in the case of $G = \widetilde{K}_p$) for $h \in \mathbb{Q}_{> 0}$. 
For $\mathbb{G} = \mathbb{G}_h, \mathring{\mathbb{G}}_h, \mathbb{G}^{(h)}$, the spaces $C^{J-\la} (\mathbb{G}, B)$ of $J$-locally analytic functions $\mathbb{G} \to B$ for a Banach space $B$ is defined as in Section \ref{J-la_vectors}. 

\begin{lemma}
\label{J-la_space_ON_colom}
For $I' \subset I$, the sheaf $\widehat{\mathcal{O}}_{I'} \widehat{\otimes}_{\mathbb{Q}_p} C^{J-\la} (\widetilde{K}_p, W)_{\proket}$ is a colimit of ON relative locally analytic pro-Kummer \'etale Banach $\widehat{\mathcal{O}}_{I'}$-modules. 
In particular, this sheaf satisfies the conditions of Lemma \ref{n_cohomology_induce}. 
\end{lemma}

\begin{proof}
We can write as $C^{\mathbb{Q}_p -\la} (\widetilde{K}_p, W) = \varinjlim_h C^{\mathbb{Q}_p -\an} (\mathbb{G}^{(h)}, W)$ (see \cite[Definition 2.1.10]{JC23}). 
For each $h > 0$, the inclusion $C^{\mathbb{Q}_p -\an} (\mathbb{G}^{(h)}, W) \subset C^{\mathbb{Q}_p -\an} (\mathbb{G}^{(h)}, W)$ is a closed embedding of ON relative locally analytic Banach spaces. 

Then the section $H^0 (\widetilde{U}_{I'}, C^{J-\la} (\widetilde{K}_p, W)_{\proket})$ on the pro-Kummer \'etale cover $\widetilde{U}_{I'} \to U_{I'}$ equals to $C^{J-\la} (\widetilde{K}_p, W)$ and the lemma follows. 
\end{proof}

Set $A \coloneqq \widehat{\mathcal{O}}_{I'} (\widetilde{U}_{p^{\infty, I'}})$. 
We write 
\[
A \widehat{\otimes}_{\mathbb{Q}_p} C^{J-\la} (\widetilde{\mathfrak{g}}, W) \coloneqq \varinjlim_{K'_p \subset K_p} (A \widehat{\otimes}_{\mathbb{Q}_p} \varinjlim_{K_p} C^{J-\la} (\widetilde{K}_p, W))^{\widetilde{K}'_p -\sm}, 
\]
where $K'_p$ is a compact open subgroup of $K_p$, for the subspace of smooth vectors. 

\begin{lemma}
\label{n0_cohomology_polydisc}
There is a quasi-isomorphism 
\[
R\Gamma(\mathfrak{n}_{\mu}^0, A \widehat{\otimes}_{\mathbb{Q}_p} C^{J-\la} (\widetilde{\mathfrak{g}}, W)) \cong \varinjlim_{r \to \infty} (A \widehat{\otimes}_{\mathcal{O}(V)} \mathcal{O}(\mathbb{X}'_r) \widehat{\otimes}_E W), 
\]
where $\{ \mathbb{X}'_r \}_{r \gg 0}$ is a shrinking system of closed polydiscs of the same dimension over $V$. 
\end{lemma}

\begin{proof}
We follow the argument of Step 3 and Step 4 of the proof of \cite[Proposition 6.2.8]{Cam22b}. 
The case of (1) in Condition \ref{condition_comparison}, the lemma follows from loc.\ cit., so we may assume the condition (2). 
Recall that 
\begin{equation}
\label{curve_Liealg_E_decomp}
\mathfrak{g}_E \coloneqq \mathfrak{g} \otimes_{\mathbb{Q}_p} E \cong \mathfrak{z}_E \times \prod_{\sigma \in \bigsqcup_{\wp_i | p} \Sigma_{\wp_i}} \mathfrak{g}\mathfrak{l}_{2, E, \sigma} = \prod_{\sigma \in \{ z \} \cup \bigsqcup_{\wp_i | p} \Sigma_{\wp_i}} \mathfrak{g}_{E, \sigma}, 
\end{equation}
where $\mathfrak{z}_E \coloneqq \Lie \mathbb{G}_{m, E}$, so we can write 
\[
\widetilde{\mathfrak{g}}_E = \prod_{\sigma \in \{ z \} \cup \bigsqcup_{\wp_i | p} \Sigma_{\wp_i}} \widetilde{\mathfrak{g}}_{E, \sigma}, 
\]
where $\widetilde{\mathfrak{g}}_{E, \sigma}$ is a quotient of $\mathfrak{g}_{E, \sigma}$ and set 
\[
\widetilde{\mathfrak{g}}_{J, E} \coloneqq \prod_{\sigma \in J} \widetilde{\mathfrak{g}}_{E, \sigma}. 
\]
We also let 
\[
\mathfrak{n}_{\mu, E} \coloneqq \mathfrak{n}_{\mu} \otimes_{\mathbb{Q}_p} E \cong \mathfrak{n}_{E, \tau}, 
\]
where $\tau \in \Sigma$ is the fixed place and $\mathfrak{n}_{E, \tau} \subset \widetilde{\mathfrak{g}}_{E, \tau}$ is a one dimensional unipotent subalgebra, by Proposition \ref{Sh_curve_HT_filt_-1} and its following argument. 
Let $\mathfrak{N}$ be a generator of $\mathfrak{n}_{E, \tau}$. 
Also let $\widetilde{\mathfrak{g}}^0 \coloneqq \mathcal{O}_{\mathscr{F}\ell_{\mu}} \otimes_{\mathbb{Q}_p} \widetilde{\mathfrak{g}}^0$ and $\widetilde{\mathfrak{g}}_E^0 \coloneqq \mathcal{O}_{\mathscr{F}\ell_{\mu}} \otimes_{\mathbb{Q}_p} \widetilde{\mathfrak{g}}_E^0$. 


Recall that each element of the set $\{ z \} \cup \bigsqcup_{\wp_i | p} \Sigma_{\wp_i}$ corresponds to each component of the direct product decomposition of $\widetilde{\mathfrak{g}}_E^0$ in \eqref{curve_Liealg_E_decomp}. 
For each $\sigma \in J \backslash \{ \tau \}$, let $\{ \mathfrak{Y}_{1, \sigma}, \ldots, \mathfrak{Y}_{g_{\sigma}, \sigma} \}$,  
where each $g_{\sigma} \geq 0$ is an integer, be a basis of $\widetilde{\mathfrak{g}}_{E, \sigma}^0$ over the fixed $V \subset \mathscr{F}\ell_{\mu}$. 
If $\tau \in J$, for an integer $g_{\tau} \geq 0$, we also let $\{ \mathfrak{Y}_{1, \tau}, \ldots, \mathfrak{Y}_{g_{\tau}, \tau} \}$ be a set of elements in $\widetilde{\mathfrak{g}}_{E, \tau}^0$ such that $\mathfrak{Y}_{1, \tau}, \ldots, \mathfrak{Y}_{g_{\tau}, \tau}$, and $\mathfrak{N}$ generates $\widetilde{\mathfrak{g}}_{E, \tau}^0$. 

For $r \gg 0$, we can construct an affinoid rigid space $\mathbb{X}_{J, r}$ over $V_E$ isomorphic to a polydisc of dimension $g$ from the basis 
\[
\left\{ \begin{array}{ll} \{ p^r \mathfrak{Y}_{s, \sigma}, p^r \mathfrak{N} \}_{\sigma \in J \backslash \{ \tau \}, \ 1 \leq s \leq g_{\sigma}} & (\tau \in J) \\ \{ p^r \mathfrak{Y}_{s, \sigma} \}_{\sigma \in J, \ 1 \leq s \leq g_{\sigma}} & (\tau \not\in J) 
\end{array}
\right.
\]
by the exponential map (see also Step 3.\ in the proof of \cite[Proposition 6.2.8]{Cam22b} for the detail). 

Let $\mathring{\mathbb{X}}_{J, r} \coloneqq \bigcup_{r' < r} \mathbb{X}_{J, r}$ be the Stein space. 
For each $r \gg 0$, we fix a compact open subgroup $K_p (r) \subset K_p$ such that $\mathbb{X}_{J, r}$ and $\mathring{\mathbb{X}}_{J, r}$ (also $\mathcal{O} (\mathbb{X}_{J, r})$ and $\mathcal{O} (\mathring{\mathbb{X}}_{J, r})$) are stable by the (continuous) actions of left multiplications of $K_p (r)$. 
Let $\mathbb{X}'_{J, r}$ be the rigid space constructed from $\{ p^r \mathfrak{Y}_{s, \sigma}, p^r \}_{\sigma \in J \backslash \{ \tau \}, \ 1 \leq s \leq g_{\sigma}}$ as above. 
Note that $\dim \mathbb{X}'_{J, r} + 1 = \dim \mathbb{X}_{J, r}$ if $\tau \in J$ and otherwise $\mathbb{X}'_{J, r} = \mathbb{X}_{J, r}$.

By the construction and \cite[Example 2.3]{Sch10}, we can write 
\begin{equation}
\label{colim_bundle_infin}
A \widehat{\otimes}_{\mathbb{Q}_p} C^{J-\la} (\widetilde{\mathfrak{g}}, W) = \varinjlim_{r \to \infty} A \widehat{\otimes}_{\mathcal{O}(V)} \mathcal{O}(\mathbb{X}_{J, r}) \widehat{\otimes}_E W = \varinjlim_{r \to \infty} A \widehat{\otimes}_{\mathcal{O}(V)} \mathcal{O}(\mathring{\mathbb{X}}_{J, r}) \widehat{\otimes}_E W, 
\end{equation}
where the tensor product of the third term is projective tensor product of Fr\'echet spaces. 

Then we have a quasi-isomorphism 
\[
R\Gamma(\mathfrak{n}_{\mu}^0, A \widehat{\otimes}_{\mathcal{O}(V)} \mathcal{O}(\mathring{\mathbb{X}}_{J, r}) \widehat{\otimes}_E W) \cong A \widehat{\otimes}_{\mathcal{O}(V)} \mathcal{O}(\mathring{\mathbb{X}}'_{J, r} \widehat{\otimes}_E W)
\]
by the same argument of \cite[Lemma 6.2.9]{Cam22b}. 
The quasi-isomorphism and the equation \eqref{colim_bundle_infin} proves the lemma. 
\end{proof}

\begin{definition}
\label{def_la_sheaves}
\begin{enumerate}
\item Let $\mathcal{S}h_{K^p, \an}^{\tor}$ denote the analytic site of the perfectoid space $\mathcal{S}h_{K^p}^{\tor}$. 
We call an open subspace $U \subset \mathcal{S}h_{K^p}^{\tor}$ \emph{rational} if it is the pullback of a rational open subspace at finite level. 
We may similarly define an analytic site $D_{\infty, I', \an}$ of the adic space $D_{\infty, I'} \coloneqq D_{K_p, I'} \times_{\mathcal{S}h} \mathcal{S}h_{K^p}^{\tor}$ (This space is a perfectoid space such that $D_{\infty, I'} \sim \varprojlim_{K_p} D_{K_p, I'}$ by the last part of the proof of \cite[Proposition 6.1]{Lan22}) and its rational open subspaces. 
\item Write $\widehat{\mathcal{O}}_{\mathcal{S}h, \an}$ for the sheaf of topological rings restricting the sheaf $\widehat{\mathcal{O}}_{\mathcal{S}h}$ on the site $\mathcal{S}h_{\proket}$ to the site $\mathcal{S}h_{K^p, \an}^{\tor}$. 
We may similarly define sheaves $\widehat{\mathcal{I}}_{\mathcal{S}h, \an}$ on $\mathcal{S}h_{K^p, \an}^{\tor}$ and $\widehat{\mathcal{O}}_{I', \an}$ on $D_{\infty, I', \an}$. 
\item Let $\mathcal{O}_{\mathcal{S}h, W}^{J-\la} \subset \widehat{\mathcal{O}}_{\mathcal{S}h, W, \an} \coloneqq \widehat{\mathcal{O}}_{\mathcal{S}h, \an} \otimes_{\mathbb{Q}_p} W$ denote a sub-(pre)sheaf on $\mathcal{S}h_{K^p, \an}^{\tor}$ such that for each rational open subspace $U \subset \mathcal{S}h_{K^p}^{\tor}$, 
\[
\mathcal{O}_{\mathcal{S}h, W}^{J-\la} (U) = \varinjlim_{K_{p, U}} \widehat{\mathcal{O}}_{\mathcal{S}h, W, \an} (U)^{K_{p, U, J} -\an}
\]
where $K_{p, U} = \prod_{\sigma \in \{ z \} \cup \sqcup_{\wp_i |p} \Sigma_{\wp_i}} K_{p, U, \sigma}$ runs over open subgroups of $K_p$ (as a $p$-adic Lie subgroup) which stabilizes $U$. 
We may similarly define sub-(pre)sheaves $\mathcal{O}_{I', W}^{J-\la} \subset \widehat{\mathcal{O}}_{I', W, \an} \coloneqq \widehat{\mathcal{O}}_{I', \an} \otimes_{\mathbb{Q}_p} W$ on $D_{\infty, I', \an}$ and $\mathcal{I}_{\mathcal{S}h, W}^{J-\la} \subset \widehat{\mathcal{I}}_{\mathcal{S}h, W, \an} \coloneqq \widehat{\mathcal{I}}_{\mathcal{S}h, \an} \otimes_{\mathbb{Q}_p} W$ on $\mathcal{S}h_{K^p, \an}^{\tor}$. 
If $W$ is the trivial representation, we write $\mathcal{O}_{\mathcal{S}h}^{J-\la} = \mathcal{O}_{\mathcal{S}h, W}^{J-\la}$, $\mathcal{O}_{I'}^{J-\la} = \mathcal{O}_{I', W}^{J-\la}$, and $\mathcal{I}_{\mathcal{S}h}^{J-\la} = \mathcal{I}_{\mathcal{S}h, W}^{J-\la}$. 
\end{enumerate}
\end{definition}

Since taking $J$-locally analytic vectors is left exact, $\mathcal{O}_{\mathcal{S}h, W}^{J-\la}$, $\mathcal{O}_{I', W}^{J-\la}$, and $\mathcal{I}_{\mathcal{S}h, W}^{J-\la}$ satisfies the gluing condition on rational open subspaces (cf.\ \cite[Lemma 6.2.2]{Cam22b}). 
In the following, we let $\mathcal{O}_{\mathcal{S}h, W}^{J-\la}$, $\mathcal{O}_{I', W}^{J-\la}$ denote the associated sheaves. 

\begin{proposition}
\label{Jla_long_sequence}
For a sufficient small compact open subgroup $K_p \subset G(\mathbb{Q}_p)$ and each open subspace $U \subset \mathcal{S}h_{K_p K^p}^{\tor}$ in Condition \ref{U_V_conditions}, there are natural quasi-isomorphisms 
\[
R\Gamma_{\proket} (U, \iota_{K_p, I', \ast} \widehat{\mathcal{O}}_{I'} \widehat{\otimes}_{\mathbb{Q}_p} C^{J-\la} (\widetilde{K}_p, W)_{\proket}) \cong \mathcal{O}_{I'}^{J-\la} (\pi_{K_p}^{\tor, -1} (U) \cap D_{\infty, I'})
\]
and a long exact sequence 
\[
0 \to \mathcal{I}_{\mathcal{S}h, W}^{J-\la} (\pi_{K_p}^{\tor, -1} (U)) \to \mathcal{O}_{\mathcal{S}h, W}^{J-\la} (\pi_{K_p}^{\tor, -1} (U)) \to \cdots \to \mathcal{O}_{\mathcal{S}h, W}^{J-\la} (\pi_{K_p}^{\tor, -1} (U) \cap D_{\infty, I}) \to 0. 
\]
\end{proposition}

\begin{proof}
This follows by Lemma \ref{n_cohomology_induce}, Lemma \ref{J-la_space_ON_colom}, Lemma \ref{n0_cohomology_polydisc}, and the remaining part of the proof of \cite[Proposition 6.2.8]{Cam22b}. 
\end{proof}

\begin{proposition}
\label{derived_comparison}
There are the following natural $G(\mathbb{Q}_p)$-equivariant quasi-isomorphisms of derived solid $\mathbb{C}_p$-vector spaces 
\begin{enumerate}
\item $(R\Gamma_{\proket} (\mathcal{S}h_{K^p}^{\tor}, W_0) \widehat{\otimes}_{\mathbb{Z}_p}^L \mathbb{C}_p)^{RJ-\la} = R\Gamma_{\an} (\mathcal{S}h_{K^p}^{\tor}, \mathcal{O}_{\mathcal{S}h, W}^{J-\la})$, 
\item $(R\Gamma_{\proket} (\mathcal{S}h_{K^p}^{\tor}, j_! W_0) \widehat{\otimes}_{\mathbb{Z}_p}^L \mathbb{C}_p)^{RJ-\la} = R\Gamma_{\an} (\mathcal{S}h_{K^p}^{\tor}, \mathcal{I}_{\mathcal{S}h, W}^{J-\la})$, 
\end{enumerate}
where $R\Gamma_{\an}$ is the right derived functors of taking global sections in the analytic topology. 
\end{proposition}

\begin{proof}
This follows by Proposition \ref{Jla_long_sequence} and the proof of \cite[Theorem 6.2.6]{Cam22b}. 
\end{proof}

Write $H_{\an}^i \coloneqq H^i R\Gamma_{\an}$. 
\begin{corollary}
\label{Qpla_comparison}
For $i \geq 0$, there are natural $G(\mathbb{Q}_p)$-equivariant isomorphisms 
\begin{align*}
(\widetilde{H}^i (K^p, W) \widehat{\otimes}_{\mathbb{Q}_p} \mathbb{C}_p)^{\la} &\cong H_{\an}^i (\mathcal{S}h_{K^p}^{\tor}, \mathcal{O}_{\mathcal{S}h, W}^{\la}), \\ 
(\widetilde{H}_c^i (K^p, W) \widehat{\otimes}_{\mathbb{Q}_p} \mathbb{C}_p)^{\la} &\cong H_{\an}^i (\mathcal{S}h_{K^p}^{\tor}, \mathcal{I}_{\mathcal{S}h, W}^{\la}). 
\end{align*} 
\end{corollary}

\begin{proof}
This follows by Theorem \ref{admissible_LAacyclic}, Corollary \ref{infin_proket_comp_comparison}, and Proposition \ref{derived_comparison}, and the construction of $\widetilde{H}_?^i (K^p, W)$ for $? \in \{ \emptyset, c \}$. 
\end{proof}


\subsection{The case of $\widetilde{H}^1$ of unitary Shimura curves}

In this subsection, we assume Condition \ref{condition_comparison} (1) on $(G, X)$, $E$, and $J$. 

\begin{proposition}
\label{Jla_comparison_H1}
For any subset $J \in \{ z \} \cup \bigsqcup_{\wp_i | p} \Sigma_{\wp_i}$, 
there is a natural $G(\mathbb{Q}_p)$-equivariant isomorphism 
\[
(\widetilde{H}^1 (K^p, W) \widehat{\otimes}_{\mathbb{Q}_p} \mathbb{C}_p)^{J-\la} \cong H_{\an}^1 (\mathcal{S}h_{K^p}^{\tor}, \mathcal{O}_{\mathcal{S}h, W}^{J-\la}). 
\]
\end{proposition}

\begin{proof}
By Proposition \ref{derived_comparison} (1), it suffices to show an isomorphism 
\[
H^1 ((R\Gamma_{\proket} (\mathcal{S}h_{K^p}^{\tor}, W) \widehat{\otimes}_{\mathbb{Q}_p}^L \mathbb{C}_p)^{RJ-\la}) \cong (\widetilde{H}^1 (K^p, W) \widehat{\otimes}_{\mathbb{Q}_p} \mathbb{C}_p)^{J-\la}. 
\]
By spectral sequence, we may reduce to show 
\[
H^2 (\widetilde{H}^0 (K^p, W) \widehat{\otimes}_{\mathbb{Q}_p} \mathbb{C}_p)^{RJ-\la} = 0. 
\]
Since the representation $\widetilde{H}^0 (K^p, W)$ of $G(\mathbb{Q}_p)$ is admissible by \cite[Theorem 2.1.5]{Eme06b}, it suffices to show 
\begin{align*}
H_{\cont}^0 (\mathfrak{g}^J, H^2 (\widetilde{H}^0 (K^p, W) \widehat{\otimes}_{\mathbb{Q}_p} \mathbb{C}_p)^{R\la}) &= 0, \quad \text{and} \\ 
H_{\cont}^2 (\mathfrak{g}^J, H^1 (\widetilde{H}^0 (K^p, W) \widehat{\otimes}_{\mathbb{Q}_p} \mathbb{C}_p)^{R\la}) &= 0
\end{align*}
by Proposition \ref{spec_seq_RJlad}. 
Then the two vanishings of the cohomology follow by Theorem \ref{admissible_LAacyclic}. 
\end{proof}

For a subset $\mathscr{J} \subset \{ z \} \cup [1, s]$, let 
\[
G_{\mathscr{J}} \coloneqq \prod_{k \in \mathscr{J}} G_k, 
\]
where $G_z \coloneqq \mathbb{G}_{m, \mathbb{Q}_p}$ and $G_i \coloneqq \Res_{F_{\wp_i} / \mathbb{Q}_p} \GL_2 (F_{\wp_i})$ for $i \in [1, s]$. 
Let $\mathscr{J}' \coloneqq (\{ z \} \cup [1, s]) \backslash \mathscr{J}$. 
Let $\mathfrak{z} \coloneqq \Lie Z$, $\mathfrak{g}_k \coloneqq \Lie (G_k)$ for $k \in \{ z \} \cup [1, s]$, and $\mathfrak{g}_p^{\wp} \coloneqq \Lie (G(\mathbb{Q}_p^{\wp}))$ be the Lie algebras over $\mathbb{Q}_p$. 
We characterize $\widetilde{H}^1 (K^{\wp}, W)$ as the subspace of $\mathscr{Q}$-locally analytic vectors in $\widetilde{H}^1 (K^p, W)$ by the method of \cite{Eme06b}. 

\begin{proposition}
\label{H1comp_pela_colimit}
There is a natural $G(\mathbb{A}^{\infty})$-equivariant isomorphism 
\[
(\widetilde{H}^1 (K^p, W) \widehat{\otimes}_{\mathbb{Q}_p} \mathbb{C}_p)^{\wp-\la} \cong \varinjlim_{K_p^{\wp}} (\widetilde{H}^1 (K^p K_p^{\wp}, W) \widehat{\otimes}_{\mathbb{Q}_p} \mathbb{C}_p)^{\la}, 
\]
where $K_p^{\wp}$ runs over compact open subgroups of $G(\mathbb{Q}_p^{\wp})$. 
\end{proposition}

\begin{proof}
By the same argument of \cite[Corollary 2.2.18]{Eme06b}, we have a spectral sequence 
\[
E_2^{i, j} = \Ext_{\mathfrak{g}_p^{\wp}}^i (W_E^{\vee}, (\widetilde{H}^j (K^p, E) \widehat{\otimes}_{\mathbb{Q}_p} \mathbb{C}_p)^{\wp-\la}) \Longrightarrow \varinjlim_{K_p^{\wp}} (\widetilde{H}^{i+j} (K^p K_p^{\wp}, W) \widehat{\otimes}_{\mathbb{Q}_p} \mathbb{C}_p)^{\la}. 
\]
Then it suffices to show 
\[
\Ext_{\mathfrak{g}_p^{\wp}}^2 (W^{\vee}, (\widetilde{H}^0 (K^p, E) \widehat{\otimes}_{\mathbb{Q}_p} \mathbb{C}_p)^{\la}) = 0. 
\]

Recall that $\wp_1 = \wp$. 
Let $W_E$ be the induced (finite dimensional) representation of $G_E^c$ (and also $G_E$). 
We may assume that $W_E$ is a finite dimensional simple algebraic representation of $G_E$ because any finite dimensional algebraic representation of a connected reductive group over a field of characteristic $0$ is semisimple. 
By the classification of simple algebraic representations of reductive groups (see e.g., \cite{Jan07} and also \cite[\S 2.2]{Sch10}), we write 
\[
W_E = \bigotimes_{k \in \{ z \} \cup [1, s]} W_{k, E}, 
\]
where $W_{k, E}$ is a simple algebraic representation of $G_{k, E}$. 
Write $\mathfrak{z}_{k, E}$ for the center of $\mathfrak{g}_{k, E}$. 
Let $\mathfrak{z}_E \coloneqq \bigoplus_{k \in \{ z \} \cup [1, s]} \mathfrak{z}_{k, E}$ be the center of $\mathfrak{g}_E$ and $\mathfrak{s} \coloneqq \bigoplus_{k \in [1, s]} \mathfrak{s}\mathfrak{l}_{2, E}$. 
Since the $0$-th completed cohomology $\widetilde{H}^0 (K^p, E)$ is the $E$-vector space of the profinite set of connected component of infinite level unitary Shimura curves isomorphic to the tensor product of a finite dimensional trivial representation of $G(E)$ and $C_{\cont} (\mathbb{Z}_p^{\times}, E)$ where $G$ acts on the profinite set via the reduced norm, the subgroup $\prod_{k \in [1, s]} \SL_2 (E) \subset G(E)$ trivially acts on $\widetilde{H}^0 (K^p, E)$. 

Then by K\"unneth formula (cf.\ the proof of \cite[Proposition 4.3.1]{Eme06b}), for an integer $i$, we have an isomorphism 
\begin{align*}
\bigoplus_{a+b = i} (\Ext_{\mathfrak{z}}^a (E(\chi^{-1}), (\widetilde{H}^0 (K^p, E) \widehat{\otimes}_{\mathbb{Q}_p} \mathbb{C}_p)^{\la}) &\otimes_E \Ext_{\mathfrak{s}}^b (W_E^{\vee}, E)) \\  &\xrightarrow{\sim} \Ext_{\mathfrak{g}_E}^i (W_E^{\vee}, (\widetilde{H}^0 (K^p, E) \widehat{\otimes}_{\mathbb{Q}_p} \mathbb{C}_p)^{\la})
\end{align*}
where $\chi_E$ is the central character of the representation $W_E$. 
The module 
\[
\Ext_{\mathfrak{z}}^2 (E(\chi^{-1}), (\widetilde{H}^0 (K^p, E) \widehat{\otimes}_{\mathbb{Q}_p} \mathbb{C}_p)^{\la})
\]
vanishes by that $\mathfrak{z}$ is a direct sum of one dimensional Lie algebras and Hochschild-Serre spectral sequence. 

Then by K\"unneth formula again, we have 
\[
\bigoplus_{a_1 + \cdots + a_s = i} \bigotimes_{k \in [1, s]} \Ext_{\mathfrak{s}\mathfrak{l}_{2, E}}^{a_k} (W_{k, E}^{\vee}, E) \xrightarrow{\sim} \Ext_{\mathfrak{s}}^i (W_E^{\vee}, E). 
\]
Thus we may reduce to show for each $k \in [2, s]$ and $i = 1$ or $2$, 
\begin{equation}
\label{reduce_gl2_coh}
\Ext_{\mathfrak{g}\mathfrak{l}_{2, E}}^i (W_{k, E}^{\vee}, (\widetilde{H}^0 (K^p, E) \widehat{\otimes}_{\mathbb{Q}_p} \mathbb{C}_p)^{\la}) = 0. 
\end{equation}
The vanishings of \eqref{reduce_gl2_coh} follows by the similar argument of \cite[Proposition 4.3.1, Corollary 4.3.2]{Eme06b}. 
\end{proof}


\begin{corollary}
\label{H1comp_taula_colimit}
There is a natural $G(\mathbb{A}^{\infty})$-equivariant isomorphism 
\[
(\widetilde{H}^1 (K^p, W) \widehat{\otimes}_{\mathbb{Q}_p} \mathbb{C}_p)^{\tau-\la} \cong \varinjlim_{K_p^{\wp}} (\widetilde{H}^1 (K^p K_p^{\wp}, W) \widehat{\otimes}_{\mathbb{Q}_p} \mathbb{C}_p)^{\tau-\la}, 
\]
where $K_p^{\wp}$ runs over compact open subgroups of $G(\mathbb{Q}_p^{\wp})$. 
\end{corollary}

\begin{proof}
This follows by taking $\tau-\la$ of the isomorphism in Proposition \ref{H1comp_pela_colimit}
\end{proof}

\section{Sheaves of locally analytic sections}
\label{sec_la_comp_coh}

In this section, we will show that locally analytic vectors for the action of $G(\mathbb{Q}_p)$ (not $\GL_2 (F_{\wp})$) in completed cohomology of unitary Shimura curves comes from locally analytic completed cohomology of Siegel modular varieties. 
As a consequence, we give a description of completed cohomology of $K^{\wp}$-level in terms of cohomology of sheaves of locally analytic sections on $K^{\wp}$-level unitary Shimura curve. 
Let $(\GSp_{2g}, D_{\GSp_{2g}})$ be a Shimura data which gives a Siegel modular variety in Section \ref{ss_Shimura_Hodge}. 
For a compact open subgroup $K = K_p K^p \subset \GSp_{2g} (\mathbb{A}^{\infty})$, let $\mathcal{Y}_K$ be the associated Siegel modular variety with $(\GSp_{2g}, D_{\GSp_{2g}})$. 

\subsection{Surjectivity of the maps of locally analytic functions}
We recall the construction of Hodge-Tate period map in Proposition \ref{HT_period_maps}, in the case of Siegel modular varieties (cf.\ \cite[\S III.3]{Sch15}). 
Let $f \colon \mathcal{A}_K \to \mathcal{Y}_K$ be the universal abelian variety as for the moduli interpretation of $\mathcal{X}$ (cf.\ \cite[\S 1]{FC90}). 
The standard representation $\GSp_{2g, \mathbb{Q}_p} \to \GL(V_{\std}) \cong \GL_{2n, \mathbb{Q}_p}$ defines a $\mathbb{Q}_p$-local system $V_{\std, \proet}$ on the pro-\'etale site $\mathcal{Y}_{K, \proet}$. 
The local system is naturally isomorphic to $R^1 f_{\ast} \mathbb{Q}_p$. 
On the other hand, $V_{\dR}$ is naturally isomorphic to the relative de Rham cohomology of $\mathcal{A}_K / \mathcal{Y}_K$. 
This is a rank $2g$ vector bundle on $\mathcal{Y}_K$ equipped with a (decreasing) Hodge filtration $\Fil^{\bullet}$ and Gauss-Manin connection $\nabla$. 

Let $\mathcal{Y}_K^{\tor}$ be a toroidal compactification of $\mathcal{Y}_K$. 
Similarly, we have a $\mathbb{Q}_p$-local system $V_{\std, \proket}$ on $\mathcal{Y}_{K, \proket}^{\tor}$ and the corresponding vector bundle $V_{\dR, \log}$ on $\mathcal{Y}_K^{\tor}$ equipped with a logarithmic connection $\nabla_{\log}$. 
The graded components of $V_{\dR, \log}$ vanish except in degree $0$ or $1$, and we have an exact sequence 
\begin{equation}
\label{Ytor_HT_sequence}
0 \to \gr^0 (V_{\dR, \log}) \otimes_{\mathcal{O}_{\mathcal{Y}_K^{\tor}}} \widehat{\mathcal{O}}_{\mathcal{Y}_K^{\tor}} \to V_{\std, \proket} \otimes_{\widehat{\mathbb{Q}}_p} \widehat{\mathcal{O}}_{\mathcal{Y}_K^{\tor}} \to \gr^1 (V_{\dR, \log}) \otimes_{\mathcal{O}_{\mathcal{Y}_K^{\tor}}} \widehat{\mathcal{O}}_{\mathcal{Y}_K^{\tor}} (-1) \to 0
\end{equation}
by \cite[Proposition 7.9]{Sch13}. 

Let $\mathcal{Y}_{K^p}^{\tor} \sim \varprojlim_{K_p} \mathcal{Y}_{K_p K^p}^{\tor}$ be a perfectoid space constructed by Theorem \ref{perfd_Hodge_type}. 
Since $V_{\std, \proket}$ is trivialized on $\mathcal{Y}_{K^p}^{\tor}$, we have an exact sequence 
\begin{equation}
\label{YKptor_HT_sequence}
0 \to \gr^0 (V_{\dR, \log}) \otimes_{\mathcal{O}_{\mathcal{Y}^{\tor}}} \widehat{\mathcal{O}}_{\mathcal{Y}_{K^p}^{\tor}} \to \underline{V_{\std}} \otimes_{\widehat{\mathbb{Q}}_p} \widehat{\mathcal{O}}_{\mathcal{Y}_{K^p}^{\tor}} \to \gr^1 (V_{\dR, \log}) \otimes_{\mathcal{O}_{\mathcal{Y}^{\tor}}} \widehat{\mathcal{O}}_{\mathcal{Y}_{K^p}^{\tor}} (-1) \to 0
\end{equation}
of vector bundles from \eqref{Ytor_HT_sequence}. 
Fix a Hodge cocharacter $\widetilde{\mu} \colon \mathbb{G}_m \to \GSp_{2g, \mathbb{C}}$ defined over $\mathbb{Q}_p$ attached to an element $h \in X$ and let $\mathscr{F}\ell_{\widetilde{\mu}}$ be the associated rigid space with $\Fl_{\mathbb{C}_p} \coloneqq G/P_{\widetilde{\mu}} \times_{\Spec \mathbb{Q}_p} \Spec \mathbb{C}_p$. 
The position 
\[
\gr^0 (V_{\dR, \log}) \otimes_{\mathcal{O}_{\mathcal{Y}^{\tor}}} \widehat{\mathcal{O}}_{\mathcal{Y}_{K^p}^{\tor}} \subset \underline{V_{\std}} \otimes_{\widehat{\mathbb{Q}}_p} \widehat{\mathcal{O}}_{\mathcal{Y}_{K^p}^{\tor}}
\]
in the exact sequence \eqref{YKptor_HT_sequence} induces the $\GSp_{2g} (\mathbb{Q}_p)$-equivariant Hodge-Tate period map 
\[
\pi_{\HT}^{\tor} \colon \mathcal{Y}_{K^p}^{\tor} \to \mathscr{F}\ell_{\widetilde{\mu}}
\]
as constructed in Proposition \ref{HT_period_maps}. 

Fix a $\mathbb{Q}_p$-valued point $\infty \in \mathscr{F}\ell_{\widetilde{\mu}}$, a point $\widetilde{\infty} \in \mathcal{Y}_{K^p}^{\tor}$ which is mapped to $\infty$ by $\pi_{\HT}^{\tor}$, and a symplectic basis $\{ e_1, \ldots, e_{2g} \} \in V_{\std}$ such that $\{ e_1 \otimes 1, \ldots, e_g \otimes 1 \}$ generates $\gr^0 (V_{\dR, \log}) \otimes \{ \widetilde{\infty} \}$, where $\{ \widetilde{\infty} \} \hookrightarrow \mathcal{Y}_{K^p}^{\tor}$ is the closed immersion of perfectoid spaces. 
Taking global sections of \eqref{YKptor_HT_sequence}, we have an exact sequence 
\begin{align}
\label{YKptor_HT_sequence_H0}
0 \to H^0 (\mathcal{Y}_{K^p}^{\tor}, \gr^0 (V_{\dR, \log}) \otimes_{\widehat{\mathbb{Q}}_p} \widehat{\mathcal{O}}_{\mathcal{Y}^{\tor}}) 
&\to H^0 (\mathcal{Y}_{K^p}^{\tor}, \underline{V_{\std}} \otimes_{\widehat{\mathbb{Q}}_p} \widehat{\mathcal{O}}_{\mathcal{Y}^{\tor}}) \notag \\ 
&\to H^0 (\mathcal{Y}_{K^p}^{\tor}, \gr^1 (V_{\dR, \log}) \otimes_{\widehat{\mathbb{Q}}_p} \widehat{\mathcal{O}}_{\mathcal{Y}^{\tor}}) \to 0 
\end{align}
since $H^1 (\mathcal{Y}_{K^p}^{\tor}, \gr^0 (V_{\dR, \log}) \otimes_{\widehat{\mathbb{Q}}_p} \widehat{\mathcal{O}}_{\mathcal{Y}^{\tor}}) = 0$ by the almost purity theorem \cite[Theorem 5.4.3]{DLLZ19}. 
We also write $e_1, \ldots, e_{2g}$ for the images of $e_1 \otimes 1, \ldots, e_{2g} \otimes 1$ by the third map of \eqref{YKptor_HT_sequence_H0}. 

Let $(G, X)$ be as in Section \ref{ss_Shimura_Hodge} and fix a Hodge cocharacter $\widetilde{\mu} \colon \mathbb{G}_m \to G_{\mathbb{C}}$ defined over $\mathcal{F}$. 
Let $(G, X) \hookrightarrow (\GSp_{2g}, D_{\GSp_{2g}})$ be a closed embedding of Shimura data such that the Hodge cocharacters $\mu$ and $\widetilde{\mu}$ are compatible by the embedding. 
Let $K^p \subset G(\mathbb{A}^{\infty, p}) \subset G(\mathbb{A}^{\infty})$ be a neat compact open subgroup which satisfies the condition of Theorem \ref{perfd_Hodge_type}. 
Let $i \colon \mathscr{F}\ell_{\mu} \hookrightarrow \mathscr{F}\ell_{\widetilde{\mu}}$ be the closed immersion of the flag varieties induced by $\iota$. 

We recall the construction of perfectoid Shimura varieties of Hodge type $\mathcal{S}h_{K^p}^{\tor}$. 
Let $K'_p {K'}^p \subset \GSp_{2g} (\mathbb{A}^{\infty})$ be a sufficiently small compact open subgroup such that $K_p K^p = K'_p {K^p}' \cap G(\mathbb{A}^{\infty})$. 
Let $\mathcal{Y} \coloneqq \mathcal{Y}_{K'_p {K'}^p}^{\tor}$ and $U' \subset \mathcal{Y}$ be an open subspace such that there exist an open subspace $V' \subset \mathscr{F}\ell_{\widetilde{\mu}}$ where $(U', V')$ satisfies Condition \ref{U_V_conditions}, $U' \subset \mathfrak{B}$ where $\mathfrak{B}$ is as in Proposition \ref{HT_period_maps} and $V \coloneqq i^{-1} (V') \subset \mathscr{F}\ell_{\widetilde{\mu}}$ is non-empty. 
Let ${\pi}_{K'_p {K'}^p}^{\tor} \colon \mathcal{Y}_{{K'}^p}^{\tor} \to \mathcal{Y}_{K'_p {K'}^p}^{\tor}$ and ${\pi}_{K_p {K}^p}^{\tor} \colon \mathcal{S}h_{{K}^p}^{\tor} \to \mathcal{S}h_{K_p {K}^p}^{\tor}$ be the natural projections. 
For any open subgroup ${K^p}' \subset \GSp_{2g} (\mathbb{A}^{\infty, p})$ which satisfies the condition of Theorem \ref{perfd_Hodge_type}, the open subspace $\mathcal{Y}_{{K^p}'}^{\tor} (U') \coloneqq \pi_{K'_p {K'}^p}^{\tor, -1} (U')$ is affinoid perfectoid (cf.\ \cite[Theorem III.3.18]{Sch15}) and we write $\mathcal{Y}_{{K^p}'}^{\tor} (U') = \Spa (S_{{K^p}'}, S_{{K^p}'}^+)$. 
It follows that 
\begin{equation}
\label{perfd_Hodge_SKp}
\mathcal{Y}_{K^p}^{\tor} (U') \coloneqq \varprojlim_{K^p \subset {K^p}' \subset \GSp_{2g} (\mathbb{A}^{\infty})} \mathcal{Y}_{{K^p}'}^{\tor} (U') = \Spa(S_{K^p}, S_{K^p}^+), 
\end{equation}
where $S_{K^p}^+$ is the $p$-adic completion of $\varinjlim_{K^p \subset {K^p}'} S_{{K^p}'}^+$, is affinoid perfectoid. 

For any sufficient small compact open subgroup $K'_p {K^p}' \subset \GSp_{2g} (\mathbb{A}^{\infty})$ with $K_p K^p = K'_p {K^p}' \cap G(\mathbb{A}^{\infty})$, a closed immersion 
\[
(\mathcal{Y}_{K^p}^{\tor} \times_{\mathcal{Y}_{K'_p {K^p}'}^{\tor}} \mathcal{S}h_{K_p K^p}^{\tor}) (U') \coloneqq \mathcal{Y}_{K^p}^{\tor} (U') \times_{\mathcal{Y}_{K'_p {K^p}'}^{\tor}} \mathcal{S}h_{K_p K^p}^{\tor} \subset \mathcal{Y}_{K^p}^{\tor} (U')
\]
obtained by the base change of the closed immersion $i_{K'_p {K^p}'} \colon \mathcal{S}h_{K_p K^p}^{\tor} \hookrightarrow \mathcal{Y}_{K'_p {K^p}'}^{\tor}$ is defined by some ideal $I \subset S_{K^p}$ by \cite[Proposition 1.15]{Del71}, and is affinoid perfectoid by \cite[Lemma II.2.2]{Sch15}. We write 
\begin{equation}
\label{perfd_Hodge_RKpKp}
(\mathcal{Y}_{K^p}^{\tor} \times_{\mathcal{Y}_{K'_p {K^p}'}^{\tor}} \mathcal{S}h_{K_p K^p}^{\tor}) (U') = \Spa (R_{K^p, K'_p {K^p}'}, R_{K^p, K'_p {K^p}'}^+). 
\end{equation}
Then we can write 
\begin{equation}
\label{const_perfd_Hodge_type}
\mathcal{S}h_{K^p}^{\tor} (U') \coloneqq \varprojlim_{K'_p, K^p \subset {K^p}'} (\mathcal{Y}_{K^p}^{\tor} \times_{\mathcal{Y}_{K'_p {K^p}'}^{\tor}} \mathcal{S}h_{K_p K^p}^{\tor}) (U') = \Spa(R_{K^p}, R_{K^p}^+) \subset \mathcal{S}h_{K^p}^{\tor}
\end{equation}
as an affinoid open perfectoid subspace, where $R_{K^p}^+$ is the $p$-adic completion of $\varinjlim_{K'_p, K^p \subset {K^p}'} R_{K^p, K'_p {K^p}'}$. 
The open subspaces $\mathcal{S}h_{K^p}^{\tor} (U')$ covers the whole $\mathcal{S}h_{K^p}^{\tor}$. 
Then we have the closed immersions of perfectoid spaces $i_{\infty} \colon \mathcal{S}h_{K^p}^{\tor} \hookrightarrow \mathcal{Y}_{K^p}^{\tor}$ and $i_{\infty, U'} \colon \mathcal{S}h_{K^p}^{\tor} (U') \hookrightarrow \mathcal{Y}_{K^p}^{\tor} (U')$. 
We note that the construction of $\mathcal{S}h_{K^p}^{\tor}$ is compatible with the method of \cite{Cam22b} by the last part of the proof of \cite[Proposition 6.1]{Lan22}.

Let $i \colon \mathscr{F}\ell_{\mu} \hookrightarrow \mathscr{F}\ell_{\widetilde{\mu}}$ be the closed immersion of the flag varieties induced by $\iota$, $\mu$, $\widetilde{\mu}$. 
The map $i$ is compatible with the Hodge-Tate period maps by Proposition \ref{HT_period_maps}. 

For each compact open subgroup $K'_p {K^p}' \subset \GSp_{2g} (\mathbb{A}^{\infty})$ which gives a closed immersion $i_{K'_p {K^p}'} \colon \mathcal{S}h_{K_p K^p}^{\tor} \hookrightarrow \mathcal{Y}_{K'_p {K^p}'}^{\tor}$, let $j_{K'_p {K^p}'} \colon \mathcal{Y}_{K'_p {K^p}'}^{\tor} \backslash \mathcal{S}h_{K_p K^p}^{\tor} \hookrightarrow \mathcal{Y}_{K'_p {K^p}'}^{\tor}$ be the open immersion. 
We have an exact sequence of Kummer \'etale sheaves 
\[
0 \to j_{K'_p {K^p}', !} \mathbb{Z}/p^s \to \mathbb{Z}/p^s \to i_{K'_p {K^p}', \ast} \mathbb{Z}/p^s \to 0
\]
for each integer $s > 0$ by \cite[Lemma 2.1.5]{LLZ19} and that the cone decompositions which define the toroidal compactifications are compatibly taken. 
Taking a derived limit, we have an exact sequence of Kummer \'etale sheaves 
\[
0 \to j_{K'_p {K^p}', !} \mathbb{Z}_p \to \mathbb{Z}_p \to i_{K'_p {K^p}', \ast} \mathbb{Z}_p \to 0.
\]
Tensoring with $\mathcal{O}_{\mathcal{Y}_{K'_p {K^p}'}^{\tor}} / p^s$, we have an exact sequence of Kummer \'etale sheaves 
\[
0 \to j_{K'_p {K^p}', !} \mathbb{Z}_p \otimes_{\mathbb{Z}_p} \mathcal{O}_{\mathcal{Y}_{K'_p {K^p}'}^{\tor}} / p^s \to \mathcal{O}_{\mathcal{Y}_{K'_p {K^p}'}^{\tor}} / p^s \to i_{K'_p {K^p}', \ast} \mathcal{O}_{\mathcal{S}h_{K_p K^p}^{\tor}} / p^s \to 0
\]
by \cite[Lemma 4.5.7]{DLLZ19}. 
Taking a derived limit and inverting $p$, we have an exact sequence of pro-Kummer \'etale sheaves 
\begin{equation}
\label{exact_seq_perfd_immersion}
0 \to \widehat{\mathcal{I}}_{\mathcal{Y}_{K'_p {K^p}'}^{\tor}} \to \widehat{\mathcal{O}}_{\mathcal{Y}_{K'_p {K^p}'}^{\tor}} \to i_{K'_p {K^p}', \ast} \widehat{\mathcal{O}}_{\mathcal{S}h_{K_p K^p}^{\tor}} \to 0
\end{equation}
where $\widehat{\mathcal{I}}_{\mathcal{Y}_{K'_p {K^p}'}^{\tor}} \coloneqq \varprojlim_s j_{K'_p {K^p}', !} \mathbb{Z}_p \otimes_{\mathbb{Z}_p} \mathcal{O}_{\mathcal{Y}_{K'_p {K^p}'}^{\tor}} / p^s$. 
We let 
\[
\mathcal{O}_{\mathcal{Y}}^{\GSp_{2g} (\mathbb{Q}_p)-\la} (\mathcal{Y}_{K^p}^{\tor} (U')) \coloneqq \widehat{\mathcal{O}}_{\mathcal{Y}}(\mathcal{Y}_{K^p}^{\tor} (U'))^{\GSp_{2g} (\mathbb{Q}_p)-\la} . 
\]


\begin{proposition}
\label{perfd_la_surjectivity}
The $G(\mathbb{Q}_p)$-equivariant map 
\[
\mathcal{O}_{\mathcal{Y}}^{\GSp_{2g} (\mathbb{Q}_p)-\la} (\mathcal{Y}_{K^p}^{\tor} (U')) \to \mathcal{O}_{\mathcal{S}h}^{G(\mathbb{Q}_p)-\la} (\mathcal{S}h_{K^p}^{\tor} (U'))
\]
induced from taking locally analytic vectors of the completed colimit of third maps of the exact sequence \eqref{exact_seq_perfd_immersion} is surjective. 
The map remains to be surjective when the source is restricted to $\mathcal{O}_{\mathcal{Y}}^{\GSp_{2g} (\mathbb{Q}_p)-\la} (\mathcal{Y}_{K'^p}^{\tor} (U'))$. 

Moreover, for a Lie subalgebra $\mathfrak{h} \subset \mathfrak{g}$, if an element in the source is $\mathfrak{h}$-smooth, its image in the target is $\mathfrak{h}$-smooth. 
\end{proposition}

\begin{proof}
We rewrite $\mathcal{Y}$ as $\mathcal{Y} = \mathcal{Y}_{K''_p {K^p}''}$ and 
let $K'_p {K^p}' \subset K''_p {K^p}'' \subset \GSp_{2g} (\mathbb{A}^{\infty})$ be a compact open subgroup as in \eqref{const_perfd_Hodge_type}. 
Note that for the open subspace $V \coloneqq i^{-1} (V') \subset \mathscr{F}\ell_{\mu, I}$, the pair $(i_{K''_p {K^p}''}^{-1} (U'), V)$ also satisfies Condition \ref{U_V_conditions}. 
We use the argument of the proof of part (2) of \cite[Proposition 6.2.8]{Cam22b}. 

Let $U'_{K'_p {K^p}'} \coloneqq U' \times_\mathcal{Y} \mathcal{Y}_{K'_p {K^p}'}$. 
We have an injection and an isomorphism 
\begin{align}
\label{H0_translation}
H^0 (U'_{K'_p {K^p}'}, \widehat{\mathcal{O}}_{\mathcal{Y}_{K'_p {K^p}'}} \widehat{\otimes}_{\mathbb{Q}_p} C^{\GSp_{2g} (\mathbb{Q}_p)-\la} (\widetilde{K}'_p, E)_{\proket})  &\hookrightarrow \mathcal{O}_{\mathcal{Y}}^{\GSp_{2g} (\mathbb{Q}_p)-\la} (\mathcal{Y}_{K^p}^{\tor} (U')), \notag \\ 
H^0 (i_{K''_p {K^p}''}^{-1} (U'), \widehat{\mathcal{O}}_{\mathcal{S}h} \widehat{\otimes}_{\mathbb{Q}_p} C^{G(\mathbb{Q}_p)-\la} (\widetilde{K}_p, E)_{\proket}) &\cong \mathcal{O}_{\mathcal{S}h}^{G(\mathbb{Q}_p)-\la} (\mathcal{S}h_{K^p}^{\tor} (U')), 
\end{align}
where $\widetilde{K}_p$ is a certain quotient of $K_p$ and $\widetilde{K}'_p = K'_p$, by Proposition \ref{Jla_long_sequence} and the compatibility of $i$, $i_{\infty}$ with the Hodge-Tate period maps. 

Retake toric charts of $i_{K'_p {K^p}'}^{-1} (U')$ and $U'$ such that the following diagram commutes:
\begin{equation}
\label{compatible_toric_charts}
\vcenter{
\xymatrix{i_{K'_p {K^p}'}^{-1} (U'_{K'_p {K^p}'}) \ar[r] \ar[d] & \mathbb{T}^{e'} \times \mathbb{D}^{d' - e'} \ar[d] \\ U'_{K'_p {K^p}'} \ar[r] & \mathbb{T}^e \times \mathbb{D}^{d - e},}
}
\end{equation}
where $d' \geq d$, $e' \geq e$, and the right vertical map is the product of the natural projections to first $e$ (resp.\ $d-e$) coordinates. 
We may assume to take toric charts compatibly with choices of $K'_p {K^p}'$. 
We note ote that $i_{K''_p {K^p}''}^{-1} (U') = i_{K'_p {K^p}'}^{-1} (U'_{K'_p {K^p}'})$

Then we can construct affinoid perfectoid covers $\widetilde{U}'_{K'_p {K^p}', p^{\infty}}$ and $\widetilde{i_{K''_p {K^p}''}^{-1} (U')}_{p^{\infty}}$ of $U'_{K'_p {K^p}'}$ and $i_{K''_p {K^p}''}^{-1} (U')$ with Galois groups $\widetilde{K}'_p \times \Gamma'_p$ and $\widetilde{K}_p \times \Gamma_p$ respectively as in Section \ref{ss_comparison_theorem}, which are compatible with the map $i_{K'_p {K^p}'}$. 

We have an injection and an isomorphism  
\begin{align}
\label{colimit_trans}
\varinjlim_{r \to \infty} (\widehat{\mathcal{O}}_{\mathcal{Y}} (\widetilde{U}'_{K'_p {K^p}', p^{\infty}}) \widehat{\otimes}_{\mathcal{O}(V')} \mathcal{O}(\mathbb{X}'_r))^{\widetilde{K}'_p -\sm, \Gamma'_p} &\hookrightarrow \mathcal{O}_{\mathcal{Y}}^{\GSp_{2g} (\mathbb{Q}_p)-\la} (\mathcal{Y}_{K^p}^{\tor} (U')), \notag \\ 
\varinjlim_{r \to \infty} (\widehat{\mathcal{O}}_{\mathcal{S}h} (\widetilde{i_{K''_p {K^p}''}^{-1} (U')}_{p^{\infty}}) \widehat{\otimes}_{\mathcal{O}(V)} \mathcal{O}(\mathbb{X}_r))^{\widetilde{K}_p -\sm, \Gamma_p} &\cong \mathcal{O}_{\mathcal{S}h}^{G(\mathbb{Q}_p)-\la} (\mathcal{S}h_{K^p}^{\tor} (U')) , 
\end{align}
where $\{ \mathbb{X}'_r \}_{r \gg 0}$ and $\{ \mathbb{X}_r \}_{r \gg 0}$ are shrinking system of closed polydiscs over $V'$ and $V$ respectively (We consider the case of $\mathbb{Q}_p$-coefficient and $\mathbb{Q}_p$-locally analytic subspaces in this proposition, nevertheless the same proof of Lemma \ref{n0_cohomology_polydisc} works) by Lemma \ref{n_cohomology_induce}, Lemma \ref{J-la_space_ON_colom}, Lemma \ref{n0_cohomology_polydisc}, and the isomorphisms \eqref{H0_translation}. 

By the detailed construction of $\{ \mathbb{X}'_r \}_{r \gg 0}$ and $\{ \mathbb{X}_r \}_{r \gg 0}$ in the proof of Lemma \ref{n0_cohomology_polydisc}, we have a natural surjection $\mathcal{O}(\mathbb{X}'_r) \twoheadrightarrow \mathcal{O}(\mathbb{X}_r)$, so we obtain a surjection for $r \gg 0$ 
\[
\left( \varinjlim_{K'_p {K^p}'} \widehat{\mathcal{O}}_{\mathcal{Y}} (\widetilde{U}'_{K'_p {K^p}', p^{\infty}}) \right)^{\wedge} \widehat{\otimes}_{\mathcal{O}(V')} \mathcal{O}(\mathbb{X}'_r) \twoheadrightarrow \widehat{\mathcal{O}}_{\mathcal{S}h} (\widetilde{i_{K''_p {K^p}''}^{-1} (U')}_{p^{\infty}}) \widehat{\otimes}_{\mathcal{O}(V)} \mathcal{O}(\mathbb{X}_r), 
\]
where $(-)^{\wedge}$ denotes taking $p$-adic completions. 
Then it suffices to show that the restriction for any (or, some) $K'_p {K^p}'$, 
\[
\widehat{\mathcal{O}}_{\mathcal{Y}} (\widetilde{U}'_{K'_p {K^p}', p^{\infty}}) \widehat{\otimes}_{\mathcal{O}(V')} \mathcal{O}(\mathbb{X}'_r) \to \widehat{\mathcal{O}}_{\mathcal{S}h} (\widetilde{i_{K''_p {K^p}''}^{-1} (U')}_{p^{\infty}}) \widehat{\otimes}_{\mathcal{O}(V)} \mathcal{O}(\mathbb{X}_r)
\]
amounts to be surjective after taking locally analytic vectors. 

Let $\widetilde{A}_{K'_p {K^p}', \infty} \coloneqq \widehat{\mathcal{O}}_{\mathcal{Y}} (\widetilde{U}'_{K'_p {K^p}', p^{\infty}})$ and $\widetilde{B}_{\infty} \coloneqq \widehat{\mathcal{O}}_{\mathcal{S}h} (\widetilde{i_{K''_p {K^p}''}^{-1} (U')}_{p^{\infty}})$. 
Write $A_{K'_p {K^p}', H', n} \coloneqq \widetilde{A}_{K'_p {K^p}', \infty}^{H' \times p^n \Gamma'_p}$ and $B_{H, n} \coloneqq \widetilde{B}_{\infty}^{H \times p^n \Gamma_p}$ for a positive integer $n$ and compact open subgroups $H' \subset K'_p$ (or $\widetilde{K}'_p$) and $H \subset K_p$ (or $\widetilde{K}_p$). 
Let 
\begin{align*}
\mathcal{F}'_{K'_p {K^p}', r} &\coloneqq \widehat{\mathcal{O}}_{\mathcal{Y}} (\widetilde{U}'_{K'_p {K^p}', p^{\infty}}) \widehat{\otimes}_{\mathcal{O}(V')} \mathcal{O}(\mathbb{X}'_r), \\ 
\mathcal{F}''_r &\coloneqq \widehat{\mathcal{O}}_{\mathcal{S}h} (\widetilde{i_{K''_p {K^p}''}^{-1} (U')}_{p^{\infty}}) \widehat{\otimes}_{\mathcal{O}(V')} \mathcal{O}(\mathbb{X}'_r), \\ 
\mathcal{F}_r &\coloneqq \widehat{\mathcal{O}}_{\mathcal{S}h} (\widetilde{i_{K''_p {K^p}''}^{-1} (U')}_{p^{\infty}}) \widehat{\otimes}_{\mathcal{O}(V)} \mathcal{O}(\mathbb{X}_r). 
\end{align*}
For each $r \gg 0$, there exist open compact subgroups $\widetilde{K}'_p (r) \subset \widetilde{K}'_p$ and $\widetilde{K}_p (r) \coloneqq \widetilde{K}'_p (r) \cap G(\mathbb{Q}_p) \subset \widetilde{K}_p$ such that $\mathcal{F}'_{K'_p {K^p}', r}$ is $\widetilde{K}'_p (r)$-stable and $\mathcal{F}_r$ is $\widetilde{K}_p (r)$-stable. 

We can take $\mathbb{Q}_p$-ON bases $\{ v'_{k'} \}_{k'}$ of $\mathcal{O}(\mathbb{X}'_r)$, and $\{ v_k \}_k$ of $\mathcal{O}(\mathbb{X}_r)$ which contains $\{ v'_{k'} \}_{k'}$ as a subset. 
Then by \cite[Theorem 2.4.3, Remark 2.5.2]{Cam22a}, there is a compact open subgroup $H'(r) \times p^n \Gamma'_p \subset \widetilde{K}'_p (r) \times \Gamma'_p$ such that $\mathcal{F}'_{K'_p {K^p}', r}$ admits a unique decompletion (in \cite[Definition 2.4.1]{Cam22a}, it is called Sen module) to an ON Banach $\widetilde{K}'_p (r) \times \Gamma'_p$-representation $S_{H'(r), n} (\mathcal{F}'_{K'_p {K^p}', r})$ over $A_{K'_p {K^p}', H'(r), n}$ and $\mathcal{F}''_r$ and $\mathcal{F}_r$ admit unique decompletions to an ON Banach $H(r) \times p^n \Gamma_p$-representations $S_{H(r), n} (\mathcal{F}''_r)$ and $S_{H(r), n} (\mathcal{F}_r)$ over $B_{H(r), n}$ where $H(r) \coloneqq H'(r) \cap \widetilde{K}_p (r)$. 
Note that the maps $A_{K'_p {K^p}', H'(r), n} \to B_{H(r), n}$ 
is surjective by the construction. 

The uniqueness of the decompletions (\cite[Theorem 2.4.3]{Cam22a}) yields 
\begin{align*}
S_{H(r), n} (\mathcal{F}''_r) &\cong S_{H'(r), n} (\mathcal{F}'_{K'_p {K^p}', r}) \widehat{\otimes}_{A_{H'(r), n}} B_{H(r), n} \\ 
&= S_{H'(r), n} (\mathcal{F}'_{K'_p {K^p}', r}) \otimes_{A_{H'(r), n}} B_{H(r), n}
\end{align*}
and then we have a surjection $S_{H'(r), n} (\mathcal{F}'_r) \twoheadrightarrow S_{H(r), n} (\mathcal{F}''_r)$. 
We also have a surjection $S_{H(r), n} (\mathcal{F}''_r) \twoheadrightarrow S_{H(r), n} (\mathcal{F}_r)$ since the decompletions preserves surjectivity by \cite[Proposition 2.5.4]{Cam22a}. 
So we have a surjection 
\begin{equation}
\label{decompletion_surjection}
S_{H'(r), n} (\mathcal{F}'_{K'_p {K^p}', r}) \twoheadrightarrow S_{H(r), n} (\mathcal{F}_r)
\end{equation}
as the composite of the two surjections. 
Since the Sen operators of $\mathcal{F}'$ and $\mathcal{F}$ are trivial by the construction, the action of $\widetilde{K}'_p (r) \times \Gamma'_p$ (resp.\ $\widetilde{K}_p (r) \times \Gamma_p$) on $S_{H'(r), n} (\mathcal{F}'_{K'_p {K^p}', r})$ (resp.\ $S_{H(r), n} (\mathcal{F}_r)$) is smooth. 
Write $\Gamma'_p = \Gamma_p \times \Gamma''_p$ by the compatibility of toric charts (\ref{compatible_toric_charts}), and extend the $\Gamma_p$-action on $S_{H(r), n} (\mathcal{F}_r)$ to a $\Gamma'_p$-action by the trivial action of $\Gamma''_p$. 
The group $\Gamma'_p$ acts on the surjective map \eqref{decompletion_surjection} via a finite quotient. 
Thus, taking $\Gamma'_p$ invariant of \eqref{decompletion_surjection} and the colimit along $r \to \infty$, we have a surjection 
\begin{align*}
\varinjlim_{r \to \infty} S_{H'(r), n} (\mathcal{F}'_{K'_p {K^p}', r})^{\Gamma'_p} \twoheadrightarrow \varinjlim_{r \to \infty} S_{H(r), n} (\mathcal{F}_r)^{\Gamma_p}
\end{align*}
By the definition of decompletions and the equations \eqref{colimit_trans}, the colimit along $n \to \infty$ and $H(r) \to 1$ of the surjection is the same as that induced from the third map of the exact sequence \eqref{exact_seq_perfd_immersion}. 
The last claim follows from the construction. 
\end{proof}

Let $\mathfrak{g}$, $\mathfrak{p}_{\mu}$, $\mathfrak{p}_{\widetilde{\mu}}$, $\mathfrak{n}_{\mu}$, $\mathfrak{n}_{\widetilde{\mu}}$ is the Lie algebras as defined in Section \ref{ss_HT_period_maps}. 
Let $\widetilde{\mathfrak{p}}_{\mu} \coloneqq \Lie (P_{\mu}^c)$. 

Since the locus of points $x \in \mathscr{F}\ell_{\mu}$ such that $\mathfrak{n}_{\mu, x}^0 \cap \widetilde{\mathfrak{p}}_{\mu, x}^0 \neq 0$ is non-empty Zariski closed in $\mathscr{F}\ell_{\mu}$, we need to argue with the localized Sen modules for proving the following proposition by the method of the proof of Proposition \ref{perfd_la_surjectivity}.

\begin{proposition}
\label{perfd_la_Bsm_surjectivity}
The $G(\mathbb{Q}_p)$-equivariant map 
\begin{equation}
\label{perfd_la_Bsm_surjectivity_map}
\mathcal{O}_{\mathcal{Y}}^{\GSp_{2g} (\mathbb{Q}_p)-\la} (\mathcal{Y}_{K^p}^{\tor} (U'))^{\mathfrak{p}_{\widetilde{\mu}}} \to \mathcal{O}_{\mathcal{S}h}^{G(\mathbb{Q}_p)-\la} (\mathcal{S}h_{K^p}^{\tor} (U'))^{\mathfrak{p}_{\mu}}
\end{equation}
induced from the third map of the exact sequence \eqref{exact_seq_perfd_immersion} is surjective. 
Moreover, for a Lie subalgebra $\mathfrak{h} \subset \mathfrak{g}$, if an element in the source is $\mathfrak{h}$-smooth, its image in the target is $\mathfrak{h}$-smooth. 
\end{proposition}

\begin{proof}
The proof is similar to Lemma \ref{n0_cohomology_polydisc} and Proposition \ref{perfd_la_surjectivity}. 
We just explain different points of the proof. 

Instead of the whole space of locally analytic functions $C^{\la} (\widetilde{\mathfrak{g}\mathfrak{s}\mathfrak{p}_{2g}}, \mathbb{Q}_p)$, we use the subspace of $\mathfrak{p}_{\widetilde{\mu}}$-smooth functions  (of $\mathbb{Q}_p$-coefficients) 
\[
C^{\la, \mathfrak{p}_{\widetilde{\mu}}-\sm} (\widetilde{\mathfrak{g}\mathfrak{s}\mathfrak{p}_{2g}}, \mathbb{Q}_p) \coloneqq \varinjlim_{K'_p \subset \GSp_{2g} (\mathbb{Q}_p)} C^{\la, P_{\widetilde{\mu}} \cap \widetilde{K}'_p -\sm} (\widetilde{K}'_p, \mathbb{Q}_p) \subset C^{\la} (\widetilde{\mathfrak{g}\mathfrak{s}\mathfrak{p}_{2g}}, \mathbb{Q}_p). 
\]
We let (LHS) and (RHS) be the source and the target of the map \eqref{perfd_la_Bsm_surjectivity_map} respectively and set 
\begin{align*}
\mathcal{H}' &\coloneqq H^0 (\mathfrak{n}_{\widetilde{\mu}}^0, \mathcal{O}_{\mathcal{Y}}^{\GSp_{2g} (\mathbb{Q}_p)-\la} (\mathcal{Y}_{K^p}^{\tor} (U')) \widehat{\otimes}_{\mathbb{Q}_p} C^{\la, \mathfrak{p}_{\widetilde{\mu}}-\sm} (\widetilde{\mathfrak{g}\mathfrak{s}\mathfrak{p}_{2g}}, \mathbb{Q}_p)), \\ 
\mathcal{H} &\coloneqq H^0 (\mathfrak{n}_{\mu}^0, \mathcal{O}_{\mathcal{S}h}^{\GSp_{2g} (\mathbb{Q}_p)-\la} (\mathcal{S}h_{K^p}^{\tor} (U')) \widehat{\otimes}_{\mathbb{Q}_p} C^{\la, \mathfrak{p}_{\mu}-\sm} (\widetilde{\mathfrak{g}\mathfrak{s}\mathfrak{p}_{2g}}, \mathbb{Q}_p)). 
\end{align*}
By Lemma \ref{n_cohomology_induce}, we have isomorphisms $(\mathrm{LHS}) \cong {\mathcal{H}'}^{\widetilde{K}'_p \times \Gamma'_{p}}$ and $(\mathrm{RHS}) \cong \mathcal{H}^{\widetilde{K}_p \times \Gamma_{p}}$. 
Since (LHS) and (RHS) are $\mathcal{O}_{\mathscr{F}\ell_{\widetilde{\mu}}} (V')$-module, it suffices to show for each classical point $x \in V'$, the localized map $(\mathrm{LHS})_x \to (\mathrm{RHS})_x$ is surjective. 

Note that localizations commute with taking (Lie algebra) cohomology. 
We obtain the analog of Lemma \ref{n0_cohomology_polydisc} for the localizations of the $\mathfrak{p}_{\widetilde{\mu}}$-(or $\mathfrak{p}_{\mu}$-) invariant subspaces in the same way and it follows that the localized maps $\mathcal{H}'_x \to \mathcal{H}_x$ are surjective. 
So we can follow the argument of Proposition \ref{perfd_la_surjectivity}. 
\end{proof}

\begin{corollary}
\label{pmu_annihilate}
Suppose that $\mathfrak{n}_{\mu, x}^0 \cap \widetilde{\mathfrak{p}}_{\mu, x}^0 = 0$ for each point $x \in V \coloneqq i^{-1} (V')$. 
If an element $f \in \mathcal{O}_{\mathcal{S}h}^{G(\mathbb{Q}_p)-\la} (\mathcal{S}h_{K^p}^{\tor} (U'))$ is annihilated by $\mathfrak{p}_{\mu}$, then $f$ is annihilated by $\mathfrak{g}$, that is, $f$ is fixed by an open subgroup of $G(\mathbb{Q}_p)$. 
\end{corollary}

\begin{proof}
First we consider the case of Siegel modular varieties. 
In this case, the argument of \cite[Remark 4.3.4]{Pan20} works. 

By using the action of the element $\begin{pmatrix} p & 0 \\ 0 & 1 \end{pmatrix} \in \GSp_{2g} (\mathbb{Q}_p)$, where each block is of size $n \times n$, we may assume that $\pi_{K'_p}^{\tor, -1} (U')$ is a rational open subset of some $\epsilon$-neighborhood of the anticanonical locus $\mathcal{X}_{\Gamma(p^{\infty})}^{\ast} (\epsilon)_a$ in \cite[\S III.2]{Sch15} for some $\epsilon \in (0, 1/2)$. 
The element $f$ is annihilated by an open congruence subgroup $\Gamma'_0$ of the Siegel parabolic subgroup of $\GSp_{2g} (\mathbb{Q}_p)$ by the assumption. 
We may assume $\Gamma_0 (p^m) \subset K'_p$ and let $U'_{\Gamma'_0}$ be the preimage of $U'_{K'_p {K'}^p}$ in $\mathcal{X}_{\Gamma'_0}^{\ast} (\epsilon)_a$. 
By \cite[Corollary III.2.23]{Sch15}, there is a Tate's normalized trace from $\mathcal{O}_{\mathcal{X}_{\Gamma'_0}^{\ast} (\epsilon)_a} (U'_{\Gamma'_0})$ to sections of finite levels. 
Then $f$ is fixed by some open subgroup of $\GSp_{2g} (\mathbb{Q}_p)$ by the proof of \cite[Th\'eor\`eme 3.2]{BC16} or \cite[Lemma 3.2.5]{Pan20}. 

In the case of general Shimura varieties of Hodge type, the claim follows from the Siegel case and Proposition \ref{perfd_la_Bsm_surjectivity}. 
\end{proof}

\begin{remark}
Corollary \ref{pmu_annihilate} can be also obtained by geometric Sen theory of Shimura varieties. 
By \cite[Theorem 5.2.5]{Cam22b}, the image of the geometric Sen operator is generated by $\mathfrak{n}_{\mu}^0$. Then by \cite[Corollary 3.4.6]{Cam22a}, the action of $\mathfrak{n}_{\mu}^0$ on $\mathcal{O}_{\mathcal{S}h}^{G(\mathbb{Q}_p)-\la}$ vanishes. 
So for each $(U, V)$ in our settings, $\mathfrak{p}_{\mu} \otimes_{\mathbb{Q}_p} \mathcal{O}_{\mathscr{F}\ell_{\mu}}$ and $\mathfrak{n}_{\mu}^0$ generates $\mathfrak{g} \otimes_{\mathbb{Q}_p} \mathcal{O}_{\mathscr{F}\ell_{\mu}}$, and thus the corollary follows. 
\end{remark}

\begin{proposition}
\label{la_section_tate}
For each element $s \in \mathcal{O}_{\mathcal{Y}}^{\GSp_{2g} (\mathbb{Q}_p)-\la} (\mathcal{Y}_{K^p}^{\tor} (U'))$, there exists a sufficiently small compact open subgroup $K'_p \subset K_p$ and $e'_1, \ldots, e'_l \in \mathcal{O}_{\mathcal{Y}}^{\GSp_{2g} (\mathbb{Q}_p)-\la} (\mathcal{Y}_{K^p}^{\tor} (U'))$, where $l$ is the dimension $g^2 + \frac{g(g+1)}{2} + 1$ of the Siegel parabolic subgroup of $\GSp_{2g}$ 
for each $i \in [1, l]$, such that 
\[
s = \sum_{k_1, \ldots, k_l \in \mathbb{Z}_{\geq 0}} c(k_1, \ldots, k_l) \prod_{i = 1}^l {e'}_i^{k_i}
\]
with $c(k_1, \ldots, k_l) \in H^0 (U', \mathcal{O}_{\mathcal{Y}_{K'_p K^p}^{\tor}})$ and $c(k_1, \ldots, k_l) \to 0$ as $\sum_{i=1}^l k_i \to \infty$. 
\end{proposition}

\begin{proof}
We may assume $\mathfrak{n}_{\widetilde{\mu}}^0 \cap \widetilde{\mathfrak{p}}_{\widetilde{\mu}}^0 = 0$ on $U'$ by the $\GSp_{2g} (\mathbb{Q}_p)$-action on $\mathscr{F}\ell_{\widetilde{\mu}}$ and shrinking $U'$ if necessary. 
We use the method of \cite[\S 4]{Pil22} and \cite[Theorem 3.4.5]{QS25} and the notation in the proof of Proposition \ref{perfd_la_surjectivity}. 

Let $\mathcal{G}_n$ for $n \in \mathbb{Z}_{\geq 0}$ be Lie subgroups constructed as in the beginning of Section \ref{J-la_vectors}. 
Write $\mathcal{P}_n \coloneqq P_{\widetilde{\mu}} \cap \mathcal{G}_n$. 


By the assumption, we can see the module $H^0 (\mathfrak{n}_{\widetilde{\mu}}^0, \mathcal{O}(\mathbb{X}'_r))$ as a $\widetilde{K}'_p (r)$-module for each $r \gg 0$. 
By the construction of $\mathbb{X}'_r$, we may assume that there is a natural isomorphism 
\begin{equation}
\label{Pn_Xr_colimit}
\varinjlim_n C^{\an} (\mathcal{P}_n, \mathbb{Q}_p) \cong \varinjlim_r \mathcal{O}(\mathbb{X}'_r). 
\end{equation}
Let $l \coloneqq g^2 + \frac{g(g+1)}{2} + 1$. 
We can choose functions $e_1, \ldots, e_l \in C^{\an} (P_{\widetilde{\mu}}, \mathbb{Q}_p)$ such that $C^{\an} (P_{\widetilde{\mu}}, \mathbb{Q}_p) = \mathbb{Q}_p \langle e_1, \ldots, e_l \rangle$ and $\{ e_i \}_{1 \leq i \leq l}$ generates a finite dimensional representation of $P_{\widetilde{\mu}}$ as in the proof of \cite[Theorem 3.4.5]{QS25}. 

Let $\widetilde{A}_{\infty} \coloneqq \widehat{\mathcal{O}}_{\mathcal{Y}} (\widetilde{U}'_{p^{\infty}})$ and $A_{H', n} \coloneqq \widetilde{A}_{\infty}^{H' \times p^n \Gamma'_p}$ for each open subgroup $H'(r) \times p^n \Gamma'_p \subset \widetilde{K}'_p \times \Gamma'_p$. 
Let $\mathcal{F}'_r \coloneqq \widehat{\mathcal{O}}_{\mathcal{Y}}(\widetilde{U}'_{p^{\infty}}) \widehat{\otimes}_{\mathcal{O}(V')} \mathcal{O}(\mathbb{X}'_r)$ and $S_{H'(r), n} (\mathcal{F}'_r)$ be the unique decompletion of $\mathcal{F}'_r$ for a sufficiently small open subgroup $H'(r) \times p^n \Gamma'_p \subset \widetilde{K}'_p \times \Gamma'_p$. 
Note that we have 
\begin{equation}
\label{decompletion_undo}
\mathcal{F}'_r = S_{H'(r), n} (\mathcal{F}'_r) \widehat{\otimes}_{A_{H'(r), n}} \widetilde{A}_{p^\infty}. 
\end{equation}

For $i \in [1, l]$, the image of $e_i$ by the composite of the natural map $C^{\an} (P_{\widetilde{\mu}}, \mathbb{Q}_p) \to \varinjlim_n C^{\an} (\mathcal{P}_n, \mathbb{Q}_p)$ and the isomorphism \eqref{Pn_Xr_colimit} is in the image of the natural injection $\mathcal{O}(\mathbb{X}'_r) \hookrightarrow \varinjlim_r \mathcal{O}(\mathbb{X}'_r)$ by the construction. 
Write $e_{i, r}$ for the corresponding element in $\mathcal{O}(\mathbb{X}'_r)$ to $e_i$. 

Write $1 \otimes e_{i, r} = \sum_j s_{i, j, r} \otimes f_{i, j, r}$ where $s_{i, j, r} \in S_{H'(r), n} (\mathcal{F}'_r)$ and $f_{i, j, r} \in \widetilde{A}_{p^\infty}$ by the equation \eqref{decompletion_undo}. 
Since $\{ e_{i, r} \}_{1 \leq i \leq l}$ generates a finite dimensional representation of $H'(r) \times p^n \Gamma'_p$, the summation $\sum_j s_{i, j, r} \otimes f_{i, j, r}$ is of finite term. 

Choose $f'_{i, j, r} \in \varinjlim_{H', m} A_{H', m}$ which is sufficiently close to $\widetilde{A}_{p^{\infty}}$. 
Then the element $e'_{i, r} \coloneqq \sum_j s_{i, j, r} \otimes f'_{i, j, r}$ is sufficiently closed to $\sum_j s_{i, j, r} \otimes f_{i, j, r}$. 
Thus we have 
\[
\mathcal{F}'_r = \widetilde{A}_{p^{\infty}} \langle e_1, \ldots, e_l \rangle =  \widetilde{A}_{p^{\infty}} \langle e'_1, \ldots, e'_l \rangle
\]
Taking $H'(r) \times p^n \Gamma'_p$-invariant, we have 
\begin{equation}
\label{Tatealg_SHrn}
A_{H'(r), n} \langle e'_1, \ldots, e'_l \rangle \cong S_{H'(r), n} (\mathcal{F}'_r)
\end{equation}
by the construction of $e'_1, \ldots, e'_l$. 

Since the Sen operator of $\mathcal{F}'_r$ is trivial, we have 
\[
{\mathcal{F}'_r}^{H', \Lie (\Gamma'_p)-\sm} = \varinjlim_m S_{H', m} (\mathcal{F}'_r) 
\]
for each open subgroup $H' \subset \widetilde{K}'_p$. 
Taking $\widetilde{K}'_p \times \Gamma'_p$-invariant, 
we have an identification 
\begin{equation}
\label{Ola_Sen_colimit}
\mathcal{O}_{\mathcal{Y}}^{\GSp_{2g} (\mathbb{Q}_p)-\la} (\mathcal{Y}_{K^p}^{\tor} (U')) = 
\varinjlim_r {\mathcal{F}'_r}^{\widetilde{K}'_p \times \Gamma'_p} = \varinjlim_m S_{H', m} ({\mathcal{F}'_r})^{\widetilde{K}'_p \times \Gamma'_p}. 
\end{equation}
The proposition follows from \eqref{Tatealg_SHrn} and \eqref{Ola_Sen_colimit}. 
\end{proof}

\subsection{Locally analyticity of perfectoid Shimura varieties}

In this subsection, we introduce an ad hoc definition of Faltings extensions of Shimura varieties of Hodge type and locally analyticiy of their perfectoid covers. 
We will show that any Shimura varieties of Hodge type admits a cover by open affinoid subspaces over which the perfectoid cover is locally analytic. 
This is a generalization of \cite[\S 3.5, \S 3.6, \S 4.2]{Pan20}. 

We use the notation in the previous section. 
Let $\widetilde{B} \coloneqq \widehat{\mathcal{O}}_{\mathcal{S}h} (\widetilde{U})$ and $\widetilde{B}_{\infty} \coloneqq \widehat{\mathcal{O}}_{\mathcal{S}h} (\widetilde{U}_{p^{\infty}})$. 
Fix an isomorphism $\Gamma_p \cong \mathbb{Z}_p^d$. 
We consider the following unipotent representation of $\Gamma_p$ on $W = \mathbb{Q}_p^{d+1}$:
\[
\Gamma_p \cong \mathbb{Z}_p^d \to \GL_{d+1} (\mathbb{Q}_p), \quad \gamma \mapsto A_\gamma = (a_{\gamma, 1}, \ldots, a_{\gamma, d}) \mapsto \begin{pmatrix} 1 & A_{\gamma} \\ 0 & 1 \\ \end{pmatrix}. 
\]
The representation defines the exact sequence 
\[
0 \to \mathbb{Q}_p \to W \to \mathbb{Q}_p^d \to 0. 
\]
Tensoring the exact sequence with $\widetilde{B}_{\infty}$ and take $\Gamma$-invariants with respect to the diagonal actions, we have an exact sequence 
\begin{equation}
\label{gen_Falt_ext}
0 \to \widetilde{B} \to (\widetilde{B}_{\infty} \otimes_{\mathbb{Q}_p} W)^{\Gamma_p} \to \widetilde{B}^d \to 0. 
\end{equation}

For $i \in [1, d]$, let $\gamma_i \in \Gamma_p$ be the corresponding element to the $i$-th fundamental vector $e_i \in \mathbb{Z}_p^d$. 
By the similar argument of \cite[Proposition 3.5.3]{Pan20}, we have 

\begin{proposition}
\label{la_covering_condition}
The following conditions are equivalent:
\begin{enumerate}
\item The sequence \eqref{gen_Falt_ext} remains exact after taking $\mathcal{G}_0$-analytic vectors for some uniform pro-$p$ subgroup $\mathcal{G}_0$ of $G(\mathbb{Q}_p)$.
\item The sequence \eqref{gen_Falt_ext} remains exact after taking $G(\mathbb{Q}_p)$-locally analytic vectors of each term. 
\item There exist $G(\mathbb{Q}_p)$-locally analytic vectors $z_1, \ldots, z_d \in \widetilde{B}_{\infty}$ such that $\gamma_i (z_i) = z_i - 1$ and $\gamma_i (z_j) = z_j$ if $i \neq j$. 
\end{enumerate}
\end{proposition}

We say that $\widetilde{U}$ is a \emph{locally analytic covering} of $U$ if one of the conditions in Proposition \ref{la_covering_condition} holds. 

\begin{proposition}
\label{lacover_laacyclic}
The topological $G$-module $\widetilde{B}$ is $\mathfrak{LA}$-acyclic if and only if $\widetilde{U}$ is a locally analytic covering of $U$. 
\end{proposition}

\begin{proof}
The ``only if'' part is clear. 
The argument of the ``if'' part is similar to \cite[\S 3.6]{Pan20}. 
First we have a natural isomorphism 
\begin{equation}
\label{Gn_Gammap_coh}
H_{\cont}^i (\mathcal{G}_n, \widetilde{B} \widehat{\otimes}_{\mathbb{Q}_p} C^{\an} (\mathcal{G}_n, \mathbb{Q}_p)) \cong H_{\cont}^i (\Gamma_p, (\widetilde{B}_{\infty})^{\mathcal{G}_n -\an})
\end{equation}
for $i, n \geq 0$ by the almost purity theorem (cf.\ \cite[Lemma 3.6.4]{Pan20}) and we have an isomorphism (cf.\ \cite[\S 3.6.5]{Pan20}) 
\[
R^i \mathfrak{LA} (\widetilde{B}) \cong \varinjlim_{n, k} H_{\cont}^i (\Gamma_p, (\widetilde{B}_{\infty})^{(\mathcal{G}_n \times p^k \Gamma_p)-\an}). 
\]


We can take locally analytic vectors $z_1, \ldots, z_d \in \widetilde{B}_{\infty}$ as in Proposition \ref{la_covering_condition} (3) by the assumption. 
Note that each $z_j$ is in $\widetilde{B}_{\infty, j}$. 
Let $\widetilde{B}_{\infty}^{\la, \fin} \coloneqq \varinjlim_{n, k} \widetilde{B}_{\infty}^{(\mathcal{G}_n \times p^k \Gamma_p) -\an}$. 
For each $j \in [1, d]$, consider the complex 
\[
C_j \coloneqq (\widetilde{B}_{\infty}^{\la, \fin} \xrightarrow{\gamma_j - 1} \widetilde{B}_{\infty}^{\la, \fin}), 
\]
where the first term is of degree $0$. 
Then each $C_j$ has trivial $H^1$ by the argument of \cite[\S 3.6.7]{Pan20}. 

The cohomology group \eqref{Gn_Gammap_coh} can be computed using the tensor product complex $\bigotimes_{j=1}^d C_j$, which is the Koszul complex attached to the $\Gamma_p$-action on $\widetilde{B}_{\infty}^{\la, \fin}$. 
Then its vanishing for $i \geq 1$ follows from each $H^1 (C_j) = 0$ and degenerating of K\"unneth spectral sequence. 
\end{proof}

\begin{proposition}
\label{perfd_Shimura_lacover}
The pro-Kummer \'etale cover $\widetilde{U}$ is a locally analytic covering of $U$. 
\end{proposition}

\begin{proof}
Let $f \colon U \to \mathbb{T}^{e} \times \mathbb{D}^{d-e}$ denote the fixed toric chart of $U$ and $T_i$ be the $i$-th coordinate of $\mathbb{T}^{e} \times \mathbb{D}^{d-e}$ for $i \in [1, d]$. 
By \cite[Proposition 2.3.15, Corollary 2.4.2]{DLLZ18}, identifying $\widetilde{B}^d$ with $\widetilde{B} \otimes_B \Omega_{U/\mathbb{C}_p}^1 (\log)$ sending $e_i$ to $f^{\ast} (\frac{dT_i}{T_i})$ for $i \in [1, e]$ and $e_i$ to $d\log T_i$ for $i \in [e+1, d]$, the exact sequence \eqref{gen_Falt_ext} is naturally isomorphic to the section on $\widetilde{U}$ (which is also exact by the almost purity theorem) of the exact sequence 
\begin{equation}
\label{Falt_ext_Shimura}
0 \to \widehat{\mathcal{O}}_{\mathcal{S}h} \to \gr^1 \mathcal{O}\mathbb{B}_{\dR, \log, \mathcal{S}h} (-1) \xrightarrow{\overline{\nabla}_{\log}} \widehat{\mathcal{O}}_{\mathcal{S}h} (-1) \otimes_{\mathcal{O}_{\mathcal{S}h}} \Omega_{\mathcal{S}h}^1 (\log) \to 0, 
\end{equation}
where $\overline{\nabla}_{\log}$ is induced from the graded quotient of log derivation $\nabla_{\log} \colon \mathcal{O}\mathbb{B}_{\dR, \mathcal{S}h} \to \mathcal{O}\mathbb{B}_{\dR, \mathcal{S}h} \otimes_{\mathcal{O}_{\mathcal{S}h}} \Omega_{\mathcal{S}h}^1 (\log)$ by Griffiths transversality.

Let $G\text{-QCoh}(\Fl_{\mu, \mathbb{C}_p})$ be the category of quasi-coherent sheaves on $\Fl_{\mu, \mathbb{C}_p}$ and $G\text{-}\mathcal{O}_{\mathscr{F}\ell_{\mu}}\text{-}\text{mod}$ be the category of $G$-equivariant $\mathcal{O}_{\mathscr{F}\ell_{\mu}}$-modules. 
Let $\pi \colon G_{\mathbb{C}_p} \to \Fl_{\mu, \mathbb{C}_p}$ be the natural quotient map and the functor (cf.\ \cite[Proposition 3.1.1]{Cam22b}) 
\[
\mathcal{W} \colon \Rep_{\mathbb{C}_p} (P_{\mu, \mathbb{C}_p}) \to G\text{-QCoh}(\Fl_{\mu, \mathbb{C}_p}), \quad W \mapsto (\mathcal{V} \mapsto (\mathcal{O}(\pi^{-1} (\mathcal{V})) \otimes_{\mathbb{C}_p} W)^{P_{\mu, \mathbb{C}_p}}), 
\]
where $P_{\mu, \mathbb{C}_p}$ acts on $\mathcal{O}(\pi^{-1}(\mathcal{V}))$ as the right regular representation and $\mathcal{V} \subset \Fl_{\mu, \mathbb{C}_p}$ is an open subset. 
The algebraic group $G_{\mathbb{Q}_p}$ acts on each $\mathcal{W}(W)$ as the left regular representation on (the translations of) $\mathcal{O}(\pi^{-1}(\mathcal{V}))$. 
We write $\mathcal{W}^{\ad}$ for the composite of $\mathcal{W}$ with the natural functor $G\text{-QCoh}(\Fl_{\mu, \mathbb{C}_p}) \to G\text{-}\mathcal{O}_{\mathscr{F}\ell_{\mu}}\text{-}\text{mod}$. 

By \cite[Theorem 5.1.4]{Cam22b}, the exact sequence \eqref{Falt_ext_Shimura} is naturally isomorphic to 
\begin{equation}
\label{auto_Falt_ext}
0 \to \widehat{\mathcal{O}}_{\mathcal{S}h} \to \pi_{\HT}^{\tor, \ast} (\mathcal{W}^{\ad}(\mathcal{O}(N_{\mu}^c)^{\leq 1})) \to \pi_{\HT}^{\tor, \ast} (\mathfrak{n}_{\mu}^{c, 0, \vee}) \to 0, 
\end{equation}
where the notation $\mathcal{O}(N_{\mu}^c)^{\leq 1}$ is given in \cite[\S 3.3]{Cam22b}, in particular, it is a finite dimensional algebraic representation of $P_{\mu, \mathbb{C}_p}$, and the third map is obtained from taking $(\pi_{\HT}^{\tor, \ast} \circ \mathcal{W}^{\ad}) (-)$ on the quotient map $\mathcal{O}(N_{\mu}^c)^{\leq 1} \to \gr_1 (\mathcal{O}(N_{\mu}^c))$  

Recall that an open subset $V \subset \mathscr{F}\ell_{\mu}$ can be taken such that Condition \ref{U_V_conditions} is satisfied. 
We may assume that $V$ is contained in the subspace of $\mathscr{F}\ell_{\mu}$ which comes from an affine open subspace $\mathcal{V} \subset \Fl_{\mu, \mathbb{C}_p}$. 
Here $\mathcal{O}(\pi^{-1} (\mathcal{V})) \otimes_{\mathbb{C}_p} \mathcal{O}(N_{\mu}^c)^{\leq 1}$, so $\mathcal{W}(\mathcal{O}(N_{\mu}^c)^{\leq 1})(\mathcal{V}) = (\mathcal{O}(\pi^{-1} (\mathcal{V})) \otimes_{\mathbb{C}_p} \mathcal{O}(N_{\mu}^c)^{\leq 1})^{P_{\mu, \mathbb{C}_p}}$ is an algebraic (and so analytic) representation of an open subgroup of $G(\mathbb{Q}_p)$. 
The module $\mathcal{W}(\mathcal{O}(N_{\mu}^c)^{\leq 1})(\mathcal{V})$ is also finite projective over $\mathcal{O}(\mathcal{V})$ by that the map $\pi \colon G_{\mathbb{C}_p} \to \Fl_{\mu, \mathbb{C}_p}$ is faithfully flat and descent of finite projective modules via the map $\pi$. 

Then we have equalities
\begin{align}
\label{tensor_decomplete}
\pi_{\HT}^{\tor, \ast} (\mathcal{W}^{\ad}(\mathcal{O}(N_{\mu}^c)^{\leq 1}))(\widetilde{U}) 
&= \mathcal{W}(\mathcal{O}(N_{\mu}^c)^{\leq 1})(\mathcal{V}) \widehat{\otimes}_{\mathcal{O}(\mathcal{V})} \mathcal{O}(\widetilde{U}) \\ 
&= \mathcal{W}(\mathcal{O}(N_{\mu}^c)^{\leq 1})(\mathcal{V}) \otimes_{\mathcal{O}(\mathcal{V})} \mathcal{O}(\widetilde{U}). \notag
\end{align}
Since the functor $\mathcal{W}$ is exact (cf.\ \cite[Proposition 3.1.1]{Cam22b}), the map of algebraic representations 
\begin{equation}
\label{surj_ev_V}
\mathcal{W}(\mathcal{O}(N_{\mu}^c)^{\leq 1})(\mathcal{V}) \to \mathfrak{n}_{\mu}^{c, 0, \vee}(\mathcal{V})
\end{equation}
obtained from taking $\mathcal{W}$ on the quotient map $\mathcal{O}(N_{\mu}^c)^{\leq 1} \to \gr_1 (\mathcal{O}(N_{\mu}^c))$ 
is surjective. 
So the exact sequence gained by taking sections on $\widetilde{U}$ of \eqref{auto_Falt_ext} remains exact by taking locally analytic vectors since tensoring the surjective map \eqref{surj_ev_V} with $\mathcal{O}(\widetilde{U})^{\la}$ remains to be surjective. 
Thus $\widetilde{U}$ is a locally analytic covering of $U$. 
\end{proof}

\subsection{Proof of the main theorem}

We prove the main theorem using the results obtained in the section. 
For a sufficiently small compact open subgroup $K^p \subset G(\mathbb{A}^{\infty, p})$, let $\mathcal{O}_{K^p} \coloneqq \pi_{\HT, \ast}^{\tor} (\widehat{\mathcal{O}}_{\mathcal{S}h}|_{\mathcal{S}h_{K^p}})$ and $\mathcal{I}_{K^p} \coloneqq \pi_{\HT, \ast}^{\tor} (\widehat{\mathcal{I}}_{\mathcal{S}h}|_{\mathcal{S}h_{K^p}})$ be sheaves of topological algebras on $\mathscr{F}\ell_{\mu}$ in analytic topology. 
Let $\mathcal{O}_{K^p}^{\la} \subset \mathcal{O}_{K^p}$ and $\mathcal{I}_{K^p}^{\la} \subset \mathcal{I}_{K^p}$ be subsheaves of locally analytic sections as in Definition \ref{def_la_sheaves}. 

\begin{theorem}
\label{main_theorem_la}
For any $i \geq 0$ and a compact open subgroup $K^p \subset G(\mathbb{A}^{\infty, p})$, there are natural $G(\mathbb{Q}_p)$-equivariant isomorphisms 
\begin{align*}
(\widetilde{H}^i (K^p, \mathbb{Q}_p) \widehat{\otimes}_{\mathbb{Q}_p} \mathbb{C}_p)^{\la} &\cong H^i (\mathscr{F}\ell_{\mu}, \mathcal{O}_{K^p}^{\la}), \\ 
(\widetilde{H}_c^i (K^p, \mathbb{Q}_p) \widehat{\otimes}_{\mathbb{Q}_p} \mathbb{C}_p)^{\la} &\cong H^i (\mathscr{F}\ell_{\mu}, \mathcal{I}_{K^p}^{\la})
\end{align*}
\end{theorem}

\begin{proof}
This follows from Proposition \ref{la_section_tate}, Proposition \ref{lacover_laacyclic}, Proposition \ref{perfd_Shimura_lacover}, \cite[Theorem IV.2.1]{Sch15}, and the argument of the proof of \cite[Theorem 4.4.6]{Pan20}. 
\end{proof}

\subsection{The case of unitary Shimura curves}

In the last subsection, we observe $\tau$-locally analytic vectors in completed cohomology of unitary Shimura curves using the method which has been developed until the previous subsection. 

Let $(G, X)$ be a Shimura data which gives a unitary Shimura curve as in Section \ref{ss_Shimura_curves}. 
Since unitary Shimura curves are compact and we want to focus on the unitary Shimura curves, write $U \coloneqq i_{K'_p {K^p}'}^{-1} (U')$, $\mathcal{S}h_{K^p} = \mathcal{S}h_{K^p}^{\tor}$, and $\widetilde{U} \coloneqq \mathcal{S}h_{K^p}^{\tor} (U')$. 
We now explicitly write down the sections 
\[
\mathcal{O}_{\mathcal{S}h, E}^{\tau-\la} (\mathcal{S}h_{K^p} (U)) = H^0 (\mathcal{S}h_{K^p} (U), \mathcal{O}_{\mathcal{S}h, E}^{\tau-\la}). 
\]
by the method of expansions along Borel subalgebras as in \cite[\S 4.3]{Pan20}. 
Recall the exact sequence of vector bundles on $\mathcal{S}h = \mathcal{S}h_K = \mathcal{S}h_{K_p K^p}$ 
\begin{equation}
\label{level_K_curve_sequence}
0 \to (\varepsilon_{\ast} \Omega_{\mathcal{A} / (\Sh_K)_{\tau, E}}^1)_{\wp, \tau}^{-, 1, \rig} \to (R^1 \varepsilon_{\ast} \Omega_{\mathcal{A} / (\Sh_K)_{\tau, E}}^{\bullet})_{\wp, \tau}^{-, 1, \rig} \to (R^1 \varepsilon_{\ast} \mathcal{O}_{\mathcal{A}})_{\wp, \tau}^{-, 1, \rig} \to 0, 
\end{equation}
which is the rigidified version of the exact sequence in Proposition \ref{Sh_curve_HT_filt_-1} (2). 
Taking the colimit along $K_p \to 1$ and the completion, we have an exact sequence of vector bundles on $\mathcal{S}h_{K^p, \proet}$ 
\begin{equation}
\label{infinite_curve_sequence}
0 \to \mathcal{V}^0 \to \underline{V_{\wp, \tau}^{-, 1}} \otimes_{\widehat{\mathbb{Q}}_p} \widehat{\mathcal{O}}_{\mathcal{S}h_{K^p}} \to \mathcal{V}^1 \to 0, 
\end{equation}
which is naturally embedded in the exact sequence \eqref{YKptor_HT_sequence}. 
Note that $V_{\wp, \tau}^{-, 1}$ is isomorphic to the standard representation of $\GL_2 (F_{\wp})$ (and fix such an isomorphism), and $\mathcal{V}^0$ and $\mathcal{V}^1$ are vector bundles defined over $\mathcal{S}h = \mathcal{S}h_{K_p K^p}$ and of rank $1$. 

Let $\{ e_1, e_2 \}$ denote the basis of the standard representation $V_{\wp, \tau}^{-, 1}$ of $\GL_2 (F_{\wp})$. 
Taking global sections of \eqref{infinite_curve_sequence} on $\mathcal{S}h_{K^p}$, we have an exact sequence 
\begin{equation}
\label{infinite_curve_sequence_H0}
0 \to H^0 (\mathcal{S}h_{K^p}, \mathcal{V}^0) \to H^0 (\mathcal{S}h_{K^p}, \underline{V_{\wp, \tau}^{-, 1}} \otimes_{\widehat{\mathbb{Q}}_p} \widehat{\mathcal{O}}_{\mathcal{S}h_{K^p}, E}) \to H^0 (\mathcal{S}h_{K^p}, \mathcal{V}^1) \to 0 
\end{equation}
by the almost purity theorem \cite[Theorem 5.4.3]{DLLZ19}. 
Fix a generator $v \in H^0 (\mathcal{S}h_{K^p}, \mathcal{V}^1)$ and the induced isomorphism $H^0 (\mathcal{S}h_{K^p}, \mathcal{O}_{\mathcal{S}h_{K^p}, E}) \cong H^0 (\mathcal{S}h_{K^p}, \mathcal{V}^1)$ by $v$. 
We equip with the induced norm in $H^0 (\mathcal{S}h_{K^p}, \mathcal{V}^1)$ from the Banach algebra $H^0 (\mathcal{S}h_{K^p}, \mathcal{O}_{\mathcal{S}h_{K^p}, E})$ by the above isomorphism. 
Recall that the flag variety $\mathscr{F}\ell_{\mu}$ associated with a Hodge cocharacter attached to $(G, X)$ is given by the filtration defined by the exact sequence \eqref{infinite_curve_sequence_H0} as observed in Section \ref{ss_flag_varieties}. 
Fix a $\mathbb{C}_p$-point $\infty \in \mathcal{S}h_{K^p}$ such that the point $\pi_{\HT} (\infty) \subset \mathscr{F}\ell_{\mu}$ corresponds to the filtration $\langle e_1 \rangle \subset V_{\wp, \tau}^{-, 1}$. 

We also write $e_1, e_2$ for the image of $e_1 \otimes 1$, $e_2 \otimes 1$ by the third map of \eqref{infinite_curve_sequence_H0}. 
Since the Hodge-Tate period map $\pi_{\HT} \colon \mathcal{S}h_{K^p} \to \mathscr{F}\ell_{\mu}$ is $\GL_2 (F_{\wp})$-equivariant, we may assume that $e_1$ generates $H^0 (\mathcal{S}h_{K^p}, \mathcal{V}^1)$ over $\mathcal{S}h_{K^p} (U)$ by the action of some element in $\GL_2 (F_{\wp})$. 
Let $x \coloneqq \frac{e_2}{e_1} \in \widehat{\mathcal{O}}_{\mathcal{S}h} (\mathcal{S}h_{K^p} (U))$. 


By the decomposition of $G_{\mathbb{Q}_p}$ (\ref{curve_G_decomposition}) and the argument of connected components of Shimura varieties, 
\[
\widetilde{H}^0 (K^p, E) \xrightarrow{\sim} C(\mathcal{O}_{\wp}^{\times}, E)^{\oplus | I_K |} \otimes_E C^{\wp}
\]
where $I_K$ be the ideal class group of $\mathcal{O}_K$ and $C^{\wp}$ is an admissible representation of $G$, on which $\GL_2 (F_{\wp})$ acts trivially (see also \cite[\S 4.2]{Eme06b} for te case of the modular curve). 
Let $t \in H^0 (\mathcal{S}h_{K^p}, \widehat{\mathcal{O}}_{\mathcal{S}h, E}) = \widetilde{H}^0 (K^p, E) \widehat{\otimes}_{\mathbb{Q}_p} \mathbb{C}_p$ be an element given by $\omega_{\tau}^{\oplus | I_K |} \otimes 1 \in \widetilde{H}^0 (K^p, E)$ where $\omega_{\tau} \coloneqq \tau |_{\mathcal{O}_{\wp}^{\times}} \colon \mathcal{O}_{\wp}^{\times} \hookrightarrow E$ is the restriction of the embedding $\tau \in \Sigma_{\wp}$. 
Note that $e_1$, $e_2$, $x$, $t$ are $\tau$-locally analytic sections by the construction. 


Recall that $U$ is taken to be in $\mathfrak{B}$ in Proposition \ref{HT_period_maps}. 
Fix a uniform pro-$p$ group $\mathcal{G}_0 \subset K_p$ and let $\mathcal{G}_n \subset \mathcal{G}_0$ be a subgroup defined in Section \ref{J-la_vectors}. 
For each integer $n \geq 0$, let $U_{\mathcal{G}_n} \subset \mathcal{S}h_{\mathcal{G}_n K^p}$ be the preimage of $U$. 
Let $r(-1) = 0$ and we can find 
\begin{enumerate}
\item An integer $r(n) > r(n-1)$. 
\item An element $x_n \in H^0 (U_{\mathcal{G}_{r(n)}}, \mathcal{O}_{\mathcal{S}h_{\mathcal{G}_{r(n)} K^p}})$ such that $||x - x_n||_{\mathcal{G}_{r(n)}} \leq p^{-n}$ in $H^0 (\widetilde{U}, \mathcal{O}_{\mathcal{S}h_{K^p}})$. 
\item An element $t_n \in H^0 (\mathcal{S}h_{\mathcal{G}_{r(n)} K^p}, \mathcal{O}_{\mathcal{S}h_{\mathcal{G}_{r(n)} K^p}, E})$ such that $||t - t_n||_{\mathcal{G}_{r(n)}} \leq p^{-n}$ in $H^0 (\mathcal{S}h_{K^p}, \mathcal{O}_{\mathcal{S}h_{K^p}})$. 
\item An element $e_{1, n} \in H^0 (U_{\mathcal{G}_{r(n)}}, \mathcal{V}_1)$ such that $||e_1 - e_{1, n}||_{\mathcal{G}_{r(n)}} \leq p^{-n}$ in $H^0 (\widetilde{U}, \mathcal{V}_1)$. 
\end{enumerate}
for each integer $n \geq 0$ by \cite[Lemma 2.1.5]{Pan20}. 

We define $\mathcal{O}_E^n (U)\{ x, e_1, t \} \subset \mathcal{O}_{\mathcal{S}h, E}^{\tau-\la} (\mathcal{S}h_{K^p} (U))$ to be the subset of elements $f$ which can be written as 
\[
f = \sum_{i, j, k \geq 0} c_{i, j, k}^{(n)} (x - x_n)^i \left( \log \left( \frac{e_1}{e_{1, n}} \right) \right)^j \left( \log \left( \frac{t_1}{t_{1, n}} \right) \right)^k 
\]
such that $c_{i, j, k}^{(n)} \in H^0 (U_{\mathcal{G}_{r(n)}}, \mathcal{O}_{\mathcal{S}h_{\mathcal{G}_{r(n)} K^p}, E})$ for integers $i, j, k \geq 0$ and $|| c_{i, j, k} || \leq p^{(n-1)(i+j+k)} C'$ holds for some uniform constant $C'$. 
It is a Banach algebra over $\mathbb{C}_p$ with norm 
\[
||f||_n \coloneqq \sup_{i, j, k \geq 0} ||c_{i, j, k}^{(n)} (f) p^{(n-1)(i+j+k)}||. 
\]

\begin{proposition}
\label{curve_taula}
For any positive integer $m > 0$, there exists a positive integer $n$ and a natural continuous embeddings of Banach spaces 
\[
\mathcal{O}_{\mathcal{S}h, E}^{\tau-\la} (\mathcal{S}h_{K^p} (U))^{\mathcal{G}_m -\an} \subset \mathcal{O}_E^n (U)\{ x, e_1, t \} \subset \mathcal{O}_{\mathcal{S}h, E}^{\tau-\la} (\mathcal{S}h_{K^p} (U))^{\mathcal{G}_{r(n)} -\an}. 
\]
In particular, there is a natural isomorphism of topological vector spaces 
\[
\varinjlim_n \mathcal{O}_{\mathcal{S}h, E}^{\tau-\la} (\mathcal{S}h_{K^p} (U))^{\mathcal{G}_n -\an} \cong \varinjlim_n \mathcal{O}_E^n (U)\{ x, e_1, t \}. 
\]
\end{proposition}

\begin{proof}
This follows from the argument of \cite[\S 4.3]{Pan20}. 
\end{proof}

Let $K^{\wp} \coloneqq K^p K_p^{\wp}$ for some compact open subgroup $K_p^{\wp} \subset G(\mathbb{Q}_p^{\wp})$. 
For an object $W$ in $\Rep_E (G^c)$, let $\mathcal{O}_{K^p, W} \coloneqq \pi_{\HT, \ast} \mathcal{O}_{\mathcal{S}h_{K^p}, W}$. 
In the case of unitary Shimura curves, there exists perfectoid unitary Shimura curves of level $K^{\wp}$. 

\begin{theorem}[{\cite[Theorem 3.3.3]{JL18} (and also e.g.\ \cite[Theorem 3.2.1]{QS25})}]
\label{perfd_curve_Kwp}
If $K^{\wp}$ is sufficiently small, there exists a unique perfectoid space $\mathcal{S}h_{K^{\wp}}$ over $\Spa(\mathbb{C}_p, \mathcal{O}_{\mathbb{C}_p})$ such that 
\[
\mathcal{S}h_{K^{\wp}} \sim \varprojlim_{K_{\wp}} \mathcal{S}h_{K^{\wp} K_{\wp}}, 
\]
where $K_{\wp}$ runs over all compact open subgroup of $G(\mathbb{Q}_{\wp})$. 
The space $\mathcal{S}h_{K^{\wp}}$ is a pro-\'etale $K_{\wp}$-torsor over $\mathcal{S}h_{K^{\wp} K_{\wp}}$. 
Moreover, there exists a Hodge-Tate period map 
\[
\pi_{\HT, \wp} \colon \mathcal{S}h_{K^{\wp}} \to \mathscr{F}\ell_{\mu}
\]
which is compatible with $\pi_{\HT}$ and satisfies the similar properties in Proposition \ref{HT_period_maps}. 
\end{theorem}

We define $\mathcal{O}_{K^{\wp}, W} \coloneqq \pi_{\HT, \ast} \mathcal{O}_{\mathcal{S}h_{K^{\wp}}}$. 
Note that the analog of Proposition \ref{curve_taula} in the case of $K^{\wp}$-level perfectoid unitary Shimura curves can be verified in the same way. 
For a subset $J \subset \{ z \} \cup \bigsqcup_{\wp_i | p} \Sigma_{\wp_i}$, we can define $J$-locally analytic subsheaves $\mathcal{O}_{K^p, W}^{J-\la} \subset \mathcal{O}_{K^p, W}$ and $\mathcal{O}_{K^{\wp}, W}^{J-\la} \subset \mathcal{O}_{K^{\wp}, W}$ as in Definition \ref{def_la_sheaves}. 

\begin{theorem}
\label{main_theorem}
There is natural $G(\mathbb{Q}_p)$-equivariant isomorphisms 
\[
(\widetilde{H}^1 (K^{\wp}, E) \widehat{\otimes}_{\mathbb{Q}_p} \mathbb{C}_p)^{\tau-\la} \cong H^1 (\mathscr{F}\ell_{\mu}, \mathcal{O}_{K^p, E}^{\tau-\la})^{K_p^{\wp}} \cong H^1 (\mathscr{F}\ell_{\mu}, \mathcal{O}_{K^{\wp}, E}^{\tau-\la}). 
\]
\end{theorem}

\begin{proof}
Initially, we show the first isomorphism. 
We claim that there is a natural $G(\mathbb{Q}_p)$-equivariant isomorphism 
\[
H_{\an}^i (\mathcal{S}h_{K^p}, \mathcal{O}_{\mathcal{S}h, E}^{\tau-\la}) \cong H^i (\mathscr{F}\ell_{\mu}, \mathcal{O}_{K^p, E}^{\tau-\la})
\]
for $i \geq 0$. 
It follows that there exists an affinoid open cover $\{ U_i \}_i$ of $\mathcal{S}h$ such that each $\pi_{K_p}^{-1} (U_i)$ is a locally analytic covering of $U_i$ in the sense of \cite[Definition 3.5.4]{Pan20} by the argument of \cite[Corollary 4.2.3]{Pan20} and the Kodaira-Spencer isomorphisms for unitary Shimura curves (cf.\ \cite[\S 4]{DT94}, \cite[\S 3.3.1.3]{Din17}). 
Then the claim follows by Proposition \ref{curve_taula} the first part of the proof of \cite[Theorem 4.4.6]{Pan20}. 

By Proposition \ref{Jla_comparison_H1}, we have a natural $G(\mathbb{Q}_p)$-equivariant isomorphism 
\[
(\widetilde{H}^1 (K^p, E) \widehat{\otimes}_{\mathbb{Q}_p} \mathbb{C}_p)^{\tau-\la} \cong H_{\an}^1 (\mathcal{S}h_{K^p}, \mathcal{O}_{\mathcal{S}h, E}^{\tau-\la}). 
\]
Then it suffices to construct natural $G(\mathbb{Q}_p)$-equivariant isomorphisms 
\begin{align*}
(\widetilde{H}^1 (K^p, E) \widehat{\otimes}_{\mathbb{Q}_p} \mathbb{C}_p)^{K_p^{\wp}, \tau-\la} &\cong (\widetilde{H}^1 (K^{\wp}, E) \widehat{\otimes}_{\mathbb{Q}_p} \mathbb{C}_p)^{\tau-\la}. 
\end{align*}
The first isomorphism follows from Corollary \ref{H1comp_taula_colimit} and that $\varinjlim_{K_p^{\wp}} (\widetilde{H}^1 (K^p K_p^{\wp}, E) \widehat{\otimes}_{\mathbb{Q}_p} \mathbb{C}_p)^{\tau-\la}$ is a smooth $\GL_2 (F_{\wp})$-representation by the degeneration of spectral sequence. 

Next we show the second isomorphism. 
Take an open cover $\bigcup_i V_i = \mathscr{F}\ell_{\mu}$ in analytic topology, where Then each section $\mathcal{O}_{K^p, E}^{\tau-\la} (V_i)$ is a smooth $K_p^{\wp}$-representation by Proposition \ref{curve_taula}. 
So there is a spectral sequence 
\[
E_2^{p, q} = H^p (K_p^{\wp}, H^q (\mathscr{F}\ell_{\mu}, \mathcal{O}_{K^p, E}^{\tau-\la})) \Longrightarrow H^{p+q} (\mathscr{F}\ell_{\mu}, \mathcal{O}_{K^{\wp}, E}^{\tau-\la})
\]
degenerating at $E_2^{0, 1}$-term and the second isomorphism follows from this. 
\end{proof}







\end{document}